\documentclass[11pt]{article}
\usepackage{bbm}
 \usepackage{amssymb}
\usepackage{amssymb, amsthm, amsmath, amscd}
\usepackage[usenames,dvipsnames]{color}

\newtheorem{thm}{Theorem}[section]
\newtheorem{lem}[thm]{Lemma}
\newtheorem{prop}[thm]{Proposition}
\newtheorem{cor}[thm]{Corollary}
\theoremstyle{definition}

\theoremstyle{definition}
\newtheorem{df}[thm]{Definition}
\theoremstyle{definition}
\newtheorem{rem}[thm]{Remark}

\theoremstyle{definition}

\newcommand{\red}{\textcolor{red}}

\renewcommand{\phi}{\varphi}

\newcommand{\N}{\mathbb{N}}
\newcommand{\Z}{\mathbb{Z}}
\newcommand{\Q}{\mathbb{Q}}
\newcommand{\R}{\mathbb{R}}
\newcommand{\C}{\mathbb{C}}
\newcommand{\T}{\mathbb{T}}

\numberwithin{equation}{section}

\newcommand{\Aff}{\operatorname{Aff}}

\newcommand{\id}{\operatorname{id}}

\newcommand{\aff}{\rm aff}

\newcommand{\morp}{contractive completely positive linear map}
\newcommand{\cp}{completely positive linear map}

\newcommand{\hm}{homomorphism}
\newcommand{\dt}{\delta}
\newcommand{\ep}{\varepsilon}

\newcommand{\td}{\tilde}

\newcommand{\DT}{\Delta}

\newcommand{\andeqn}{\,\,\,{\rm and}\,\,\,}
\newcommand{\rforal}{\,\,\,{\rm for\,\,\,all}\,\,\,}
\newcommand{\CA}{$C^*$-algebra}
\newcommand{\SCA}{$C^*$-subalgebra}

\newcommand{\af}{{\alpha}}
\newcommand{\bt}{{\beta}}

\newcommand{\diag}{{\rm diag}}

\newcommand{\wilog}{without loss of generality}
\newcommand{\Wlog}{Without loss of generality}

\newcommand{\D}{\mathbb D}
\newcommand{\beq}{\begin{eqnarray}}
\newcommand{\eneq}{\end{eqnarray}}
\newcommand{\tforal}{\,\,\,\text{for\,\,\,all}\,\,\,}
\newcommand{\tand}{\,\,\,\text{and}\,\,\,}

\usepackage{amsfonts}
\usepackage{mathrsfs}
\usepackage{textcomp}
\usepackage[all]{xy}

\title{\hspace{-0.2in}A classification of finite simple amenable ${\cal Z}$-stable C*-algebras, II. C*-algebras with rational generalized tracial rank one.}
\author{Guihua Gong,  Huaxin Lin and Zhuang Niu
 }
\date{
}

\begin{document}

\maketitle

\begin{abstract}



A classification theorem is obtained for a class of unital simple separable amenable ${\cal Z}$-stable  $C^*$-algebras 
which exhausts all possible values of the Elliott invariant for unital stably finite simple separable amenable ${\cal Z}$-stable \CA s.  Moreover,  it contains all unital simple separable  amenable \CA s which satisfy the UCT and have finite  rational tracial rank.

\end{abstract}

\setcounter{section}{21}
\tableofcontents

\vspace{0.3in}

This is the second part of the paper entitled
``A classification of finite simple amenable ${\cal Z}$-stable C*-algebras" (see \cite{GLN-I}).


The main theorem of this part is the following isomorphism theorem:
\vspace{0.1in}

{\bf Theorem} (see Theorem \ref{MFTh}).
Let $A$ and $B$ be two unital separable simple amenable ${\cal Z}$-stable \CA s which satisfy the UCT.
Suppose that $gTR(A\otimes Q)\le 1$ and $gTR(B\otimes Q)\le 1.$ Then
$A\cong B$ if and only if
\begin{equation*}
{\rm Ell}(A)\cong {\rm Ell}(B).
\end{equation*}

See Section 29 for a brief explanation.
We also refer to the first part  \cite{GLN-I}, in particular, Section 2 of \cite{GLN-I},  for the  notations and definitions.
\vspace{0.1in}


{\bf Acknowledgements}:
A large part of this article {{was}} written during the summers  of 2012, 2013, and 2014
when all three authors visited The Research Center for Operator Algebras in East China Normal University
 which is in part supported  by NNSF of China (11531003)  and Shanghai Science and Technology
 Commission (13dz2260400)
 and  Shanghai Key Laboratory of PMMP.
They were
partially
supported by the {{Center}} (also during summer 2017 when
some of revision were made).  Both the first named author and  the second named author were supported partially by NSF grants {(DMS 1665183 and DMS 1954600).}
 The work of the third named author was partially supported by a NSERC Discovery Grant, a Start-Up Grant from the University of Wyoming, a Simons Foundation Collaboration Grant, and a NSF grant (DMS-1800882).
{{The authors would like to take this opportunity to express their most sincere gratitude to  the referees 
and the editor for their efforts to improve the paper.}}

\section{{{Construction of maps}}}
{In this section, we will introduce some technical results on the existence of certain maps.}

{{Recall that ${\cal C}$ is the class of \CA s which are 1-dimensional NCCWs (see 3.1 of \cite{GLN-I}).
Let $A$ be a unital
simple \CA. We say $A\in {\cal B}_1$\index{${\cal B}_1$} if the following property holds:
Let $\ep>0,$ let $a\in A_+\setminus \{0\},$ and let ${\cal F}\subset A$ be a
finite subset.
There exist a non-zero projection $p\in A$ and a \SCA\, $C\in {\cal C}$
with $1_C=p$ such that
\begin{eqnarray}\nonumber
&&\|xp-px\|<\ep\tforal x\in {\cal F},\\
&&{\rm dist}(pxp, C)<\ep\tforal x\in {\cal F},\andeqn \\
&&1-p\lesssim a.
\end{eqnarray}}}
{{If $C$ as above can always be chosen in $\mathcal C_0$, that is, with $K_{{1}}(C)=\{0\},$ then we say that $A\in {\cal B}_0.$
\index{${\cal B}_0.$}}}

\vspace{0.2in}

Recall that we  refer to the first part  \cite{GLN-I}, in particular, Section 2 of \cite{GLN-I},  for the  notations and definitions.

\begin{lem}\label{preBot1}
Let $X$ be a finite CW complex, let $C=PM_k(C(X))P,$ and let $A_1\in {\cal B}_0$ be a unital simple \CA. Assume that $A=A_1\otimes U$ for
a UHF-algebra $U$
of infinite type.
Let $\af\in KK_e(C, A)^{++}$ (see Definition 2.10 of \cite{GLN-I}).
Then there exists a unital  monomorphism $\phi: C\to A$ such that $[\phi]=\af.$ Moreover
we may write $\phi=\phi_n'\oplus \phi_n'',$
where $\phi_n': C\to (1-p_n)A(1-p_n)$ is a unital monomorphism,  $\phi_n'': C\to p_nAp_n$ is a unital homomorphism with $[\phi_n'']=[\Phi]$ in $KK(C, p_nAp_n)$ for some
\hm\, $\Phi$ with finite dimensional range,
and
$$\lim_{n\to\infty} \max\{\tau(1-p_n):\tau\in T(A)\}=0\rforal \tau\in T(A),$$
where $p_n\in A$ is a sequence of projections.
\end{lem}

\begin{proof}
To simplify the matter, we may assume that $X$ is connected.
Suppose that the lemma holds for the case
$C=M_k(C(X))$ for some integer $k\ge 1.$
Consider the case
$C=PM_k(C(X))P.$ Note $C\otimes {\cal K}\cong C(X)\otimes {\cal K}.$
Let $q\in M_m(A)$  be a projection (for some integer $m\ge 1$) such that $[q]=\af([1_{M_k(C(X))}]).$ Put $A_2=qM_m(A)q.$
Then $\af\in KK_e(M_k(C(X)), A_2)^{++}.$
Let $\psi: M_k(C(X))\to qM_m(A)q$ be the map given by the lemma for the case that $C=M_k(C(X)).$
Note now $C=PM_k(C(X))P.$
Let $\psi'=\psi|_C.$ Since $P\le 1_{M_k(C(X))},$ $\psi'(1_C)=\psi'(P)\le q.$ Moreover,
$[\psi'(1_C)]=[1_A].$ Since $A=A_1\otimes U,$ there is a unitary $v\in A_2$ such
that  $v^*\psi'(1_C)v=1_A.$ Define $\phi={\rm Ad} v\,\circ \psi'.$ We see that the general case reduces to the
case  $C=M_k(C(X)).$ This case then reduces to the case $C=C(X).$

Since quasitraces of $C$ and $A$ are traces (see  9.10 of
\cite{GLN-I})
  by Corollary 3.4 of \cite{Blatrace},
$\alpha ({\rm ker}\rho_C) \subset {\rm ker}\rho_A$.

Since $K_i(C)$ is finitely generated, $i=0,1,$ $KK(C,A)=KL(C,A).$
Let $\af\in KL_e(C,A)^{++}.$  We
may identify $\af$  with an element in ${\rm Hom}_{\Lambda}(\underline{K}(C),
\underline{K}(A))$ by a result
in (\cite{DL}).

{ {Write  $A=\lim_{n\to\infty}(A_1\otimes M_{r_n}, \imath_{n,n+1}),$  where $r_n|r_{n+1},$
$r_{n+1}=m_nr_n$ and
$\imath_{n,n+1}(a)=a\otimes 1_{M_{m_n}},$ $n=1,2,....$ Since $K_*(C)$ is finitely generated and consequently, $\underline{K}(C)$ is finitely generated modulo Bockstein $\Lambda$ operations, there is an element  $\alpha_1\in KK(C, A_1\otimes M_{r_n})$ such that $\alpha=\alpha_1\times [  \imath_n],$ where $ [  \imath_n] \in KK(A_1\otimes M_{r_n}, A)$ is induced by the inclusion $ \imath_n: A_1\otimes M_{r_n} \to A$.
Increasing $n$, we may assume that $\alpha_1 ({\rm ker}\rho_C) \subset {\rm ker}\rho_{A_1\otimes M_{r_n}}$ and further that $\alpha_1\in KK_e(C, A_1\otimes M_{r_n})^{++}$. Replacing  $A_1$ by $A_1\otimes M_{r_n}$,
we may assume that $\alpha=\alpha_1 \times [\imath]$, where $\alpha_1\in KK_e(C, A_1)^{++}$ and $ \imath: A_1 \to A$ is the inclusion.}}

It induces an element ${{\tilde \af}}_1\in KL(C\otimes U, A\otimes U).$
Let $K_0(U)=\D,$ a dense subgroup of $\Q$.
Note that $K_i(C\otimes U)=K_i(C)\otimes \D,$ $i=0,1,$ by the K\"{u}nneth formula.

We verify that ${{\tilde \af}}_1(K_0(C\otimes U)_+\setminus \{0\})\subset K_0(A\otimes U)_+\setminus\{0\}.$
Consider $x=\sum_{i=1}^m x_i\otimes d_i\in K_0(C\otimes U)_+\setminus \{0\}$ with $x_i\in K_0(C)$ and $d_i\in \D,$ $i=1,2,...,m.$   There  is a projection $p\in M_r(C)$ for some $r\ge 1$ such that $[p]=x.$
 Let $t\in T(C);$ then
\beq\label{prebot1-n1}
\sum_{i=1}^m t(x_i)d_i>0.
\eneq
It should be noted that, since $C=C(X)$ and $X$ is connected, $t(x_i)\in \Z$ and $t(x_i)=t'(x_i)$ for all $t, t'\in T({{C}}).$
Since $\af_{ {1}}([1_C])=[1_{A_{{1}}}],$ $\tau\circ \af_{{1}}(x_i)=t(x_i)$ for any $\tau\in T(A_{{1}})$ and $t\in T(C).$
By (\ref{prebot1-n1}),
\beq\label{prebot1-n2}
\tau({{\tilde \af}}_1(x))=\sum_{i=1}^m \tau\circ\af_{{1}}(x_i)d_i=\sum_{i=1}^m t(x_i)d_i>0
\eneq
for all $\tau\in T(A_{{1}}).$  This shows that ${{\tilde \af}}_1$ is strictly positive.
For any \CA\, $A',$ in this proof, we will use $j_{A'}: A'\to A'\otimes U$ for the \hm\, $j_{A'}(a)=a\otimes 1_U$
for all $a\in A'.$
Evidently,
\beq\label{prebot1-n3}
\alpha= {\tilde \af}_1\circ j_C=j_{A_1}\circ \af_1.
\eneq

Let $\mathfrak{d}:=p_1^{n_1}p_2^{n_2}\cdots $ be the supernatural number associated with $U$ (and $\D$),
where each $p_i$ is a distinct prime number.
If there are infinitely many of them, we may also write $\mathfrak{d}:=\prod_{i=1}^\infty l_i,$
where each $l_i$ is an integer. We define $m_1:=l_1p_{k_1}$ so that the prime number 
is not a factor of $l_1,$ and $m_i:=l_i p_{k_i}$ so that $p_{k_i}$ is not a factor of $l_i$ and $m_1m_2\cdots m_{i-1}.$ 
Since $p_{k_i}\to \infty$ as $i\to\infty,$ $\lim_{i\to \infty} \frac{m_i}{l_i}=\infty.$ Moreover
$\prod_{i=1}^\infty m_i=\mathfrak{d}.$
If there are only finitely many distinct $p_i$'s, write $\mathfrak{d}=p_1^{n_1}p_2^{n_2}\cdots p_{f}^{n_f}p_{k_1}^\infty \cdots p_{k_I}^{\infty},$ where $n_i<\infty$ ($1\le i\le f$). Let $l_1:=p_1^{n_1}p_2^{n_2}\cdots p_{f}^{n_f},$
$m_1:=l_1p_{k_1}\cdots p_{k_I},$ $l_i:=p_{k_1}p_{k_2}\cdots p_{k_I}$ and
$m_i:=p_{k_1}^ip_{k_2}^i \cdots p_{k_l}^i$ for $i\ge 2.$ Then $\mathfrak{d}=\prod_{i=1}^\infty l_i=\prod_{i=1}^\infty m_i$
and $\lim_{i\to \infty} \frac{m_i}{l_i}=\infty.$

Write  $U=\lim_{n\to\infty}(M_{r_n}, \imath_n),$  where $r_1=1,$  and $r_n=\prod_{i=1}^{n-1}m_i$ ($n>1$),
$r_{n+1}=m_nr_n$ and
$\imath_n(a)=a\otimes 1_{M_{m_n}}$ for $a\in M_{r_n},$ $n=1,2,....$   We may assume that $r_1=1.$
Let $r_1'=1,$ ${{r_n'}}=\prod_{i=1}^{n-1}l_i$ ($n>1$). Let
$U_1:=\lim_{n\to\infty}(M_{r_n'}, \imath_n'),$ where
$\imath_n'(b)=b\otimes M_{l_n}$ for $b\in M_{r_n'}.$

Let 
$\eta_n: M_{r_i'}\to M_{r_i}$ by $\eta_i(a)=a\otimes 1_{m_i/l_i}$ for $a\in M_{r_i'},$ respectively.
%
%
Then $\{\eta_n\}$ induces  a unital \hm\,  $\eta: U_1\to U$
which induces an isomorphism from $K_0(U_1)$ onto $K_0(U).$

%
%
Recall that we assume that $X$ is connected. Fix a base point
$x_0\in X.$   Let $C_0:=C_0(X\setminus \{x_0\}).$  Then $C$ is $KK$-equivalent
to $\C\oplus C_0$ and
$\underline{K}(C)=\underline{K}(\C)\oplus \underline{K}(C_0).$
 Let {{$\{x_n\}$}} be a sequence of points in {{$X\setminus \{x_0\}$}} such that {{$\{x_k,x_{k+1},...,x_n,...\}$}} is dense in $X$ for each $k$
and  each point in {{$\{x_n\}$}} repeated infinitely many times.
Let ${ {B}}=\lim_{n\to\infty} (C_n:=M_{r_n}(C), \psi_n),$ where
$$
\psi_n(f)={\rm diag}({ {\underbrace{f,f..., f}_{l_n}}}, {{f(x_1),f(x_2),...,f(x_{m_n-{{l_n}}})}})\tforal f\in M_{r_n}(C),
$$
$n=1,2,....$  Note that $\psi_n$ is injective.
Set $e_n={\rm diag}(1_{M_{r_n{{\cdot l_n}}}},0,...,0)\in M_{r_{n+1}}({{C}}),$ $n=1,2,...$

It is standard that ${{B}}$ has tracial rank zero (see \cite{Goodearl-AH} and also, 3.77 and 3.79 of \cite{Lnbok}).
Moreover, $\underline{K}(B)=\underline{K}(U)\oplus \underline{K}(C_0\otimes U_1)$ and
$K_0(B)_+=\{(d, z): d\in \D_+, z\in K_0(C_0\otimes U_1)\}\cup\{(0,0)\}.$
Note that
 $B$ is a unital  simple AH-algebra with no dimension growth, with real rank zero, and with a unique tracial state.
 Also $h:=\psi_{1, \infty}: C\to B$ gives $[h]|_{\underline{K}(\C)}(z)=z\otimes 1_{\D}$
 for $z\in \underline{K}(\C)$ and $[h]|_{\underline{K}(C_0)}(x)=x\otimes [1_{U_1}]$
 for $x\in \underline{K}(C_0).$ Let $\eta_0: C_0\otimes U_1\to C_0\otimes U$  {{be}}
 defined by $\eta_0:={\rm id}_{C_0}\otimes \eta.$
 Define $\kappa\in KL(B, B)$ by $\kappa|_{\underline{K}(\C\otimes U)}=\id_{\underline{K}(\C\otimes U)}$ and
 $\kappa|_{\underline{K}(C_0\otimes U_1)}=\eta_0.$  Note that $\kappa\in KL(B,B)^{++}.$
  Recall $\underline{K}(C\otimes U)=\underline{K}(\C\otimes U)\oplus \underline{K}(C_0\otimes U).$
  One may also view $\kappa$ as an element
  in ${\rm Hom}_{\Lambda}(\underline{K}(\C\otimes U)\oplus  \underline{K}(C_0\otimes U_1), \underline{K}(C\otimes U))
  =KL(B, C\otimes U).$ It has an inverse
  $\kappa^{-1}\in KL(C\otimes U, B).$  We have $\kappa^{-1}\circ [j_C]=[h].$

 %
%
%
Note  that $1-\psi_{n, \infty}(e_n)$ commutes with the image of $h$  for all $n\ge N.$ Moreover,
$(1-\psi_{n, \infty}(e_n))h(c)(1-\psi_{n, \infty}(e_n))=\psi_{n, \infty}((1-e_n)\psi_{N, n}\circ
 \phi_{1, N}(c)(1-e_n))$ for all $c\in C.$
Therefore the map $(1-\psi_{n, \infty}(e_n))h(C)(1-\psi_{n, \infty}(e_n))$ has finite dimensional range.

We also have ${{{\tilde\af}_1\circ \kappa}}\in KL_e(B, {{A}})^{++}, $ where (recall) $A=A_1\otimes U.$
We also note that $B$ has a unique tracial state. Let $\gamma: T(A)\to T(B)$ be defined  by
$\gamma(\tau)=t_0$ where $t_0\in T(B)$ is the unique tracial state. It follows that
${{{\tilde\af}_1\circ \kappa}}$ and $\gamma$ are compatible.  By
Corollary 21.11 of \cite{GLN-I}, 
there is a unital \hm\, $H: B\to {{A}}$ such that
$[H]={{{\tilde\af}_1\circ \kappa.}}$
Define $\phi: C\to A$ by $\phi= H\circ h.$
Then,  $\phi$ is injective,  and, by (\ref{prebot1-n3}) and {{${{[h]}}={\kappa}^{-1}\circ {{[j_C]}}$ , we have}}
$[\phi]=\af.$

To show the last part,  define $q_n=\psi_{n+1, \infty}{{(e_n)}}\in B,$  $n=N+1, N+2,....$
Define $p_n={ {1- H(q_n)}},$ $n=N+1, N+2,....$ One checks that
\beq\label{prebot1-n6}
\lim_{n\to\infty}\max\{\tau(1-p_n):\tau\in T(A)\}=\lim_{n\to\infty} \frac{l_n}{m_n}=0.
\eneq
Note that for $n>N$, $q_n$ commutes with the image of $h$ and the homomorphism $(1-q_n)h(1-q_n): C\to (1-q_n)B(1-q_n)$  has finite dimensional range.
Define $\phi_n': C\to {{ (1-p_n)A(1-p_n)}}$ by $\phi_n'(f)= H(q_n)H\circ {{ h(f)H(q_n)}}$ for $f\in C.$
 Define
{ {$\phi_n''(f)=(1- p_n)H\circ  h(f)(1-p_n),$  which is }}a point-evaluation map. The lemma follows.

\end{proof}




We also have the following:
\begin{lem}\label{Circlef-1}
Let $C=M_k(C(\T))$ and let $A$ be a unital  infinite dimensional simple \CA\, with stable rank one and with the property (SP).  Then the conclusion of \ref{preBot1}
also holds for a given $\af\in KK_e(C,A)^{++}.$
\end{lem}

\begin{proof}
Let $p_1\in C$ be a minimal rank one projection. Since $k\af([p_1])=\af([1_C])=[1_A],$
$A$ contains mutually equivalent and mutually orthogonal projections
$e_1, e_2,...,e_k$ such that $\sum_{i=1}^k e_i=1_A.$
Thus $A=M_k(A'),$ where $A'\cong e_1Ae_1.$
Since $e_iAe_i$ are unital infinite dimensional simple \CA s with stable rank one and with
(SP), the general case can be reduced to the case that $k=1.$
Fix $1>\dt>0.$ Choose a non-zero projection $p\in A$ such that $\tau(p)<\dt$ for all $\tau\in T(A).$
Note $K_1(pAp)=K_1(A),$ since $A$ is simple.  Let $\af_1: K_1(C(\T))\to K_1(pAp)$ be the \hm\, given by $\af.$
Let $z\in C(\T)$ be the standard unitary generator. Let $x=\af_1([z])\in K_1(pAp).$ Since
$pAp$ has stable rank one, there is a unitary $u\in pAp$ such that $[u]=x$ in $K_1(pAp)=K_1(A).$
Define $\phi': C(\T)\to pAp$ by $\phi'(f)=f(u)$ for all $f\in C(\T).$ Define $\phi'': C(\T)\to (1-p)A(1-p)$ by
$\phi''(f)=f(1)(1-p)$ for all $f\in C(\T)$ (where $f(1)$ is the point evaluation at $1$ on the unit circle).
Define $\phi=\phi'\oplus \phi'': C(\T)\to A.$  The map $\phi$ verifies the conclusion of lemma follows.

\end{proof}



\begin{cor}\label{istBotC}
Let  $X$ be a  connected  finite CW complex, let $C=PM_m(C(X))P,$
where $P\in M_m(C(X))$ is a projection, let $A_1\in {\cal B}_0$ be a unital separable simple \CA\, which satisfies
the UCT, and let $A=A_1\otimes U,$ where $U$ is a UHF-algebra of infinite type.
Suppose that $\af\in KK_e(C, A)^{++}$ and $\gamma: T(A)\to T_f(C(X))$ is a continuous affine map.
Then
there exists a
sequence of \morp  s $h_n: C\to A$ such that
\begin{enumerate}
\item $\lim_{n\to\infty}\|h_n(ab)-h_n(a)h_n(b)\|=0,$ for any $a,b\in C$,
\item for each $h_n$, the map $[h_n]$ is well defined and $[h_n]=\alpha$, and
\item $\lim_{n\to\infty}\max\{|\tau\circ h_n(f)-\gamma(\tau)(f)|: \tau\in T(A)\}=0$  for any $f\in C$.
\end{enumerate}
\end{cor}

\begin{proof}
By Theorem
21.10 of \cite{GLN-I},
one may assume that $A$ is a unital \CA\, as described in
Theorem 14.10 of \cite{GLN-I}.
It follows from Lemma \ref{preBot1} that there is a unital \hm\,  $h_n: C\to A$ such that
$[h_n]=\alpha$. Moreover, $$h_n=h_n'\oplus h_n'',$$ where $h''_n: C\to p_nAp_n$ is a homomorphism with $[h''_n]=[\Phi^{{'}}]$ in $KK(C, p_nAp_n)$ for some point evaluation map $\Phi^{{'}}$, where $p_n$ is a projection in $A$ with $\tau(1-p_n)$ converging to $0$ uniformly as $n\to\infty$. We will modify the map $ h_n=h_n'\oplus h_n''$ to get the  \hm.

We assert that for any finite subset $\mathcal H\subset C_{s.a}$, and $\epsilon>0$, and any sufficiently large $n$, there is a unital homomorphism $\tilde{h}_n: C\to p_nAp_n$ such that  $[\tilde{h}_n]=[\Phi]$ in $KK(C,p_n Ap_n)$ for 
a homomorphism $\Phi$ with finite dimensional range, and
$$|\tau\circ \tilde{h}_n(f)-\gamma(\tau)(f)|<\epsilon\rforal \tau\in T(A)$$ for all $f\in\mathcal H$. The corollary then  follows by replacing the map $h_n''$ by the map $\tilde{h}_n$---of course, we use the fact that $\lim_{n\to \infty}\tau(1-p_n)=0$.

Let $\mathcal H_{1, 1}$ (in place of $\mathcal H_{1, 1}$) be the finite subset of Lemma 17.1 of \cite{GLN-I} 
with respect to $\mathcal H$ (in place of $\mathcal H$), $\epsilon/{8}$ (in place of $\sigma$), and $C$
 (in place of $C$). Since $\gamma(T(A))\subset T_f(C(X))$, there is $\sigma_{1, 1}>0$ such that $$\gamma(\tau)(h)>\sigma_{1, 1}\rforal h\in\mathcal H_{1, 1}\rforal \tau\in T(A).$$

Let $\mathcal H_{1, 2}\subset C^+$ (in  place of $\mathcal H_{1, 2}$) be the finite subset of Lemma 17.1 of \cite{GLN-I} 
with respect to $\sigma_{1, 1}$. Since $\gamma(T(A))\subset T_f(C(X))$, there is $\sigma_{1, 2}>0$ such that $$\gamma(\tau)(h)>\sigma_{1, 2}\rforal h\in\mathcal H_{1, 2} \rforal \tau\in T(A).$$

Let $M$ be the constant of Lemma 17.1 of \cite{GLN-I} 
with respect to $\sigma_{1, 2}$. 
By Lemma 16.12 of \cite{GLN-I} 
(also see the proof of Lemma 16.12 of \cite{GLN-I})
for sufficiently large $n$, there are a C*-subalgebra $D\subset p_nAp_n\subset A$ such that $D\in \mathcal C_0$, and a continuous affine map $\gamma': T(D)\to T(C)$ such that
$$|\gamma'(\frac{1}{\tau(p)}\tau|_D)(f)-\gamma(\tau)(f)|<\epsilon/4 \rforal \tau\in T(A)\rforal f\in\mathcal H,$$
where $p=1_D$, $\tau(1-p)<\epsilon/(4+\epsilon)$, and further (see part (2) of Lemma 16.12 of \cite{GLN-I})
\beq
&&\gamma'(\tau)(h)>\sigma_{1, 1}\rforal \tau\in T(D) \rforal h\in\mathcal H_{1, 1},\andeqn\\
&&\gamma'(\tau)(h)>\sigma_{1, 2}\rforal \tau\in T(D)\rforal h\in\mathcal H_{1, 2}.
\eneq
Since $A$ is simple and not elementary, one may assume that the dimensions of the irreducible representations of $D$ are at least $M$. Thus, by Lemma 17.1 of \cite{GLN-I}, 
there is a homomorphism $\phi: C\to D$ such that
$[\phi]=[\Phi]$ in $KK(C, D)$ for a point evaluation map $\Phi$, and that
 $$|\tau\circ\phi(f)-\gamma'(\tau)(f)|<\epsilon/{{4}}\rforal f\in \mathcal H\rforal \tau\in T(D).$$

Pick a point $x\in X$, and define ${\tilde h}: C\to p_nAp_n$ by
$$f\mapsto f(x)(p_n-p)\oplus\phi(f)\rforal f\in C.$$ Then a calculation as in the proof of Theorem 17.3 of \cite{GLN-I} 
shows that the homomorphism $h_n'\oplus {\tilde h}$ verifies
the assertion.
\end{proof}




\begin{cor}\label{istBotC+}
Let   $C\in {\bf H}$
(see Definition 14.5 of \cite{GLN-I})
and let $A_1\in {\cal B}_0$ be a unital separable simple \CA\, which satisfies
the UCT and let $A=A_1\otimes U$ { for some UHF-algebra $U$ of infinite type.}
Suppose that $\af\in KK_e(C, A)^{++},$ $\lambda: U(C)/CU(C)\to U(A)/CU(A)$ is
a continuous \hm, and $\gamma: T(A)\to T_f({{C}})$ is a continuous affine map such that
$\af, \lambda,$ and $\gamma$ are compatible.
Then there exists a
sequence of  unital completely positive linear maps $h_n: C\to A$ such that
\begin{enumerate}
\item $\lim_{n\to\infty}\|h_n(ab)-h_n(a)h_n(b)\|=0$ for any $a,b\in C$,
\item for each $h_n$, the map $[h_n]$ is well defined and $[h_n]=\alpha$,
\item $\lim_{n\to\infty}\max\{|\tau\circ h_n(f)-\gamma(\tau)(f) |: \tau\in T(A)\}=0$ for all $f\in C,$  and
\item $\lim_{n\to\infty}{\rm dist}(h_n^{\ddag}({\bar u}), \lambda({\bar u}))=0$
for any $u\in U(C).$
\end{enumerate}
\end{cor}

\begin{proof}
Let $\epsilon>0$. Let $\mathcal U$ be a finite subset of $U(C)$ such that $\overline{U}$ generates $J_c(K_1(C))$, where $J_c(K_1(C))$ is as in Definition 2.16 of \cite{GLN-I}.
Let $\sigma>0$, $\delta>0$ and $\mathcal G$ be the constant and finite subset of Lemma 21.5 of \cite{GLN-I} 
with respect to $\mathcal U$, $\epsilon,$ and $\lambda$ (in the place of $\alpha$). Without loss of generality, one may assume that $\delta<\epsilon$.

Let $\mathcal F$ be a finite subset such that $\mathcal F\supset \mathcal G$. Let $\mathcal H\subset C$ be a finite subset of self-adjoint elements with norm at most one. By Corollary \ref{istBotC}, there is a completely positive  linear map $h': C\to A$ such that $h'$ is $\mathcal F$-$\delta$-multiplicative, $[h']$ is well defined and $[h']=\alpha$, and
\begin{equation}\label{tr-eqn-09-27}
|\tau(h'(f))-\gamma(\tau)(f)|<\epsilon,\quad\tau\in\mathrm{T}(A), \ f\in\mathcal H.
\end{equation}

By  Theorem 21.9 of \cite{GLN-I}, 
the C*-algebra $A$ is isomorphic to one of the model algebras constructed in
Theorem 14.10 of \cite{GLN-I}, 
and therefore there is an inductive limit decomposition $A=\varinjlim(A_i, \phi_i)$, where $A_i$ and $\phi_i$ are as described in
Theorem 14.10 of \cite{GLN-I}.
Without loss of generality, one may assume that $h'(C)\subset A_i$ for some $i.$ Therefore, by
Theorem 14.10 of \cite{GLN-I},
the map $\phi_{i, \infty}\circ h'$ has a decomposition $$\phi_{i, \infty}\circ h'=\psi_0\oplus\psi_1$$ such that $\psi_0, \psi_1$ satisfy (1)--(4) of Lemma 21.5 of \cite{GLN-I} 
with the $\sigma$ and $\delta$ above.

It then follows from Lemma 21.5 of \cite{GLN-I} 
that there is a homomorphism $\Phi: C\to e_0Ae_0$, where $e_0=\psi_0(1_C)$, such that

(i) $\Phi$ is homotopic to a homomorphism with finite dimensional range and
\begin{equation}\label{same-k0}
[\Phi]_{*0}=[\psi_0],\andeqn
\end{equation}

(ii) for each $w\in \mathcal U$, there is $g_w\in \text{U}_0(B)$ with $\mbox{cel}(g_w)<\epsilon$ such that
        \begin{equation}\label{k1-eqn-09-27-01}
         \lambda(\bar{w})^{-1}(\Phi\oplus \psi_1)^\ddagger(\bar{w})=\bar{g}_w.
         \end{equation}

Consider the map $h:=\Phi\oplus\psi_1.$ Then $h$ is $\mathcal F$-$\epsilon$-multiplicative. By \eqref{same-k0}, one has
$$[h]=[\psi_0]\oplus[\psi_1]=[h']=\alpha.$$
By \eqref{tr-eqn-09-27} and Condition (4) of Lemma 21.5 of \cite{GLN-I}, 
one has, for all $f\in {\cal H},$
$$|\tau(h(f))-\gamma(\tau)(f)|\leq |\tau(h'(f))-\gamma(\tau)(f)|+\delta <\epsilon+\delta<2\epsilon.
$$
It follows from \eqref{k1-eqn-09-27-01} that, for all $u\in {\cal U},$
$$\mathrm{dist}(\overline{h(u)}, \lambda(\overline{u}))<\epsilon.$$
Since $\mathcal F$, $\mathcal H$, and $\epsilon$ are arbitrary, this proves  the corollary.
\end{proof}

\begin{cor}\label{istBotC++}
Let $C\in {\bf H}$
and let $A_1\in {\cal B}_0$ be a unital separable simple \CA\, which satisfies
the UCT, and let $A=A_1\otimes U$ for some UHF-algebra $U$ of infinite type.
Suppose that $\af\in KL_e(C, A)^{++}$ and $\lambda: U(C)/CU(C)\to U(A)/CU(A)$ is
a continuous \hm, and $\gamma: T(A)\to T_f({{C}})$ is a continuous affine map such that
$\af, \lambda, $ and $\gamma$ are compatible.
Then
there exists a  unital \hm\, $h: C\to A$
 such that
\begin{enumerate}
 \item $[h]=\af, $
\item $\tau\circ h(f)=\gamma(\tau)(f) $ for any $f\in C,$ and
\item $h_n^{\ddag}=\lambda.$
\end{enumerate}
\end{cor}

\begin{proof}
Let us construct a sequence of unital completely positive linear maps
$h_n: C\to A$ which satisfies (1)--(4) of Corollary \ref{istBotC+}, and moreover,
is such that the sequence $\{h_n(f)\}$ is Cauchy for any $f\in C$. Then the limit map $h=\lim_{n\to\infty} h_n$ is the desired homomorphism.


Let $\{{\cal F}_n\}$ be an increasing sequence in the unit ball of $C$ with
union
dense in  the unit ball of $C.$
Define $\Delta(a)=\min\{\gamma(\tau)(a): \tau\in\mathrm{T}(A)\}$. Since $\gamma$ is continuous and $\mathrm{T}(A)$ is compact, the map $\Delta$ is an order preserving map from $C_+^{1, q}\setminus\{0\}$ to $(0, 1)$.
Let $\mathcal G(n), \mathcal H_1(n), \mathcal H_2(n) \subset C$,
$\mathcal U(n)\subset U_\infty(C)$, $\mathcal P(n)\subset\underline{K}(C)$, $\gamma_1(n)$, $\gamma_2(n)$,
and $\delta(n)$ be the finite subsets and constants of  Theorem 12.7 of \cite{GLN-I} 
with respect to ${\cal F}_n$, $1/2^{n+1},$ and $\Delta/2$.
We may assume that $\delta(n)$ decreases to $0$ if $n\to\infty,$  $\mathcal P(n)\subset \mathcal P(n+1),$ $n=1,2,...,$ and
$\bigcup_{n=1}^{\infty} {\cal P}(n)=
\underline{K}(C).$

Let $\mathcal G_1\subset\mathcal G_2\subset\cdots$ be an increasing sequence of finite subsets of $C$ such that $\bigcup\mathcal  G_n$ is dense in $C$, and let $\mathcal U_1\subset\mathcal U_2\subset\cdots$ be an increasing sequence of finite subsets of $U(C)$ such that $\bigcup\mathcal  U_n$ is dense in $U(C)$. One may assume that $\mathcal G_{n}\supset \mathcal G(n)\cup \mathcal G(n-1)$, $\mathcal G_{n}\supset{\cal H}_1(n)\cup {\cal H}_1(n+1)\cup \mathcal H_2(n)\cup \mathcal H_2(n-1) $, and $\mathcal U_n\supset \mathcal U(n)\cup \mathcal U(n-1)$.

By Corollary  \ref{istBotC+}, there is a $\mathcal G_1$-$\delta(1)$-multiplicative map $h'_1: C\to A$ such that
\begin{enumerate}\setcounter{enumi}{3}
\item the map $[h'_1]$ is well defined and $[h_1]=\alpha$,
\item $|\tau\circ h_n(f)-\gamma(\tau)(f)|<\min\{\gamma_1(1), \frac{1}{2}\Delta(f): f\in \mathcal H_{1}\}$ for any $f\in \mathcal G_1,$  and
\item  ${\rm dist}(h_n^{\ddag}({\bar u}), \lambda({\bar u}))<\gamma_2(1)$
for any $u\in \mathcal U_n.$
\end{enumerate}

Define $h_1=h_1'$. Assume that $h_1, h_2, ..., h_n: C\to A$ are constructed such that
\begin{enumerate}\setcounter{enumi}{6}
\item $h_i$ is $\mathcal G_i$-$\delta(i)$-multiplicative, $i=1, ..., n$,
\item the map $[h_i]$ is well defined and $[h_i]=\alpha$, $i=1, ..., n$,
\item  $|\tau\circ h_i(f)-\gamma(\tau)(f)|<\min\{\frac{1}{2}\gamma_1(i), \frac{1}{2}\Delta(f): f\in \mathcal H_1({i})\}$ for any $f\in \mathcal G_i$, $i=1, ..., n$,
\item ${\rm dist}(h_i^{\ddag}({\bar u}), \lambda({\bar u}))<\frac{1}{2}\gamma_2(i)$ for any $u\in \mathcal U_i,$ $i=1, ..., n$, and
\item $\|h_{i-1}(g)-h_{i}(g)\|<\frac{1}{2^{i-1}}$ for all $g\in {\cal G}_{i-1},$ $i=2, 3, ..., n$.
\end{enumerate}

Let us construct $h_{n+1}: C\to A$ such that
\begin{enumerate}\setcounter{enumi}{11}
\item $h_{n+1}$ is $\mathcal G_{n+1}$-$\delta(n+1)$-multiplicative,
\item the map $[h_{n+1}]$ is well defined and $[h_{n+1}]=\alpha$,
\item $|\tau\circ h_{n+1}(f)-\gamma(\tau)(f)|<\min\{\frac{1}{2}\gamma_1(n+1), \frac{1}{2}\Delta(f): f\in \mathcal H_1({n+1})\}$ for any $f\in \mathcal G_{n+1}$,
\item  ${\rm dist}(h_{n+1}^{\ddag}({\bar u}), \lambda({\bar u}))<\frac{1}{2}\gamma_2(n+1)$ for any $u\in \mathcal U_,$ $i=1, ..., n$, and
\item $\|h_{n}(g)-h_{n+1}(g)\|<\frac{1}{2^{n}}$ for all $g\in {\cal F}_n$.
\end{enumerate}

Then the statement follows.

By Corollary \ref{istBotC+}, there is $\mathcal G(n+1)$-$\delta(n+1)$-multiplicative map $h'_{n+1}: C\to A$ such that

 $h'_{n+1}$ is $\mathcal G_{n+1}$-$\delta(n+1)$-multiplicative,
 the map $[h'_{n+1}]$ is well defined and $[h_{n+1}']=\alpha$,

 \begin{equation}\label{eqn-n-09-27-02}
        |\tau\circ h'_{n+1}(f)-\gamma(\tau)(f)|<\min\{\frac{1}{2}\gamma_1(n+1), \frac{1}{2}\Delta(f): f\in \mathcal H_2({n+1})\}
        \end{equation}
        for any $f\in \mathcal G_{n+1},$
and
$$
{\rm dist}((h'_{n+1})^{\ddag}({\bar u}), \lambda({\bar u}))<\frac{1}{2}\gamma_2(n+1)
$$ for any $u\in \mathcal U_,$ $i=1, ..., n$.
In particular, this implies that
$$[h'_{n+1}]|_{\mathcal P_n}=[h_n]|_{\mathcal P_n},$$
and for any $f\in \mathcal H_2(n)$ (note that $\mathcal H_2(n)\subset \mathcal G_n$),
\begin{eqnarray*}
|\tau\circ h_n(f)-\tau\circ h'_{n+1}(f)| & < & {{\gamma_1(n)/2}}+|\gamma(\tau)(f)-\tau\circ h'_{n+1}(f)| \\
&< & \gamma_1(n)/2 + \gamma_1(n+1)/2
<\gamma_1(n).
\end{eqnarray*}

Also by \eqref{eqn-n-09-27-02}, for any $f\in\mathcal H_1(n)$, one has
$$\tau(h'_{n+1}(f))\geq \gamma(\tau)(f)-\frac{1}{2}\Delta(f)> \frac{1}{2}\Delta(f).$$
By the inductive hypothesis,  one also has
$$\tau(h_{n}(f))\geq \gamma(\tau)(f)-\frac{1}{2}\Delta(f)> \frac{1}{2}\Delta(f)\rforal f\in\mathcal H_1(n).$$

For any $u\in\mathcal U(n)$, one has
\begin{eqnarray*}
\mathrm{dist}(\overline{h'_{n+1}(u)}, \overline{h_{n}(u)}) & < & \frac{1}{2}\gamma_2(n+1)+\mathrm{dist}(\gamma(\overline{u}),\overline{h_{n}(u)})\\
&< & \frac{1}{2}\gamma_2(n+1) + \frac{1}{2}\gamma_2(n) <\gamma_1(n).
\end{eqnarray*}

Note that both $h'_{n+1}$ and $h_n$ are $\mathcal G(n)$-$\delta(n)$-multiplicative,  and so, by Theorem 12.7 of \cite{GLN-I}, 
there is a unitary $W\in A$ such that
$$\|W^*h_{n+1}'(g)W-h_n(g)\|<1/2^n \rforal g\in {\cal F}_n.$$
Then the map $h_{n+1}:={{AdW\circ h'_{n+1} }}$ satisfies the desired conditions, and the statement is proved.
\end{proof}

\begin{lem}\label{CtimesText00}
Let $C\in\mathcal C_0.$  Let $\ep>0,$ ${\cal F}\subset  C$ be
 any finite subset.
Suppose that $B$ is a unital separable simple \CA\, in ${\mathcal B}_0,$  $A=B\otimes U$ for some UHF-algebra of infinite
type, and $\af\in KK_e(C\otimes C(\T),A)^{++}$.
Then there is a unital $\ep$-${\cal F}$-multiplicative \cp\,
$\phi: C\otimes C(\T)\to A$ such that
\beq
[\phi]=\af.
\eneq
\end{lem}
\begin{proof}
Denote by $\alpha_0$ and $\alpha_1$ the induced maps induced by $\alpha$ on $K_0$-groups and $K_1$-groups.

By Theorem 18.2 of \cite{GLN-I}, 
there exist an $\mathcal F$-$\epsilon$-multiplicative map $\phi_1: C\otimes C(\T)\to A\otimes\cal K$ and a homomorphism $\phi_2: C\otimes C(\T)\to A\otimes\cal K$ with finite dimensional range such that
$$[\phi_{1}]=\alpha+[\phi_{2}]\ \mathrm{in}\ KK(C, A).$$ In particular, one has  $(\phi_1)_{*1}=\alpha_1$.
Without loss of generality, one may assume that both  $\phi_1$ and $\phi_2$ map $C$ into  $M_r(A)$ for some integer $r$.

Since $M_r(A)\in \mathcal B_0$, for any finite subset $\mathcal G\subset M_r(A)$ and any $\epsilon'>0$, there are  $\mathcal G$-$\epsilon'$-multiplicative maps $L_1: M_r(A)\to (1-p)M_r(A)(1-p)$ and $L_2: M_r(A)\to S_0\subset  pM_r(A)p$ for a C*-subalgebra $S_0\in \mathcal C_0$ with $1_{S_0}=p$ such that
\begin{enumerate}
\item $||a-L_1(a)\oplus L_2(a)||<\epsilon'$ for any $a\in\mathcal G$ and
\item $\tau((1-p))<\epsilon'$ for any $\tau\in T(M_r(A)).$
\end{enumerate}
Since $K_1(S_0)=\{0\}$, choosing $\mathcal G$ sufficiently large and $\epsilon'$ sufficiently small, one may assume that $L_1\circ\phi_1$ is {{$\cal F$-$\ep$}}-multiplicative, and
$$[L_1\circ \phi_1]|_{K_1(C\otimes C(\T))}=(\phi_1)_{*1}=\alpha_1.$$ Moreover, since the positive cone of $K_0(C\otimes C(\T))$ is finitely generated, choosing $\epsilon'$ even smaller, one may assume that the map
$$\kappa:=\alpha_0-[L_1\circ\phi_1]|_{K_0(C\otimes C(\T)}: K_0(C\otimes C(\T))\to K_0(A)$$ is positive.
Pick a point $x_0\in \T$, and consider the evaluation map $$\pi: C\otimes C(\T) \in f\otimes g \mapsto f\cdot g(x_0)\in C.$$ Then $\pi_{*0}: K_0(C\otimes C(\T))\to K_0(C)$ is an order isomorphism, since $K_1(C)=0$.

Choose a projection $q\in A$ with $[q]=\kappa([1])$. Since $qAq\in\mathcal B_0$, by Corollary 18.9 
of \cite{GLN-I}, 
there is a unital homomorphism $h: C\to qAq$ such that $$[h]_0=\kappa\circ \pi_{*0}^{-1}\quad\textrm{on $K_0(C)$},$$ and hence one has $$(h\circ\pi)_{*0}=\kappa,\quad \textrm{on $K_0(C\otimes C(\T))$.}$$
Put $\phi=(L_1\circ\phi_1)\oplus (h\circ\pi): C\otimes C(\T)\to A.$ Then  it is clear that
$$
\phi_{*0}=[L_1\circ\phi_1]|_{K_0(C\otimes C(\T))}+\kappa=[L_1\circ\phi_1]|_{K_0(C\otimes C(\T))} + \alpha_0-[L_1\circ\phi_1]|_{K_0(C\otimes C(\T))}=\alpha_0
$$
$$ \andeqn[\phi]_1=[L_1\circ\phi_1]|_{K_1(C\otimes C(\T))}=\alpha_1.$$
Since $K_*(C\otimes C(\T))$ is finitely generated and torsion free, one has that $[\phi]=\alpha$ in $KK(C\otimes C(\T), A)$.
\end{proof}

\begin{lem}\label{CtimesText0}
Let $C\in\mathcal C_0.$
Let $\ep>0,$ ${\cal F}\subset  C\otimes\mathrm{C}(\T)$ be
 a  finite subset, $\sigma>0,$ and
${\cal H}\subset (C\otimes C(\T))_{s.a.}$ be a finite subset.
Suppose that $A$ is a unital
\CA\, in ${\cal B}_0,$  $B=A\otimes U$ for some UHF-algebra {$U$} of infinite
type, $\af\in KK_e(C\otimes C(\T), B)^{++}$, and $\gamma: T(B)\to T_{\rm f}(C\otimes C(\mathbb T))$ is a continuous affine
map such that $\af$ and $\gamma$ are compatible.
Then there is a unital {{$\cal F$-$\ep$}}-multiplicative
\cp\ $\phi: C\otimes C(\T)\to B$ such that
\begin{enumerate}
\item $[\phi]=\alpha$ and
\item  $|\tau\circ \phi(h)-\gamma(\tau)(h)|<\sigma$ for any $h\in {\cal H}$.
\end{enumerate}
Moreover, if $A\in {\cal B}_1,$ $\bt\in KK_e(C, A)^{++},$ $\gamma': T(A)\to T_f(C)$ is a continuous affine map
which is compatible with $\bt,$  and ${\cal H}'\subset C_{s.a.}$ is a finite subset,
then there is also a unital \hm\, $\psi: C\to A$ such that
\beq\label{Ct-2}
[\psi]=\bt \tand |\tau\circ \psi(h)-\gamma'(\tau)(h)|<\sigma \rforal f\in {\cal H}'.
\eneq
\end{lem}

\begin{proof}
Since $K_*(C\otimes C(\T))$ is finitely generated and torsion free, by the UCT, the element $\alpha\in KK(C\otimes C(\T), A)$ is determined by the induced maps $\alpha_0\in \mathrm{Hom}(K_0(C\otimes C(\T)), K_0(A))$ and $\alpha_1\in \textrm{Hom}(K_1(C\otimes C(\T)), K_1(A))$.
We may assume that projections in $M_{r}(C\otimes C(\T))$ (for some fixed integer $r>0$) generate $K_0(C\otimes C(\T)).$

We may also assume
that $\|h\|\le 1$ for all $h\in {\cal H}.$ Fix a finite generating set $\mathcal G$ of $K_0(C\otimes C(\T))$. Since $\gamma(\tau)\in T_{\mathrm{f}}(C\otimes\mathrm{C}(\mathbb T))$ for all $\tau\in T(B)$ and $\tau(B)$ is compact, one is able to define $\Delta: {{(C\otimes C(\T))}}_+^{q, 1}\setminus\{0\} \to (0, 1)$
by
$$\Delta(\hat{h})=\frac{1}{2}\inf\{\gamma(\tau)(h):\ {{\tau}}\in T(B)\}.$$
Fix a finite generating set $\mathcal G$ of $K_0(C\otimes C(\T))$. Let $\mathcal H_1\subset C\otimes C(\mathbb T)$, $\delta>0$, and $K\in\mathbb N$ be the finite subset and the constants of Lemma 16.10 of \cite{GLN-I} 
with respect to $\mathcal F$, $\mathcal H$,  $\epsilon$, $\sigma/4$ (in place of $\sigma$), and $\Delta$.

Since $A\in \mathcal B_0$ and  $U$ is {{of infinite type,}}
for any finite subset $\mathcal G'\subset B$ and any $\epsilon'>0$, there are  unital $\mathcal G'$-$\epsilon'$-multiplicative \cp s $L_1: B\to (1-p)B(1-p)$ and $L_2: B\to D\otimes 1_{M_K}\subset D\otimes M_K\subset  pBp$ for a C*-subalgebra $D\in \mathcal C_0$ with $1_{D\otimes M_K}=p$ such that
\begin{enumerate}\setcounter{enumi}{2}
\item $||a-L_1(a)\oplus L_2(a)||<\epsilon'$ for any $a\in\mathcal G'$, and
\item $\tau((1-p))<\min\{\epsilon', \sigma/4\}$ for any $\tau\in T(B).$
\end{enumerate}
Put $S=D\otimes M_K.$ Since $K_i(C\otimes C(\T))$ is finite generated,
${\rm Hom}_{\Lambda}(\underline{K}(C\otimes C(\T)), \underline{K}(C\otimes C(\T)))$ is determined
on a finitely generated subgroup $G^{K}$  of $\underline{K}(C\otimes C(\T))$
(see Corollary 2.12 of \cite{DL}).
Choosing $\mathcal G'$ large enough and $\epsilon'$ small enough, one may assume
$[L_1]$ and $[L_2]$ are well defined on $\alpha(G^K),$
and
\begin{equation}\label{July20-1}
\alpha=[L_1]\circ\alpha+[j]\circ [L_2]\circ\alpha,
\end{equation}
where $j: S\to A$ is the embedding. Note that since $K_1(S)=\{0\}$, one has
$$\alpha_1=[L_1]\circ\alpha|_{K_1(C\otimes C(\T))}.$$
Define $\kappa'=[L_2]\circ\alpha|_{K_0(C\otimes C(\T))}$, which  is a homomorphism  from $K_0(C\otimes C(\T))$ to {{$K_0(D\otimes 1_{M_K})=K_0(D)$ (here we identify $D\otimes 1_{M_K}$ with $D$)}}
which maps  $[1_{C\otimes C(\T)}]$ to {{$[1_{D\otimes 1_{M_K}}]$}}. Let $\{e_{i,j}: 1\le i,j\le K\}$ be a system of matrix units for $M_K.$
View  $e_{i,j}\in D\otimes M_K.$ Then $e_{i,j}$ commutes with the image of $L_2.$
Define $L_2': B\to D\otimes e_{1,1}$ by $L_2'(a)={{e_{11}L_2(a) e_{1,1}}}$ for all $a\in B$. 

Put $\kappa=[L_2']\circ \af|_{K_0(C\otimes C(\T))}.$ Put $D'=D\otimes e_{1,1}.$

Choosing $\mathcal G'$ larger and $\epsilon'$ smaller,  if necessary, one has
a continuous affine map $\gamma': T(D')\to T(C\otimes C(\T))$ such that, for all $\tau\in T(A),$
\begin{enumerate}\setcounter{enumi}{4}
\item $|\gamma'(\frac{1}{\tau(e_{1,1})}\tau|_{D'})(f)-\gamma(\tau)(f)|<\sigma/4$ for any $f\in\mathcal H$,
\item $\gamma'(\tau)(h)>\Delta(\hat{h})$ for any $h\in\mathcal H_{1}$, and
\item $|\gamma'(\frac{1}{\tau(e_{1,1})}\tau|_{D'})(p)-\tau(\kappa([p]))|<\dt$ for all projections $p\in M_r(C\otimes C(\T)).$ 
\end{enumerate}
Then it follows from Theorem 16.10 of \cite{GLN-I} 
that there is an ${\cal F}$-$\epsilon$-multiplicative
\morp\, $\phi_2: C\otimes C(\T)\to M_K(D)=S$ such that
\begin{equation*}
(\phi_2)_{*0}=K\kappa=\kappa'
\end{equation*}
and
\begin{equation*}
|(1/K)t\circ  \phi_2(h)-\gamma'(t)(h)|<\sigma/4,\quad h\in\mathcal H,\ t\in T(D').
\end{equation*}

On the other hand, since $(1-p)A(1-p)\in \mathcal B_0$, by Lemma \ref{CtimesText00}, there is a unital  $\mathcal F$-$\epsilon$-multiplicative \cp\, $\phi_1: C\otimes C(\T)\to (1-p)A(1-p)$ such that
$$[\phi_1]=[L_1]\circ\alpha\quad\mathrm{in}\ KK(C\otimes C(\T), A).$$
Define $\phi=\phi_1\oplus j\circ \phi_2: C\otimes C(\T)\to (1-p)A(1-p)\oplus S\subset A$. Then, by (\ref{July20-1}), one has
$$\phi_{*0}=(\phi_1)_{*0}+(j\circ \phi_2)_{*0}=([L_1]\circ\alpha)|_{K_0(C\otimes C(\T))}+([j\circ L_2]\circ\alpha)|_{K_0(C\otimes C(\T))}=\alpha_0$$ and
$$\phi_{*1}=(\phi_1)_{*1}+(j\circ \phi_2)_{*1}=([L_1]\circ\alpha|_{K_1(C\otimes C(\T))}=\alpha_1.$$ Hence $[\phi]=\alpha$ in $KK(C\otimes C(\T))$.

For any $h\in \mathcal H$ and any $\tau\in T(A)$, one has (note that
$\|h\|\le 1$ for all $h\in {\cal H}$, and $\tau(1-p)<\dt/4$),
\begin{eqnarray*}
&&|\tau\circ\phi(h)-\gamma(\tau)(h)|\\
&<& |\tau\circ\phi(h)-\tau\circ j\circ \phi_2(h)|+|\tau\circ j\circ \phi_2(h)-\gamma(\tau)(h)|\\
&<& \sigma/4+|\tau\circ j\circ \phi_2(h)-\gamma'({1\over{\tau(e_{1,1})}}\tau|_{D'})(h)|+
|\gamma'({1\over{\tau(e_{1,1})}}\tau|_{D'})(h)-\gamma(\tau)(h)|\\
&<& \sigma/4+|\tau\circ j\circ \phi_2(h)-\gamma'(\frac{1}{\tau(p)}\tau|_S)(h)|+ |\gamma'(\frac{1}{\tau(p)}\tau|_S)(h)-  \gamma(\tau)(h)|\\
&<& \sigma/4+\sigma/4+\sigma/4<\sigma,
\end{eqnarray*}
where we identify $T(D')$ with $T(S)$ in a standard way for $S=D'\otimes M_K$.
Hence the map $\phi$ satisfies the requirements of the lemma.

To see the last part of the lemma holds, we note that, when $C\otimes C(\T)$ is replaced by $C$ and $A$ is assumed
to be in ${\cal B}_1,$ the only difference is that we cannot use \ref{CtimesText00}.  But then we can appeal to  Theorem 18.7 of \cite{GLN-I} 
to obtain $\phi_1.$   The semiprojectivity of $C$ allows us actually to obtain a unital homomorphism (see Corollary 18.9 of \cite{GLN-I}).
\end{proof}


\begin{cor}\label{CorCinjj}
Let $C\in\mathcal C_0.$
Suppose that $A$ is a unital separable simple \CA\, in $\mathcal B_{{0}}$,  $B=A\otimes U$ for some UHF-algebra of infinite
type, $\af\in KK_{e}(C, B)^{++},$ and
 $\gamma: T(B)\to T_{\rm f}(C))$ is a continuous affine
map. Suppose that $(\af, \lambda, \gamma)$ is a compatible triple. Then there is a unital
\hm\, $\phi: C\to B$ such that
$$
[\phi]=\af \tand \phi_T=\gamma.
$$
In particular, $\phi$ is a monomorphism.
\end{cor}

\begin{proof}
The proof is exactly the same as the argument employed in \ref{istBotC++} but using  the second part of Lemma \ref{CtimesText0} instead of \ref{istBotC+}.
The reason $\phi$ is a monomorphism is because $\gamma(\tau)$ is faithful for each $\tau\in T(A).$
\end{proof}

\begin{lem}\label{alg-cut}
Let $C$ be a unital C*-algebra.
Let $p\in C$ be a full projection. Then, for any $u\in U_0(C)$, there is a unitary $v\in pCp$ such that $$\overline{u}=\overline{v\oplus (1-p)}\quad\textrm{in $U_0(C)/CU(C)$}.$$
If, furthermore, $C$ is separable and has stable rank one, then, for any $u\in U(C)$, there is a unitary $v\in pCp$ such that $$\overline{u}=\overline{v\oplus (1-p)}\quad\textrm{in $U(C)/CU(C)$}.$$
\end{lem}

\begin{proof}
It suffices to prove the first part of the statement. This is essentially contained in the proof of  4.5 and 4.6 of \cite{GLX}.
As in the proof of 4.5 of \cite{GLX}, for any $b\in C_{s.a.},$ there is $c\in pCp$ such that
$b-c\in C_0,$ where
$C_0$ is the closed subspace of $A_{s.a.}$ consisting of elements
of the form $x-y,$ where $x=\sum_{n=1}^{\infty}c_n^*c_n$ and $y=\sum_{n=1}^{\infty}c_nc_n^*$
(convergence in norm) for  some sequence $\{c_n\}$ in $C.$

Now let $u=\prod_{k=1}^n\exp(ib_k)$ for some $b_k\in C_{s.a.},$ $k=1,2,...,n.$
Then there are $c_k\in pCp$ such that $b_k-c_k\in  C_0,$ $k=1,2,...,n.$
Put $v=p(\prod_{k=1}^n \exp(ic_k))p.$ Then $v\in U_0(pCp)$ and
$v+(1-p)=\prod_{k=1}^n\exp(i c_k).$   By 3.1 of \cite{Thomsen-rims},
$u^*(v+(1-p))\in CU(C).$
\end{proof}

\begin{lem}\label{Popa}
Let ${\cal D}$ be the family of unital separable  residually finite dimensional \CA s  and let
$A$ be a unital simple separable \CA\, which has the property ($L_{\cal D}$) {{(see 9.4 of \cite{GLN-I})}} and  the property  (SP). Then
$A$ satisfies  the    Popa condition:
Let $\ep>0$ and let ${\cal F}\subset A$ be a finite subset.
There exists a finite dimensional \SCA\, $F\subset A$  with $P=1_F$ such that
\beq\label{Popa-1}
\|[P, x]\|<\ep,\,\,\, PxP\in_{\ep}F\andeqn \|PxP\|\ge \|x\|-\ep
\eneq
for all $x\in {\cal F}.$  In particular, if $A\in {\cal B}_1$ and $A$ has the property (SP), then $A$ satisfies the Popa condition.
\end{lem}

\begin{proof}
We may assume that ${\cal F}\subset A^{\bf 1}$ and $0<\ep<1/2.$
Without loss of generality, we may assume that
$$
d=\min\{\|x\|: x\in {\cal F}\}>0.
$$
Since $A$ has property ($L_{\cal D}$), there are a projection $p\in A$ and a \SCA\, $D\subset A$ with $D\in\mathcal {\cal D}$ and $p=1_D$ such that

\beq\label{popa-0}
\|px-xp\|<d\epsilon/16,
\,\, pxp\in_{d\epsilon/16} D,\andeqn
\|pxp\|\ge (1-\ep/16)\|x\|
\eneq
for all $x\in {\cal F}$ (see  {{9.5 of  \cite{GLN-I}}}).

Let ${\cal F}'\subset D$ be a finite subset such that, for each $x\in {\cal F},$ there exists $x'\in {\cal F}'$ such
that $\|pxp-x'\|<d\ep/16.$
Since $D\in {\cal D},$ there is a unital surjective \hm\,
$\pi: D\to D/{\rm ker}\pi$ such that $F_1:=D/{\rm ker}\pi$ is a finite dimensional \CA\, and
\beq\label{popa-2}
 \|\pi(x')\|\ge (1-\ep/16)\|x'\|\rforal x'\in {\cal F}'.
\eneq
Let $B=\overline{({\rm ker}\pi) A({\rm ker}\pi)}.$ $B$ is a hereditary \SCA\, of $A.$
Let $C$ be the closure of $D+B.$
Note that $1_C=1_D=p.$ As in the proof of 5.2 of \cite{LnTAF},   $B$ is an ideal of $C$ and
$C/B\cong D/{\rm ker}\pi=F_1.$
The lemma then follows from Lemma 2.1 of \cite{Niu}. In fact, since $pAp$ has property (SP), by
Lemma 2.1 of \cite{Niu}, there are a projection $P\in pAp$  and a monomorphism $h: F_1\to PAP$ such that
\beq\label{popa-n3}
&&h(1_{F})=P,\,\,\, \|Px'-x'P\|<\ep/16 \andeqn\\
&& \|h\circ \pi(x')-Px'P\|<\ep\cdot d/16
\eneq
for all  $x'\in {\cal F}'.$
Put $F=h(F_1).$ Then, one estimates that, for all $x\in {\cal F},$
\beq\label{popa-n4}
&&\|Px-xP\|\le \|Ppx-Px'\|+\|Px'+x'P\|+\|x'P-xpP\|\\
&&<\ep/16+\ep/16 +\ep/16<\ep,\\
&&PxP\approx_{\ep/16} Px'P\in_{\ep/16} {{F_1}},\andeqn\\
&& \|PxP\|=\|PpxpP\|\ge \|Px'P\|-d\ep/16\ge \|h\circ \pi(x')\|-d\ep/8\\
&&=\|\pi(x')\|-d\ep/8
\ge \|x'\|-d\ep/16-d\ep/8\ge \|pxp\|-d\ep/4\\
&&\ge (1-\ep/16)\|x\|-d\ep/4
\ge \|x\|-\ep.
\eneq
\end{proof}

\begin{lem}\label{CtimesText}
Let $C\in\mathcal C_0.$
Let $\ep>0,$ ${\cal F}\subset  C$ be
a finite subset, $1>\sigma_1>0,$  $1>\sigma_2>0,$  $\overline{{\cal U}}\subset J_c(K_1(C\otimes C(\T)))\subset U(C\otimes C(\T))/CU(C\otimes C(\T))$ be a finite subset (see Definition 2.16 of \cite{GLN-I})
and
${\cal H}\subset (C\otimes C(\T))_{s.a.}$ be a finite subset.
Suppose that $A$ is a unital separable simple \CA\, in $\mathcal B_0$,  $B=A\otimes U$ for some UHF-algebra {{$U$ of infinite}}
type, $\af\in KK_{e}(C\otimes C(\T), B)^{++},$
$\lambda: J_c(K_1(C\otimes C(\T))) \to U(B)/CU(B)$ is a homomorphism, and $\gamma: T(B)\to T_{\rm f}(C\otimes C(\mathbb T))$ is a continuous affine
map. Suppose that $(\af, \lambda, \gamma)$ is a compatible triple. Then there is a unital ${\cal F}$-$\ep$-multiplicative
\cp\, $\phi: C\otimes C(\T)\to B$ such that
\begin{enumerate}
\item $[\phi]=\alpha$,
\item ${\rm dist}(\phi^{\ddag}(x), \lambda(x))<\sigma_1$, for any $x\in \overline{{\cal U}}$, and
\item  $|\tau\circ \phi(h)-\gamma(\tau)(h)|<\sigma_2$, for any $h\in {\cal H}$.
\end{enumerate}

\end{lem}

\begin{proof} Note that $\underline{K}(C\otimes C(\T))$ is finitely generated modulo Bockstein operations and $K_0(C\otimes C(\T))_+$ is a finitely generated semigroup.  Using the inductive limit $B=\lim_{n\to\infty}(A\otimes M_{r_n}, \imath_{n,n+1}),$  one can find, for $n$ large enough,  $\alpha_n\in KK_e(C\otimes C(\T), A\otimes M_{r_n})^{++}$  such that $\alpha=\alpha_n\times [  \imath_n]$ where $ [  \imath_n] \in KK(A\otimes M_{r_n}, B)$ is induced by the inclusion $ \imath_n: A\otimes M_{r_n} \to B$. Replacing  $A$ by $A\otimes M_{r_n}$,
we may assume that $\alpha=\alpha_1 \times [\imath]$, where $\alpha_1\in KK_e(C\otimes C(\T), A)^{++}$ and $ \imath: A \to A\otimes U=B$ is the inclusion.  Note that  $\lambda: J_c(K_1(C\otimes C(\T)))\to U(A\otimes U)/CU(A\otimes U).$ By
the same argument as above, we know that if the integer $n$ above is large enough, then there is a  map  $\lambda_n: J_c(K_1(C\otimes C(\T))) \to U(A\otimes M_{r_n})/CU(A\otimes M_{r_n})$ such that $|\imath_n^{\ddag}\circ \lambda_n(u)-\lambda(u)|$ is arbitrarily small (e.g smaller than {$\sigma_1/4$}) for all $u\in  \overline{{\cal U}}$.
Replacing $A$ by $ A\otimes M_{r_n}$, we may assume $\lambda=\imath^{\ddag}\circ \lambda_1$ 
 with $\lambda_1: J_c(K_1(C\otimes C(\T))) \to U(A)/CU(A)$ and $ \imath^{\ddag}: U(A)/CU(A)\to U(B)/CU(B)$
induced by the inclusion map.  Furthermore, we may assume that $\lambda_1$ is compatible with $\af_1$.

Without loss of generality, we may assume that $\|h\|\le 1$ for all $h\in {\cal H}.$ Let $p_i, q_i\in M_k(C)$ be projections such that
$\{[p_1]-[q_1], ..., [p_d]-[q_d]\}$ forms a set of  independent
generators of $K_0(C)$ (as an abelian group)
for some integer $k\ge 1$.
Choosing a specific $J_c,$
one may assume that
$$
\overline{\mathcal U}=\{(({\bf 1}_k-p_i)+p_i\otimes z)(({\bf 1}_k-q_i)+q_i\otimes z^*): 1\le i\le d\},
$$
where $z\in C(\T)$ is the identity function on the unit circle.
Put $u_i'=({\bf 1}_k-p_i)+p_i\otimes z)(({\bf 1}_k-q_i)+q_i\otimes z^*).$ Hence, $\{[u_1'], ..., [u_d']\}$ is a set of standard generators of $K_1(C\otimes C(\T))\cong K_0(C)\cong \mathbb Z^d.$ Then $\lambda$ is a homomorphism from $\mathbb{Z}^d$ to $U(B)/CU(B)$.

Let $\pi_e: C\to F_1=\bigoplus_{i=1}^l M_{n_i}$ be the standard evaluation map defined in Definition 3.1 of \cite{GLN-I}. 
By Proposition 3.5 of \cite{GLN-I}, 
the map $(\pi_e)_{*0}$ induces an embedding of $K_0(C)$ in $\mathbb{Z}^l$, and the map $(\pi_e\otimes\id)_{*1}$ induces an embedding of $K_1(C\otimes C(\T))\cong \mathbb{Z}^d$ in $K_1(\bigoplus_{i=1}^lM_{n_i}\otimes C(\T)) \cong \mathbb{Z}^l$. 
Define
$J_c({{K_1(}}\bigoplus_{i=1}^lM_{n_i}\otimes C(\T){{)}})$ to be the subgroup generated by $\{\overline{e_i\otimes z_i\oplus (1-e_i)}; i=1, ..., l\}$, where $e_i$ is a rank one projection of $M_{n_i}$ and $z_i$ is the standard unitary generator of  the $i$-th copy of $C(\T).$  Note that the image of $J_c(K_1(C\otimes C(\T)))$ under $\pi_e$ is contained in $J_c(K_1(\bigoplus_{i=1}^lM_{n_i}\otimes C(\T)))$. 
Write $w_j=e_j\otimes z_j\oplus (1-e_j)$, $1\leq j\leq l$.

Let $U$ be as in the lemma. We write $B=B_0\otimes U_{{2}},$ and {{ $B_0=A\otimes U_{1},$
with $U=U_1\otimes  U_2$,  and
both $U_1$ and $U_2$ UHF algebras of infinite type. Denote by $\imath_1:A \to  B_0$, $\imath_2:B_0 \to  B,$ and $\imath=\imath_2\circ \imath_1:A \to  B$ the inclusion maps. Recall $\af=\af_1\times [\imath] \in KK(C,B)$.}}


Applying Lemma \ref{CtimesText0},
one obtains a unital ${\cal F}'$-$\ep'$-multiplicative \cp~
$\psi: C\otimes\mathrm{C}(\T) \to B_0$ such that
\begin{equation}\label{July22-a1}
[\psi]=\alpha_{{1}}{{\times [\imath_1]}}\andeqn
\end{equation}
\begin{equation}\label{July22-a2}
|\tau\circ \psi(h)-\gamma(\tau)(h)|<\min\{\sigma_1, \sigma_2\}/3 \rforal  h\in {\cal H},\andeqn \rforal \tau\in\mathrm{T}(B_0),
\end{equation}
where $\ep/2>\ep'>0$ and ${\cal F}_1\supset {\cal F}.$  (Note that $T(B_0)=T(A)=T(B)$, and the map $\gamma: T(B)\to T_{\rm f}(C\otimes C(\mathbb T))$ can be regarded as a map with domain $T(B_0)$).
We may assume that $\ep'$ is sufficiently   small and ${\cal F}_1$ is sufficiently  large  that
not only (\ref{July22-a1}) and (\ref{July22-a2}) make sense
but also  that $\psi^{\ddag}$ can be defined on $\bar {\mathcal U},$ and induces a
homomorphism from $J_c(K_1(C\otimes C(\T)))$ to $U(B_{{0}})/CU(B_{{0}})$ (see 2.17 of \cite{GLN-I}).

Let $M$ be the integer of Corollary 15.3 of \cite{GLN-I} 
for $K_0(C)\subset \Z^l$ (in place of $G\subset \Z^l$).

For any $\ep''>0$ and any finite subset ${\cal F}''\subset B_0,$ since $B_0$ has the Popa condition and has the property (SP)
(see \ref{Popa}),
there exist a non-zero projection $e\in B_0$ and a unital ${\cal F}''$-$\ep''$-multiplicative \cp\, $L_0: B_0\to F\subset eB_0e,$
where $F$ is a finite  dimensional and $1_{{F}}=e,$ and a  unital ${\cal F}''$-$\ep''$-multiplicative \cp\, $L_1: B_0\to
(1-e)B_0(1-e)$ such that
\beq\label{July22-b1}
\|b-\imath\circ L_0(b)\oplus L_1(b)\|<\ep''\rforal b\in {\cal F}'',\\
\|L_0(b)\| \ge \|b\|/2\rforal b\in {\cal F}'', \andeqn\\ \label{Dec1407-n1}
\tau(e)<\min\{\sigma_1/2, \sigma_2/2\}\rforal \tau\in T(B_0),
\eneq
where $\imath: F\to eB_0e$ is the embedding and $L_1(b)=(1-p)b(1-p)$ for all $b\in B_0.$

Since the positive cone of $K_0(C\otimes C(\T))$ is finitely generated, with sufficiently small
$\ep''$ and sufficiently large ${\cal F}''$, one may assume that
$[L_0\circ \psi]|_{K_0(C\otimes C(\T))}$ is positive. Moreover, one may assume that $(L_0\circ \psi)^{\ddag}$ and
$(L_1\circ \psi)^{\ddag}$ are  well defined and induce homomorphisms from $J_c(K_1(C\otimes C(\T)))$ to $U(B_{{0}})/CU(B_{{0}})$. One may also assume that $[L_1\circ \psi]$ is well defined.  Moreover, we may assume that $L_i\circ \psi$ is ${\cal F}$-$\ep$-multiplicative for $i=0,1.$

There is a projection $E_c\in U_{{2}}$ such that $E_c$ is a direct sum of $M$ copies of some non-zero projections $E_{c,0}\in U_{{2}}.$
Put $E=1_{{U_2}}-E_c.$

Define $\phi_0: C\otimes C(\T)\to F\otimes EU_{{2}}E\to eB_0e\otimes EU_{{2}}E$ by $\phi_0(c)= L_0\circ \psi(c)\otimes E{{(\in B)}}$ for all $c\in C\otimes C(\T)$ and
define $\phi_1': C\to F\otimes E_cU_{{2}}E_c$ by $\phi_1'(c)=L_0\circ \psi(c)\otimes E_c$ for all $c\in C.$
Note that
$\phi_0$ is also { {${\cal F}$-$\ep$-}}multiplicative and $\phi_0^{\ddag}$ is also well defined as
$(L_0\circ \psi)^{\ddag}$ is.  Moreover $[\phi_1']$ is well defined.  Define
$$L_2={{\imath_2 \circ}}L_1\circ \psi+\phi_0{{: C\otimes C(\T) \to \big((1-e)B_0(1-e)\otimes 1_{U_2}\big)\oplus \big(eB_0e\otimes EU_2E\big)~~(\subset B)}}.$$

Denote by
$$\lambda_0=\lambda-L_2^{\ddag}=\lambda-\phi_0^{\ddag}-({{\imath_2 \circ}}L_1\circ \psi)^{\ddag}: J_c(K_1(C\otimes C(\T)))\to U(B)/CU(B).$$
{{Note that $L_0$ factors through the finite dimensional algebra $F$ and therefore $[L_0]=0$ on $K_1(B_0).$ Consequently $[\phi_0]|_{K_1(C\otimes C(\T))}=0$  and $[L_1\circ \psi]=[\imath_1]\circ[\alpha_1]$ on ${K_1(C\otimes C(\T))}$. Hence, $[\imath_2\circ L_1\circ \psi]=\alpha$ on $K_1$. Furthermore, $\alpha$ is compatible with $\lambda$. We know that the image of $\lambda_0$ is in $U_0(B)/CU(B).$}}

Note that, by Lemma 11.5 of \cite{GLN-I}, 
the group $U_0(B)/CU(B)$ is divisible. It is an injective abelian group.
Therefore there  is a homomorphism ${{\tilde \lambda}}: J_c(\bigoplus_{i=1}^lM_{n_i}\otimes C(\T))\to U_{{0}}(B)/CU(B)$
such that
\begin{equation}\label{July22-a4}
{{\tilde \lambda}}\circ (\pi_e)^{\ddag}=\lambda_0-L_2^{\ddag}.
\end{equation}

Let $\bt=[{ {L_0\circ \psi}}]|_{K_0(C)}{{:K_0(C) \to}}K_0(F)=\Z^n.$  Let $R_0\ge 1$ be the integer given by Corollary 15.3 of \cite{GLN-I} 
for $\bt: K_0(C)\to \Z^n$ (in place of $\kappa: G\to \Z^n$; note that that $[E_c]$ is divisible by $M$ implies that every element in $\bt(K_0(C))$ is divisible by $M$).
There is a unital \SCA\, $M_{MK}\subset E_cU_{{2}}E_c$ such that $K\ge R_0$ { {and such that $E_cU_{2}E_c$ can be written as $M_{MK}\otimes U_3$}}.  It follows from Corollary 15.3 of \cite{GLN-I} 
that there is a positive \hm\, $\bt_1: K_0(F_1)\to K_0(F)$ such that
$\bt_1\circ (\pi_e)_{*0}=MK\bt.$
Let $h: F_1\to F\otimes M_{MK}$ be the unital \hm\, such that $h_{*0}=\beta_1.$
Put $\phi_1''=h\circ \pi_e{{: C\to F\otimes M_{MK}}}$, and then one has $(\phi_1'')_{*0}=MK\bt.$ Let $J: M_{MK}\to E_cU_{{2}}E_c$ be the embedding.
One verifies that
\beq\label{120714-18nn2}
(\imath_{{F}}\otimes J)_{*0}\circ (\phi_1'')_{*0}=(\imath_{{F}}\otimes J)_{*0}\circ MK\bt={{\tilde\imath}}_{*0}\circ (\phi_1')_{*0}{{,}}
\eneq
 where $\imath_{F}:F\to eB_0e$ and ${\tilde \imath}: F\otimes E_cU_2E_c \to eB_0e\otimes E_cU_2E_c $
is  the unital embedding.

Choose a unitary  $y_i\in { {(\imath_F\otimes J\circ h)}}(e_j)B{ {(\imath_F\otimes J\circ h)}}(e_j)$ such that
$${\bar y}_j={{\tilde \lambda}}(w_j),\,\,\,j=1,2,...,l,$$
where we recall that $w_j=e_j\otimes z_j\oplus (1-e_j)\in F_1\otimes C(\T)=\oplus_j M_{n_j}\otimes C(\T)$ is
one of the chosen
generator of $K_1(M_{n_j}\otimes C(\T))$. Let $1_j$ be the unit of $M_{n_j}\subset F_1;$ then $1_j=\underbrace{e_j\oplus e_j\oplus \cdots\oplus e_j}_{n_j}$.

Define ${ {\tilde y}}_j={\rm diag}(\overbrace{y_j,y_j,...,y_j}^{{n_j}}){{\in  (\imath_F\otimes J\circ h)(1_j)B (\imath_F\otimes J\circ h)(1_j)}},$ $j=1,2,...,l.$ Then ${\tilde y}_j$ commutes with $(\imath_F\otimes J )(F_1)$.

Define ${\tilde  \phi}_1: F_1\otimes C(\T)\to { {\big((\imath_F\otimes J)\circ \phi_1''\big)}}(1_C)B{ {\big((\imath_F\otimes J)\circ \phi_1''\big)}}(1_C)$ by\\
 ${\tilde \phi}_1(c_j\otimes f)={ {\big((\imath_F\otimes J)\circ \phi_1''\big)}}(c_j)f({{\tilde y}}_j)$ for
all $c_j\in M_{n_j}$ and $f\in C(\T).$ Define $\phi_1={\tilde \phi}_1\circ (\pi_e\otimes {\rm id}_{C(\T)}).$
Then, by identifying $K_0(C\otimes C(\T))$ with $K_0(C),$ one has
\begin{equation}\label{July23-4}
(\phi_1)_{*0}={{\tilde \imath}}_{*0}\circ (\phi_1')_{*0}\andeqn (\phi_1)^{\ddag}={{\tilde \lambda}}.
\end{equation}
Define $\phi=\phi_0\oplus \phi_1\oplus {{\imath_2 \circ}} L_1\circ \psi.$
By \eqref{July22-a2} and \eqref{Dec1407-n1},
\begin{equation*}
|\tau\circ \phi(h)-\gamma(\tau)(h)|<\sigma_2/3+\sigma_2/3=2\sigma_2/3 \rforal  h\in {\cal H}.
\end{equation*}
It is ready to verify that
$\phi_{*0}=\af|_{K_0(C\otimes C(\T))}\andeqn \phi^{\ddag}=\lambda.$
Thus, since $\lambda$ is compatible with $\af,$
\beq\label{120714-18-nn4}
\phi_{*1}=\af|_{K_1(C\otimes C(T))}.
\eneq
Since  $K_{*i}(C\otimes C(\T))\cong K_0(C)$ is free and finitely generated,
one concludes that
\begin{equation*}
[\phi]=\af.
\end{equation*}
\end{proof}

\begin{cor}\label{CorCinj}
Let $C\in\mathcal C_0$ and $C_1=C\otimes C(\T).$
Suppose that $A$ is a unital separable simple \CA\, in $\mathcal B_0$,  $B=A\otimes U$ for some UHF-algebra $U$ of infinite
type, $\af\in KK_{e}(C_1, B)^{++},$
$\lambda: J_c(K_1(C)) \to U(B)/CU(B)$ is a homomorphism, and
 $\gamma: T(B)\to T_f(C_1))$ is a continuous affine
map. Suppose that $(\af, \lambda, \gamma)$ is a compatible triple. Then there is a unital
\hm\, $\phi: C_1\to B$ such that
$$ [\phi]=\af,\,\,
 \phi^{\ddag}|_{J_c(K_1(C))}=\lambda\andeqn
\phi_T=\gamma.$$
In particular, $\phi$ is a monomorphism.
\end{cor}

\begin{proof}
The proof is exactly the same as the argument employed in \ref{istBotC++}  using   \ref{CtimesText}.
\end{proof}

\begin{cor}\label{CCorCinj}
Let $C\in\mathcal C_0$ and let $C_1=C$ or $C_1=C\otimes C(\T).$
Suppose that $A$ is a unital separable simple \CA\, in $\mathcal B_0$,  $B=A\otimes U$ for some UHF-algebra of infinite
type, $\af\in KK_{e}(C_1, B)^{++},$
and $\gamma: T(B)\to T_f(C_1)$ is a continuous affine
map. Suppose that $(\af, \gamma)$ is  compatible. Then there is a unital
\hm\, $\phi: C_1\to B$ such that
$$
[\phi]=\alpha
\tand
\phi_T=\gamma.$$
In particular, $\phi$ is a monomorphism.
\end{cor}

\begin{proof}
To apply \ref{CorCinj}, one needs a map $\lambda.$
Note that $J_c(K_1(C_1))$ is isomorphic to $K_1(C_1)$ which is  finitely generated.
Let $J_c^{(1)}: K_1(B)\to U(B)/CU(B)$ be the splitting map defined in Definition 2.16 of \cite{GLN-I},
Define $\lambda=J_c^{(1)}\circ \af|_{K_1(C_1)}\circ \pi|_{J_c(K_1(C_1))},$
where $\pi: U(M_2(C_1))/CU(M_2(C_1))\to K_1(C_1)$ is the quotient map (note that $C$ has stable rank one and $C_1=C\otimes C(\T)$ has stable rank two). Then
$(\af, \lambda, \gamma)$ is compatible. The corollary then follows from the previous one.
\end{proof}

\begin{lem}\label{L85}
Let $B\in {\cal B}_1$
be an amenable \CA\, which satisfies the UCT, let  $A_1\in {\cal B}_0,$ let $C=B\otimes U_1,$ and let $A=A_1\otimes U_2,$ where $U_1$ and $U_2$ are UHF-algebras of infinite type. Suppose that $\kappa\in KL_e(C,A)^{++},$
$\gamma: T(A)\to T(C)$ is a continuous affine map and $\af: U(C)/CU(C)\to
U(A)/CU(A)$ is a continuous \hm\, for which $\gamma,\, \af,$ and $\kappa$ are compatible. Then there exists a unital monomorphism $\phi: C\to A$ such that
\begin{enumerate}
\item $[\phi]=\kappa$ in $KL_e(C,A)^{++}$,
\item $\phi_T=\gamma$ and $\phi^{\ddag}=\af.$
\end{enumerate}
\end{lem}
\begin{proof}
The proof follows the same lines as that of Lemma 8.5 of \cite{Lnclasn}.
By Theorem 9.11  of \cite{GLN-I}, every \CA\, $B\in {\cal B}_1$ has weakly unperforated $K_0(A).$
Then, by Corollary 19.3 of \cite{GLN-I},  $B\otimes U_1\in {\cal B}_0.$
By the classification theorem
(Theorem 21.9  and Theorem 14.10 of \cite{GLN-I}), 
one can write
$$C=\varinjlim(C_n, \phi_{n, n+1})$$
where $C_n$ is a direct sum of \CA s in $\mathcal C_0$ or {{in}} $\mathbf H$. Let $\kappa_n=\kappa\circ[\phi_{n, \infty}],$ $\alpha_n=\alpha\circ\phi_{n, \infty}^\ddag$, and $\gamma_n=(\phi_{n, \infty})_T\circ\gamma$. Write $C_n=C_n^1\oplus C_n^2$ with $C_n^1\in {\bf H}$ and $C_n^2\in {\cal C}_0$. By Corollary \ref{istBotC++} applying to $C_n^1$ and Corollary \ref{CorCinj} applying to to $C_n^2$,
there are unital monomorphisms
$\psi_n: C_n\to A$ such that
$$
[\psi_n]=\kappa_n \quad \psi_n^{\ddag}=\alpha_n,\quad\textrm{and}\quad (\psi_n)_T=\gamma_n.$$
(Note that $K_1(C_n^2)=0$, Consequently, $(\psi_n|_{C_n^2})_T=(\imath_{C_n^2,C_n})_T\circ \gamma_n$ implies $(\psi_n|_{C_n^2})^{\ddag}=\alpha_n|_{U(C_n^2)/CU(C_n^2)}$.)
In particular, the sequence of monomorphisms  $\psi_n$ satisfies
$$[\psi_{n+1}\circ\phi_{n, n+1}]=[\psi_n],\quad \psi_{n+1}^\ddag\circ\phi_{n, n+1}=\psi_{n}^\ddag,
\quad\textrm{and}\quad
(\psi_{n+1}\circ\phi_{n, n+1})_T=(\psi_n)_T.$$

Let $\mathcal F_n\subset C_n$ be a finite subset such that $\phi_{n, n+1}(\mathcal F_n)\subset \mathcal F_{n+1}$ and $\bigcup\phi_{n, \infty}(\mathcal F_n)$ is dense in $C$. Applying Theorem 12.7 of \cite{GLN-I} 
with $\Delta(h)=\inf\{\gamma(\tau)(\phi_{n, \infty}(h)): \tau\in T(A)\}$, $h\in C_n^+\setminus \{0\}$, we have
a sequence of unitaries $u_n\in A$ such that
$$\mathrm{Ad} u_{n+1}\circ\psi_{n+1}\circ\phi_{n, n+1}\approx_{1/2^n}\mathrm{Ad} u_n\circ \psi_{n}\quad \textrm{on $\mathcal F_n$}.$$

The  maps $\{\mathrm{Ad} u_n\circ \psi_{n}: n=1, 2, ...\}$ then converge to a unital homomorphism $\phi: C\to A$ which satisfies the lemma.
\end{proof}

\begin{rem}\label{Remark202011-1}
In the first few lines of the proof of Lemma \ref{L85},  we recall
that, if $B\in {\cal B}_1$ is a unital separable simple \CA\, with the UCT.
Then $C:=B\otimes U_1$ (for any infinite dimensional UHF-algebra $U_1$)
is an inductive limit of \CA s $C_n,$ where $C_n$ is a finite direct sum of \CA s in ${\cal C}_0$ or in ${\bf H}.$
This important fact  which proved in the first part of this research (\cite{GLN-I}) will be used
frequently in the rest of the paper. 

\end{rem}

\begin{thm}\label{Next2014/09}
Let $X$ be a finite CW complex and let $C=PM_n(C(X))P,$ where
$n\ge 1$ is an integer and $P\in M_n(C(X))$ is a projection.
Let $A_1\in {\cal B}_0$ and let $A=A_1\otimes U$ for a UHF-algebra {{$U$}} of infinite type.
Suppose $\af\in KL_e(C, A)^{++},$
$\lambda: U_{\infty}(C)/CU_{\infty}(C)\to  U(A)/CU(A)$ is a continuous \hm, and $\gamma: T(A)\to T_f(C)$ is a
continuous affine map such that $(\af, \lambda, \gamma)$ is compatible. Then there exists a unital \hm\,
$h: C\to A$ such that
\beq\label{Next1409-1}
[h]=\af,\,\,\, h^{\ddag}=\lambda \andeqn
h_T=\gamma.
\eneq
\end{thm}

\begin{proof}
The proof is similar to that of 6.6 of \cite{Lnclasn}.
To simplify the notation, without loss of generality, 
let us assume
that $X$ is connected. Furthermore, a standard argument shows that the general case
can be reduced to the case $C=C(X).$
We may assume that $U(M_N(C))/U_0(M_N(C))=K_1(C)$ for some integer $N$ (see \cite{Rf}).
Therefore, in this case, $$U(M_N(C))/CU(M_N(C))=U_{\infty}(C)/CU_{\infty}(C).$$
Write $K_1(C)=G_1\oplus {\rm Tor}(K_1(C)),$  where $G_1$ is the torsion free part of $K_1(C).$
Fix a point $\xi\in X$ and  let $C_0=C_0(X\setminus \{\xi\}).$  Note that $C_0$ is an ideal of $C$ and
$C/C_0{{\cong \C}}$. Write
\beq\label{Next1409-2}
K_0(C)=\Z\cdot [1_C]\oplus K_0(C_0).
\eneq
Let $B\in {\cal B}_0$ be a unital separable simple \CA\, as constructed in Corollary 14.14 of \cite{GLN-I} 
such that
\beq\label{Next1409-3}
(K_0(B), K_0(B)_+, [1_B], T(B), r_B)=(K_0({{A}}), K_0({{A}})_+, [1_{{A}}], T({{A}}), r_{{A}})
\eneq
and $K_1(B)=G_1{{\oplus}} {\rm Tor}(K_1({{A}})).$
Put
\beq\label{Next1510-1}
\Delta(\hat{g})=\inf\{\gamma(\tau)(g): \tau\in T(A)\}.
\eneq
For each $g\in C_+\setminus\{0\},$ since $\gamma(\tau)\in T_f(C),$ $\gamma$ is continuous and $T(A)$ is compact,
$\Delta(\hat{g})>0.$

Let $\ep>0,$ ${\cal F}\subset C$ be a finite subset,   $1>\sigma_1, \sigma_2>0,$
${\cal H}\subset C_{s.a.}$ be a finite
subset, and  ${\cal U}\subset U(M_N(C))/CU(M_N(C))$ be a finite subset.
\Wlog, we may assume that ${\cal U}={\cal U}_0\cup {\cal U}_1,$ where
${\cal U}_0\subset U_0(M_N(C))/CU(M_N(C))$ and ${\cal U}_1\subset  J_c(K_1(C))\subset U(M_N(C))/CU(M_N(C)).$

For each $u\in {\cal U}_0,$ write $u=\prod_{j=1}^{n(u)}\exp(\sqrt{-1} a_i(u)),$ where
$a_i(u)\in M_N(C)_{s.a.}.$  Write
\beq\label{Next1409-4}
a_i(u)=(a_i^{(k,j)}(u))_{N\times N},\,\,\,i=1,2,...,n(u).
\eneq
Write
\beq\label{Next1409-5}
c_{i,k,j}(u)={a_i^{(k,j)}(u)+(a_i^{(k,j)})^*\over{2}}\andeqn d_{i,k,j}={a_i^{(k,j)}(u)-(a_i^{(k,j)})^*\over{2i}}.
\eneq
Put
\beq\label{Next1409-6}
M=\max\{\|c\|, \|c_{i,k,j}(u)\|,\|d_{i,k,j}(u)\|: c\in {\cal H}, u\in {\cal U}_0\}.
\eneq
Choose a non-zero projection $e\in {{B}}$ such that
\begin{equation*}
\tau(e)<{\min\{\sigma_1,\sigma_2\}\over{16N^2(M+1)\max\{n(u): u\in {\cal U}_0\}}}\rforal \tau\in T({{B}}).
\end{equation*}
Let $B_2=(1-e){{B}}(1-e).$

In what follows we will use the identification (\ref{Next1409-3}).
Define $\kappa_0\in {\rm Hom}(K_0(C), K_0(B_2))$ as follows.
Define $\kappa_0(m[1_C])=m[1-e]$ for $m\in \Z$ and
$\kappa_0|_{K_0(C_0)}=\af|_{K_0(C_0)}.$ { {Note that $K_1(B)=G_1\oplus {\rm Tor}(K_1(A))$ and that $\af$ induces a map $\af|_{{\rm Tor}(K_1(C))}: {\rm Tor}(K_1(C)) \to {\rm Tor}(K_1(A))$. Using the given decomposition $K_1(C)=G_1\oplus
{\rm Tor}(K_1(C)$, we can define $\kappa_1: K_1(C) \to K_1(B)$ by $\kappa_1|_{G_1}={\rm id}$ and $\kappa_1|_{{\rm Tor}(K_1(C))}=[\af]|_{{\rm Tor}(K_1(C))}$.}}

 By the Universal Coefficient Theorem, there is
$\kappa\in KL(C, B_2)$ which gives rise to the two \hm s $\kappa_0, \kappa_1$ above.
Note that $\kappa\in KL_e(C,B_2)^{++},$  since $K_0(C_0)= {{\rm ker}\rho_C(K_0(C))}$.
Choose
$$
{\cal H}_1={\cal H}\cup\{c_{i,k,j}(u), d_{i,k,j}(u): u\in {\cal U}_0\}.
$$
Every tracial state $\tau'$ of $B_2$ has the form
$\tau'(b)=\tau(b)/\tau(1-e)$ for all $b\in B_2$ { { for some $\tau\in T(B)$}}.  Let $\gamma': T(B_2)\to
T({{C}})$ be defined { {as follows. For $\tau'\in T(B_2)$ as above, define $\gamma'(\tau')(f)=\gamma(\tau)(f)$ for $f\in C$.}}

It follows from \ref{istBotC} that
there exists a sequence of unital \cp s
$h_n: C\to B_2$ such that
\begin{eqnarray*}
\lim_{n\to\infty}\|h_n(ab)-h_n(a{{)h_n(}}b)\|=0\rforal a,\, b\in  C,\\
{[}h_n{]}=\kappa \,\,\,\,\,\, {\text (-} K_*(C)\,\,\text{is\,\,\, finitely \,\, generated\,)}, \andeqn\\
\lim_{n\to\infty}\max\{ |\tau\circ h_n(c)-\gamma'(\tau)(c)|: \tau\in T(B_2)\}=0{{ \rforal c\in C}}.
\end{eqnarray*}
Here we may assume that $[h_n]$ is well defined for all $n$ and
\beq\label{Next1409-8}
|\tau\circ h_n(c)-\gamma(\tau)(c)|<{\min\{\sigma_1, \sigma_2\}\over{8N^2}},\,\,\, n=1,2,...
\eneq
for all $c\in {\cal H}_1$ and for all $\tau\in T(B_2).$
Choose $\theta\in KL({{B}}, {{A}})$ such that
it gives the identification of (\ref{Next1409-3}), and, $\theta|_{G_1}=\af|_{G_1}$
and  $\theta|_{{\rm Tor}(K_1({{A}}))}={\rm id}_{{\rm Tor}(K_1({{A}}))}.$ {{ Let $e'\in A $ be a projection such that $[e']\in K_0(A)$ corresponds to $[e]\in K_0(B)$ under the identification \eqref{Next1409-3}.}}
Let $\bt=\af-\theta\circ \kappa.$ Then
\beq\label{Next1409-9}
\bt([1_C])={ {[}}e{{']}},\,\,\,\bt_{K_0(C_0)}=0,\andeqn \bt_{K_1(C)}=0.
\eneq
Then $\bt\in KL_e(C, {{e'Ae'}}).$ It follows from \ref{istBotC} that there exists a sequence of unital \cp s
$\phi_{0,n}: C\to {{e'Ae'}}$ such that
\beq\label{Next1409-10}
\lim_{n\to\infty}\|\phi_{0,n}(ab)-\phi_{0,n}(a)\phi_{0,n}(b)\|=0\andeqn [\phi_{0,n}]=\bt.
\eneq
Note that, for each $u\in U(M_N(C))$ with ${\bar u}\in {\cal U}_0,$
\beq\label{Next1409-11}
D_{C}(u)=\overline{\sum_{i=1}^{n(u)}\widehat{a_j(u)}},
\eneq
where $\widehat{c}(\tau)=\tau(c)$ for all $c\in C_{s.a.}$ and $\tau\in T(C).$
Since $\kappa$ and $\lambda$ are compatible,  we compute, for ${\bar u}\in {\cal U}_0,$
\beq\label{Next1409-12}
{\rm dist}((h_n)^{\ddag}({\bar u}), \lambda({\bar u}))<\sigma_2/8.
\eneq
Fix a pair of large integers $n,m,$ and
define $\chi_{n,m}: J_c(G_1) (\subset U(C)/CU(C))\to \Aff(T(A))/\overline{\rho_{A}(K_0(A))}$ to be
\beq\label{Next1409-13}
\lambda|_{J_c(G_1)}-(h_n)^{\ddag}|_{J_c(G_1)}-\phi_{0,m}^{\ddag}|_{J_c(G_1)}.
\eneq
We may may  also view  $J_c(G_1)$ as subgroup of $J_c(K_1({{B}})){{=J_c(K_1(B_2))}}.$ Write $J_c(K_1(B))=J_c(G_1) \oplus J_c({\rm Tor} (K_1(B_2)))$ and  define
$\chi_{n,m}$ to be zero on ${\rm Tor}(K_1(B_2))$, we obtain a \hm\,
$\chi_{n,m}: J_c(K_1(B_2))\to \Aff({{T(A)}})/\overline{\rho_{{{A}}}(K_0({{A}}))}.$
It follows from Lemma \ref{L85} that there is a unital \hm\, $\psi: B_2\to (1-e{{'}}){{A}}(1-e{{'}})$ such that
\beq\label{Next1409-14}
[\psi]=\theta,\,\, \psi_T={\rm id}_{T({{A}})}\andeqn\\
\psi^{\ddag}|_{J_c(K_1(B_2))}=\chi_{n,m}|_{J_c(K_1(B_2))}+J_c\circ \theta|_{K_1(B_2)},
\eneq
where we identify $K_1(B_2)$ with $K_1({{B}}).$
By (\ref{Next1409-14}),
\beq\label{Next1409-15}
\psi^{\ddag}|_{\Aff(T(B_2))/\overline{\rho_{B_2}(K_0(B_2))}}={\rm id}.
\eneq
Define $L(c)=\phi_{0,m}(c)\oplus \psi\circ h_n(c)$ for all $c\in C.$
It follows, on choosing sufficiently large $m$ and $n,$
that $L$ is $\ep$-${\cal F}$-multiplicative,
\beq\label{Next1409-16}
&&[L]=\af,\\
&&\max\{|\tau\circ \psi(f)-\gamma(\tau)(f)|: \tau\in T({{A}})\}<\sigma_1\rforal f\in {\cal H},\andeqn\\
&&{\rm dist}(L^{\ddag}({\bar u}), \lambda({\bar u}))<\sigma_2.
\eneq
This implies that that there is a sequence of \morp s $\psi_n: C\to {{A}}$ such that
\beq\label{Next1409-17}
&&\lim_{n\to\infty}\|\psi_n(ab)-\psi_n(a)\psi_n(b)\|=0\rforal a,b\in C,\\
&&{[}\psi_n{]}=\af,\\
&&\lim_{n\to\infty}\max\{|\tau\circ \psi_n(c)-\gamma(\tau)(c)|: \tau\in T(A_1)\}=0\rforal c\in C_{s.a.},\andeqn\\
&&\lim_{n\to\infty}{\rm dist}(\psi_n^{\ddag}({\bar u}), \lambda({\bar u}))=0\rforal u\in U(M_N(C))/CU(M_N(C)).
\eneq
Finally, applying Theorem 12.7 of \cite{GLN-I}, 
as in the proof of \ref{istBotC++},
using $\Delta/2$ above, we obtain a unital \hm\,
$h: C\to {{A}}$ such that
\beq\label{Next1409-18}
{[}h{]}=\af,\,\,\, h_T=\gamma,\andeqn h^{\ddag}=\lambda,
\eneq
as desired.
\end{proof}

\begin{thm}\label{BB-exi}
Let  $C\in \mathcal C_0$
and let
 $G= K_0(C)$. Write
 ${G}=\Z^k$ with
 $\Z^k$ generated by
$$
\{x_{1}=[p_1]-[q_1], x_2=[p_2]-[q_2], ..., x_{k}=[p_{k}]-[q_{k}]\},
$$
where $p_i, q_i\in M_n(C)$ (for some integer $n\ge 1$) are projections,
$i=1,...,k$.

Let $A$ be a simple C*-algebra in $\mathcal B_0$, and let $B=A\otimes U$ for a UHF algebra $U$ of infinite type. Suppose that $\phi: C\to B$ is a monomorphism. Then, for any finite subsets  $\mathcal F\subset C$ and $\mathcal P\subset \underline{K}(C)$, any $\ep>0$ and $\sigma>0$, and any homomorphism $$\Gamma: \Z^k \to U_0({{B}})/CU({{B}}),$$ there is a unitary $w\in B$ such that
\begin{enumerate}
\item $\|[\phi(f), w]\|<\ep$, for any $f\in\mathcal F$,
\item $\mathrm{Bott}(\phi, w)|_{\mathcal P}=0$, and
\item $\mathrm{dist}(
\overline{\langle(({\mathbf 1}_n-\phi(p_i))+\phi(p_i)\tilde{w})(({\mathbf 1}_n-\phi(q_i))+\phi(q_i)\tilde{w}^*)\rangle},
 \Gamma(x_i)))<\sigma$, for any $1\leq i\leq k,$
{ where $\tilde{w}=\mathrm{diag}(\overbrace{w, ..., w}^n)$.}
\end{enumerate}
\end{thm}

\begin{proof}{{ Write $B=\lim_{{n\to \infty}}(A\otimes M_{r_n},\imath_{n,n+1})$. Using the fact that $C$ is
{ semiprojective} (see \cite{ELP1}), one can construct a sequence of homomorphisms $\phi_n: C\to A\otimes M_{r_n}$ such that $\imath_n\circ\phi_n(c) \to \phi(c)$ for all $c\in C$. Without loss of generality, we may assume $\phi=\imath\circ \phi_1$ for a homomorphism $\phi_1: C\to A$ (replacing $A$ by $A\otimes M_{r_n}$), where $\imath: A \to A\otimes U=B$ is the standard inclusion.}}

We may assume that $||f||\leq 1$ for any $f\in\mathcal F$.

For any non-zero positive element $h\in C$ with norm at most $1$,  define $$\Delta(h)=\inf\{\tau(\phi(h));\ \tau\in T(B)\}.$$ Since $B$ is simple, one has that $\Delta(h)\in (0, 1)$.

Let $\mathcal H_1\subset C_+^{\bf 1}\setminus\{0\}$, $\mathcal G\subset C$, $\delta>0$, $\mathcal P\subset\underline{K}(C)$, $\mathcal H_2\subset C_{s.a.},$ and $\gamma_1>0$ be the finite subsets and constants of Theorem 12.7 of \cite{GLN-I} 
 with respect to $C,$
$\mathcal F$, $\epsilon/2,$ and $\Delta/2$ (since $K_1(C)=\{0\}$, one does not need $\mathcal U$ and $\gamma_2$).

Note that $B={{A}}\otimes U.$ Pick a unitary $z\in U$
with ${\rm sp}(u)=\T$ and consider the homomorphism $\phi': C\otimes C(\mathbb T)\to B=A\otimes U$ defined by $$a\otimes f \mapsto \phi_1(a)\otimes f(z).$$
(Recall that $\phi(a)=\phi_1(a)\otimes 1_U$.)
Set
$$\gamma=(\phi')_T: T(B)\to T_{\mathrm{f}}(C\otimes C(\mathbb{T})).$$ Also define $$\alpha:=[\phi'] \in KK(C\otimes C(\mathbb T), B).$$

Note that $K_1(C\otimes C(\mathbb T))=K_0(C)=\mathbb Z^k$. Identifying $J_c(K_1(C\otimes C(\mathbb T)))$ with $\mathbb Z^k$, define a map $\lambda: J_c(K_1(U(C\otimes C(\mathbb T))))\to U_{{0}}(B)/CU(B)$
by $\lambda(a)=\Gamma(a)$ for any $a\in\mathbb Z^k$.

Set
$$\mathcal U=\{(1_n-p_i+p_i\tilde{z'})(1_n-q_i+q_i\tilde{z'}^*):\ i=1, ..., k\}\subset J_c(U(C\otimes C(\mathbb T))),$$ where $z'$ is the standard generator of $C(\mathbb T)$, and set
$$\delta=\min\{\Delta(h)/4:\ h\in\mathcal H_1\}.$$
Applying Lemma \ref{CtimesText}, one obtains a $\mathcal F$-$\epsilon/4$-multiplicative map $\Phi: C\otimes C(\mathbb T)\to B$ such that
\beq
&&[\Phi]=\af,\,\,\,{\rm dist}(\Phi^{\ddag}(x), \lambda(x))<\sigma \rforal x\in \overline{\cal U},
\andeqn\\
\label{eq-012-02}
&&|\tau\circ \Phi(h\otimes 1)-\gamma(\tau)(h\otimes 1)|<\min\{\gamma_1,\delta \} \rforal  h\in {\mathcal H_1\cup\mathcal H_2}.
\eneq

Let $\psi$ denote the restriction of $\Phi$ to $C\otimes 1$. Then one has
$$[\psi]|_{\mathcal P}=[\phi]|_{\mathcal P}.$$
By \eqref{eq-012-02}, one has that, for any $h\in\mathcal H_1$,
$$\tau(\psi(h))>\gamma(\tau)(h)-\delta=\tau(\phi'(h\otimes 1))-\delta=\tau(\phi(h))-\delta>\Delta(h)/2,$$
and it is also clear that $$\tau(\phi(h))>\Delta(h)/2\rforal h\in \mathcal H_1.$$
Moreover, for any $h\in\mathcal H_2$, one has
\begin{eqnarray*}
|\tau\circ\psi(h)-\tau\circ\phi(h)|&=& |\tau\circ\Phi(h\otimes 1)-\tau\circ\phi'(h\otimes 1)|\\
&=&  |\tau\circ\Phi(h\otimes 1)-\gamma(\tau)(h\otimes 1)|
< \gamma_1.
\end{eqnarray*}
Therefore, by Theorem 12.7 of \cite{GLN-I}, 
there is a unitary $W\in B$ such that
$$||W^*\psi(f)W-\phi(f)||<\epsilon/2\rforal f\in\mathcal F.$$ Then the element $$w=W^*\Phi(1\otimes z')W$$ is the desired unitary.
\end{proof}

\begin{thm}\label{BB-exi+}
Let  $C$ be a unital  \CA\, which is a finite direct sum of \CA s in $\mathcal C_0$  and \CA s of
the form $PM_n(C(X))P,$  where $X$ is a finite CW complex and $P$ is a projection, and let
 $G= K_0(C)$. Write
 ${G}=\Z^k\bigoplus \mathrm{Tor}(G)$ with a basis for
 $\Z^k$ the set
$$
\{x_{1}=[p_1]-[q_1], x_2=[p_2]-[q_2], ..., x_{k}=[p_{k}]-[q_{k}]\},
$$
where $p_i, q_i\in M_n(C)$ (for some integer $n\ge 1$) are projections,
$i=1,...,k$.

Let $A$ be a simple C*-algebra in the class $\mathcal B_0$, and let $B=A\otimes U$ for a UHF algebra $U$ of infinite type. Suppose that $\phi: C\to B$ is a monomorphism. Then, for any finite subsets  $\mathcal F\subset C$ and $\mathcal P\subset \underline{K}(C)$, any $\ep>0$ and $\sigma>0$, and any homomorphism $$\Gamma: \Z^k \to U_0(M_n(B))/CU(M_n(B)),$$ there is a unitary $w\in B$ such that
\begin{enumerate}
\item $\|[\phi(f), w]\|<\ep$, for any $f\in\mathcal F$,
\item $\mathrm{Bott}(\phi, w)|_{\mathcal P}=0$, and
\item $\mathrm{dist}(
\overline{\langle(({\mathbf 1}_n-\phi(p_i))+\phi(p_i)\tilde{w})(({\mathbf 1}_n-\phi(q_i))+\phi(q_i)\tilde{w}^*)\rangle},
 \Gamma(x_i)))<\sigma$, for any $1\leq i\leq k,$
{ where $\tilde{w}=\mathrm{diag}(\overbrace{w, ..., w}^n)$.}
\end{enumerate}
\end{thm}

\begin{proof}
By Theorem \ref{BB-exi}, it suffices to prove the case that $C=PM_n(C(X))P,$   where
$X$ is a finite CW complex, $n\ge 1$ is an integer, and $P\in M_n(C(X))$ is a projection.
The proof follows the same lines
as that of Theorem \ref{BB-exi} but 
using Lemma  \ref{Next2014/09}
instead of Lemma \ref{CtimesText}.
\end{proof}

\section{A pair of almost commuting unitaries}



\begin{lem}\label{orderK0}
Let $C\in {\cal C}.$ There exists a constant $M_C>0$ satisfying the following condition:
For any $\ep>0,$ any $x\in K_0(C),$ and any $n\ge M_C/\ep,$ if
\beq\label{orderK0-0}
|\rho_C(x)(\tau)| <\ep\tforal \tau\in T(C\otimes M_n),
\eneq
then there are mutually inequivalent
and mutually orthogonal minimal projections $p_1, p_2,...,p_{k_1}$ and $q_1,q_2,...,q_{k_2}$ in
$C\otimes M_n$  and positive integers $l_1, l_2,...,l_{k_1}, m_1, m_2,...,m_{k_2}$ such that
\beq\label{orddeK0-1}
x=[\sum_{i=1}^{k_1}l_ip_i]-[\sum_{j=1}^{k_2}m_jq_{{j}}]\tand\\
\tau(\sum_{i=1}^{k_1}l_ip_i)<4\ep\andeqn \tau(\sum_{j=1}^{k_2}m_jq_{{j}})<4\ep
\eneq
for all $\tau\in T(C\otimes M_n).$

\end{lem}

\begin{proof}
Let $C=C(F_1, F_2, \phi_1, \phi_2)$ and $F_1=\bigoplus_{i=1}^l M_{r(i)}.$
By Theorem 3.15 of \cite{GLN-I}, 
there are only finitely many mutually inequivalent minimal projections in $C\otimes {\cal K}$. We can choose $N(C)>0$ such that  this set of mutually inequivalent projections is sitting in $M_{N(C)}(C)$, orthogonally. Then every projection in $C\otimes {\cal K}$ is equivalent to  a finite direct sum of   projections from this set of finitely many mutually inequivalent minimal projections (some of them may repeat in the direct sum).
We also assume that, as in Definition 3.1 of \cite{GLN-I}, 
$C$ is minimal.
Let
$$
M_C=N(C)+2(r(1)\cdot r(2)\cdots r(l))
$$
Suppose that $n\ge M_{C}/\ep.$
With the canonical embedding of $K_0(C)$ into $K_0(F_1)\cong\mathbb Z^l$, write
\beq\label{odderK0-2}
x=
\begin{pmatrix}
x_1\\
x_2\\
           \vdots \\
           x_l
           \end{pmatrix}\in\mathbb Z^l.
\eneq
By (\ref{orderK0-0}),
for any irreducible representation $\pi$ of $C$ and any tracial state $t$ on $M_n(\pi(C)),$
\beq\label{orderK0-3}
|t\circ \pi(x)|<\ep.
\eneq
It follows that
\beq\label{orderK0-4}
|x_s|/r(s)n<\ep,\,\,\, s=1,2,...,l.
\eneq
Let
\beq\label{orderK0-5}
T=\max\{|x_s|/r(s): 1\le s\le l\}.
\eneq
Define
\beq\label{orderK0-6}
y=x+T\begin{pmatrix}
                   r(1)\\
                   r(2)\\
           \vdots \\
           r(l)
         \end{pmatrix} \andeqn z=T\begin{pmatrix}
                   r(1)\\
                   r(2)\\
           \vdots \\
           r(l)
         \end{pmatrix}.
         \eneq
It is clear that $z\in K_0(C)_+$ (see Proposition 3.5 of \cite{GLN-I}). 
It follows that $y\in K_0(C).$  One also computes that
$y\in K_0(C)_+.$ It follows that there are projections $p, \, q\in M_L(C)$ for some integer $L\ge 1$ such that $[p]=y$ and $[q]=z.$
Moreover, $x=[p]-[q].$
One also computes that
\beq\label{orderK0-9}
\tau(q)<T/n<\ep\tforal \tau\in T(C\otimes M_n).
\eneq
One also has
\beq\label{orderK0-10}
\tau(p)<2\ep \tforal \tau\in T(C\otimes M_n).
\eneq

There are {{two sets of}} mutually inequivalent and mutually orthogonal minimal projections ${{\{}}p_1, p_2,...,p_{k_1}{{\}}}$ and $
{{\{}}q_1,q_2,...,q_{k_2}{{\}}}$ in $C\otimes M_n$ (since $n>N({{C}})$) such that
\beq\label{orderK0-11}
[p]=\sum_{i=1}^{k_1}l_i[p_i]\andeqn [q]=\sum_{j=1}^{k_2}m_j[q_j].
\eneq
Therefore
\beq\label{orderK0-12}
x=\sum_{i=1}^{k_1}l_i[p_i]-\sum_{j=1}^{k_2}m_j[q_j].
\eneq

\end{proof}

\begin{lem}\label{Vpair1}
Let $C\in \cal C.$  There is an integer $M_C>0$ satisfying the following condition: For any $\ep>0$
and for any
$x\in K_0(C)$ with
$$
|\tau(\rho_C(x))|<\ep/24\pi
$$
for all  $\tau\in T(C\otimes M_n),$ where $n\ge 2M_C\pi/\ep,$
there exists a pair of unitaries $u$ and $v\in C\otimes M_n$ such that
\beq\label{Vpair1-1}
\|uv-vu\|<\ep \tand \tau({\rm bott}_1(u,\, v))=\tau(x).
\eneq
\end{lem}

\begin{proof}
We may assume that $C$ is minimal.
Applying Lemma \ref{orderK0},  we obtain
mutually orthogonal and mutually inequivalent  minimal projections $p_1, p_2,...,p_{k_1}, q_1, q_2,...,q_{k_2}\in C\otimes M_n$ such that
$$
\sum_{i=1}^{k_1}l_i[p_i]-\sum_{j=1}^{k_2}m_j[q_j]=x,
$$
where $l_1,l_2,...,l_{k_1},$ $m_1, m_2,...,m_{k_2}$ are positive integers.
Moreover,
\beq\label{Vpair1-2}
\sum_{i=1}^{k_1}l_i\tau(p_i)<\ep/6\pi\andeqn \sum_{j=1}^{k_2}m_j\tau(q_j)<\ep/6\pi
\eneq
for all $\tau\in T(C\otimes M_n).$
Choose $N\le n$ such that $N=[2\pi/\ep]+1.$
By (\ref{Vpair1-2}),
\beq\label{Vpair1-3}
\sum_{i=1}^{k_1}Nl_i\tau(p_i)+\sum_{j=1}^{k_2}Nm_j\tau(q_j)<1/2\tforal \tau\in T(C\otimes M_n).
\eneq
It follows  that there are mutually orthogonal  projections
$d_{i,k}, d_{j,k}'{{\in C\otimes M_n}}, $ $k=1,2,...,N,$ $i=1,2,...,k_1,$ and $j=1,2,...,k_2$ such that
\beq\label{Vpair1-4}
[d_{i,k}]=l_i[p_i]\andeqn [d_{j,k}']=m_j[q_j], \,\,\,i=1,2,...,k_1, \, j=1,2,...,k_2
\eneq
and $k=1,2,...,N.$
Let $D_i=\sum_{k=1}^Nd_{i,k}\andeqn D_j'=\sum_{k=1}^N d_{j,k}',$ $i=1,2,...,k_1$ and $j=1,2,...,k_2.$
There are  partial isometries $s_{i,k},\, s_{j,k}'\in C\otimes M_n$ such that
\beq\label{July28-2}
s_{i,k}^*d_{i,k}s_{i,k}=d_{i,k+1},\,\,\, (s_{j,k}')^*d_{j,k}'s_{j,k}'=d_{j,k{{+1}}}',\,\,\, k=1,2,...,N-1,\\
s_{i, N}^*d_{i, N}s_{i,N}=d_{i,1},\andeqn (s_{j,N}')^* d_{j,N}'s_{j,N}'=d_{j,1}',
\eneq
 $i=1,2,...,k_1$ and  $j=1,2,...,k_2.$
Thus, we obtain unitaries $u_i\in D_i(C\otimes M_n)D_i$ and $u_j'=D_j'(C\otimes M_n)D_j'$ such that
\beq\label{Vpair1-5}
u_i^*d_{i,k}u_i=d_{i,k+1},\,\,\,u_i^*d_{i,N}u_i=d_{i,1},\,\,\,
(u_j')^*d_{j,k}'u_j'=d_{j,k+1}',\andeqn (u_j')^*d_{j,N}'u_j'=d_{j,1}',
\eneq
$i=1,2,...,k_1,$ $j=1,2,...,k_2.$
Define
$$
v_i=\sum_{k=1}^N e^{\sqrt{-1}(2k\pi/N)}d_{i,k}\andeqn v_j'=\sum_{k=1}^N e^{\sqrt{-1}(2k\pi/N)}d_{j,k}'.
$$
We compute that
\beq\label{Vpair1-6}
\|u_iv_i-v_iu_i\|<\ep\andeqn \|u_j'v_j'-v_j'u_j'\|<\ep,\\
{1\over{2\pi \sqrt{-1}}}\tau(\log v_iu_iv_i^*u_i^*)=l_i\tau(p_i),\andeqn\\
{1\over{2\pi \sqrt{-1}}}\tau(\log v_j'u_j'(v_j')^*(u_j')^*)=m_j\tau(q_j),
\eneq
for $\tau\in T(C\otimes M_n),$ $i=1,2,...,k_1$ and $j=1,2,...,k_2.$
Now define
\beq\label{Vpair1-7}
u=\sum_{i=1}^{k_1}u_i+\sum_{j=1}^{k_2}u_j'+(1_{C\otimes M_n}-\sum_{i=1}^{k_1}D_i-\sum_{j=1}^{k_2}D_j')\andeqn\\
v=\sum_{i=1}^{k_1}v_i+\sum_{j=1}^{k_2}(v_j')^*+(1_{C\otimes M_n}-\sum_{i=1}^{k_1}D_i-\sum_{j=1}^{k_2}D_j').
\eneq
We then compute that
\beq\label{Vpair1-8}
\hspace{-0.4in}\tau({\rm bott}_1(u, v))=\sum_{i=1}^{k_1}{1\over{2\pi \sqrt{-1}}}\tau(\log (v_iu_iv_i^*u_i^*))-\sum_{j=1}^{k_2}{1\over{2\pi \sqrt{-1}}}\tau(\log v_j'u_j'(v_j')^*(u_j')^*)\\
=\sum_{i=1}^{k_1} l_i\tau(p_i) -\sum_{j=1}^{k_2}m_j\tau(q_j)=\tau(x)
\eneq
for all $\tau\in T(C\otimes M_n).$
\end{proof}

\begin{lem}\label{Vpair2}
Let $\ep>0.$ There exists $\sigma>0$ satisfying the following condition:
Let $A=A_1\otimes U,$ where $U$ is a UHF-algebra of infinite type  and $A_1\in {\cal B}_0$, let $u\in U(A)$ be a unitary with $sp(u)=\T,$ and let
$x\in K_0(A)$ with
 $|\tau(\rho_A(x))|<\sigma$ for all $\tau\in T(A)$ and $y\in K_1(A) .$ Then there exists a unitary
$v\in U(A)$ such that
\beq\label{Vpair2-1}
\|uv-vu\|<\ep,\,\,\,{\rm bott}_1(u,v)=x, \tand [v]=y.
\eneq
\end{lem}

\begin{proof}
Let $\phi_0: C(\T)\to A$ be the unital monomorphism
defined by $\phi_0(f)=f(u)$ for all $f\in C(\T).$
Let $\Delta_0: C(\T)_+^{q,{\bf 1}}\setminus \{0\}\to (0,1)$ be defined by
$\Delta_0(\hat{f})=\inf \{\tau(\phi_0(f)): \tau\in T(A)\}.$
Let $\ep>0$ be given.
Choose $0<\ep_1<\ep$ such that
$$
{\rm bott}_1(z_1,z_2)={\rm bott}_1(z_1', z_2')
$$
for any two pairs of unitaries
$z_1, z_2$ and $z_1', z_2'$  satisfying the conditions $\|z_1-z_1'\|<\ep_1$, $\|z_2-z_2'\|<\ep_1$,
$\|z_1z_2-z_2z_1\|<\ep_1$ and
$\|z_1'z_2'-z_2'z_1'\|<\ep_1.$

Let ${\cal H}_1\subset C(\T)_+^{\bf 1}\setminus \{0\}$
be a finite subset, $\gamma_1>0,$ $\gamma_{{2}}>0,$  and ${\cal H}_2\subset C(\T)_{s.a.}$
be a finite subset  as provided by Corollary 12.9 of \cite{GLN-I} 
(for $\ep_1/4$ and $\Delta_0/2$). We may assume that ${\cal H}_2\subset C(\T)^{\bf 1}.$

Let
$$\dt_1=\min\{\gamma_1/16, \gamma_2/16, \min\{\Delta_0(\hat{f}): f\in {\cal H}_1\}/4\}.$$
 Let
$\sigma =\min\{\dt_1/16, (\dt_1/16)(\ep_1/32\pi)\}.$

Let $e\in 1\otimes U \subset A $ be a non-zero projection such that
$\tau(e)<\sigma$ for all $\tau\in T(A).$
Let $B=eAe$ (then $B\cong A\otimes U'$ for some UHF-algebra $U'$).
It follows from Corollary 18.10 of \cite{GLN-I} 
that there is a unital simple \CA\, $C'=\lim_{n\to\infty}(C_n, \psi_n),$ where $C_n\in {\cal C}_0$ and $C=C'\otimes U$ such that
$$
(K_0(C), K_0(C)_+, [1_C], T(C), r_C)=(\rho_A(K_0(A)), (\rho_A(K_0(A)))_+,
\rho_A([e]), T(eAe), r_A).
$$
Moreover, we may assume that all $\psi_n$ are unital.

Now suppose that $x\in K_0(A)$ with $|\tau(\rho_A(x))|<\sigma$
for all $\tau\in T(A)$ and suppose that $y\in K_1(A).$   Let $z=\rho_A(x){{\in K_0(C)}}.$ We identify $z$ with the element in $K_0(C)$ in the  identification above.
We claim that,
there is $n_0\ge 1$ such that there is $x'\in K_0(C_{n_0}\otimes U)$ such that ${{z}}= (\psi_{n_0, \infty})_{*0}(x'){{\in K_0(C)}}$
and $|t(\rho_{C_{n_0}\otimes U})(x')|<\sigma$ for all $t\in T(C_{n_0}\otimes U).$

Otherwise, there is an increasing sequence $n_k,$ $x_k\in K_0(C_{n_k}\otimes U)$ such that
\beq\label{Vpair2-2}
(\psi_{n_k,\infty})_{*0}(x_k)={{z\in K_0(C)}}\andeqn |t_k(\rho_{C_{n_k}\otimes U})(x_k)|\ge \sigma
\eneq
for some $t_k\in T(C_{n_k}\otimes U),$ $k=1,2,....$
Let $L_k: C\to C_{n_k}\otimes U$
be such that
$$
\lim_{n\to\infty}\|\psi_{n,\infty}\circ L_n(c)-c\|=0
$$
for all $c\in \psi_{k,\infty}( C_{n_k}\otimes U),$ $k=1,2,....$ It follows that any
limit point of $t_k\circ L_k$ is a tracial state of $C.$  Let $t_0$ be one  such limit.
Then, by (\ref{Vpair2-2}),
$$
t_0(\rho_C(z))\ge \sigma.
$$
This proves the claim.

Write $U=\lim_{n\to\infty}(M_{r(m)}, \imath_m),$ where $\imath_m: M_{r(m)}\to M_{r(m+1)}$ is a unital
embedding. Repeating the argument above, we obtain $m_0\ge 1$  and $y'\in K_0(C_{n_0}\otimes M_{r(m_0)})=
K_0(C_{n_0})$  such that
$(\imath_{m_0, \infty})_{*0}(y')=x'$ and $|t(\rho_{C_{n_0}}(y'))|< \sigma$ for all $t\in T(C_{n_0}\otimes M_{r(m_0)}).$
Let $M_{C_{N_0}}$ be the constant given by Lemma \ref{Vpair1}. Choose $r(m_1)\ge \max\{{{48}}M_{C_{n_0}}/\sigma, r(m_0)\}$
and let $y''=(\imath_{m_0, m_1})_{*0}(y').$ Then, we compute that
$$
|t(\rho_{C_{n_0}}(y''))|<\sigma \rforal t\in T(C_{n_0}\otimes M_{r(m_1)}).
$$

It follows from \ref{Vpair1} that there exists a pair of unitaries
$u_1', v_1'\in C_{n_0}\otimes M_{r(m_1)}$ such that
\beq\label{Vpair2-3-}
\|u_1'v_1'-v_1'u_1'\|<\ep_1/4\andeqn {\rm bott}_1(u_1', v_1')=y''.
\eneq
Put $u_1=\imath_{m_1, \infty}(u_1')$ and $v_1=\imath_{m_1, \infty}(v_1').$
Then (\ref{Vpair2-3-}) implies that
\beq\label{Vpair2-3}
\|u_1v_1-v_1u_1\|<\ep_1/4\andeqn {\rm bott}_1(u_1, v_1)=x'.
\eneq

Let $h_0: C_{n_0}\otimes U\to M_2(eAe)$ be a  \hm\, as given by Corollary 18.10 of \cite{GLN-I} 
such that
\beq\label{Vpair2-4}
\rho_A\circ (h_0)_{*0}=(\psi_{n_0, \infty})_{*0}.
\eneq
Then  the projection $e'=h_0(1_{C_{n_0}\otimes U})$ satisfies $\rho_A(e')=\rho_A(e)$. Replacing  $e$ by $e'$, we can assume $h_0$ is a unital homomorphism from $C_{n_0}\otimes U$ to $eAe$.
It follows that
\beq\label{Vpair2-5}
\rho_A((h_0)_{*0}(x')-x)=0.
\eneq
Let $u_2=h_0(u_1)$ and $v_2=h_0(v_1).$
We have
\beq\label{Vpair2-6}
\rho_A({\rm bott}_1(u_2, v_2)-x)=0.
\eneq
Choose another non-zero projection $e_1\in A$ such that $e_1e=ee_1=0$ and $\tau(e_1)<\dt_1/16$
for all $\tau\in T(A).$
It follows from \ref{preBot1} that there is a
unital \hm\,  $H:
C(\T^2)\to e_1Ae_1$ such that
\beq\label{Vpair2-7}
H_{*0}(b)=x-{\rm bott}_1(u_2, v_2){{,}}
\eneq
where $b$ is the Bott element in $K_0(C(\T^2))$. (In fact, we can also apply \ref{Next2014/09} here.)
Thus we obtain a pair of unitaries $u_3, v_3\in e_1Ae_1$ such that
\beq\label{Vpair2-8}
u_3v_3=v_3u_3
\andeqn {\rm bott}_1(u_3, v_3)=x-{\rm bott}_1(u_2, v_2).
\eneq
Let $e_2, e_3\in (1-e-e_1)A(1-e-e_1)$ be a pair of non-zero mutually orthogonal projections such that $\tau(e_2)<\dt_1/32$
and $\tau(e_3)<\dt_1/32$
for all $\tau\in T(A).$
Thus $\tau(e+e_1+e_2+e_3)<3\dt_1/16$ for all $\tau\in T(A).$
Then, together with Theorem 17.3 of \cite{GLN-I}, 
(applied to $X=\T$),
we obtain
 a unitary $u_4\in (1-e-e_1-e_2-e_3)A(1-e-e_1-e_2-e_3)$ such that
\beq\label{Vpair2-9}
|\tau (f(u_4))-\tau( f(u))\|<\dt_1/4\tforal f\in {\cal H}_2\cup {\cal H}_1\andeqn \tforal   \tau\in T(A).
\eneq
Let $w=u_2+u_3+u_4+(1-e-e_1-e_2-e_3).$
It follows from Theorem 3.10 of \cite{GLX}
that there exists $u_5\in U(e_2Ae_2)$ such that
\beq\label{Vpair2-10}
\overline{u_5}=\bar{u}{\bar w^*}\in U(A)/CU(A).
\eneq
Since $A$ is simple and has stable rank one, there exists
a unitary
$v_4\in e_3Ae_3$ such that $[v_4]=y-[v_2+v_3+({{e_2+}}e_3)]\in K_1(A).$
Now define
$$
u_6=u_2+u_3+u_4+u_5+e_3\andeqn v_6=v_2+v_3+(1-e-e_1-{{e_2-}}e_3)+{{e_2+}}v_4.
$$
Then
\beq\label{Vpair2-11}
\|u_6v_6-v_6u_6\|<\ep_1/2,\,\,\,{\rm bott}_1(u_6, v_6)=x,\andeqn
[v_6]=y.
\eneq
Moreover,
\beq\label{Vpair2-12}
\tau(f(u_6))\ge \Delta(\hat{f})/2\tforal f\in {\cal H}_1,\\ \label{Vpair2-12-1}
|\tau(f(u))-\tau(f(u_6))|<\gamma_1\andeqn
\bar{u_6}=\bar{u}.
\eneq
By  Corollary 12.7 of \cite{GLN-I} 
and by (\ref{Vpair2-10}) (\ref{Vpair2-12}) and (\ref{Vpair2-12-1}) that there exists a unitary $W\in A$ such that
\beq\label{Vpair2-13}
\|W^*u_6W-u\|<\ep_1/2.
\eneq
Now let $v=W^* v_6W.$
We compute that
\beq\label{Vpair2-14}
\|uv-vu\|<\ep,\,\,\,{\rm bott}_1(u, v)={\rm bott}_1(u_6, v_6)=x,
\andeqn [v]=y.
\eneq

\end{proof}

\begin{cor}\label{Vpair2+1} Let $\ep>0$, $C=\bigoplus_{i=1}^k C^i=\bigoplus_{i=1}^k M_{m(i)}(C(\T))$. Let ${\cal P}_0\subset K_0(C)$ and ${\cal P}_1\subset K_1(C)$ be finite sets generating $K_0(C)$ and $K_1(C)$. There exists $\sigma>0$ satisfying the following condition: Let $A=A_1\otimes U$ be as in Lemma \ref{Vpair2}, let $\iota: C\to A$ be an embedding, and let $\alpha\in KL(A\otimes C(\T), A)$ be such that
$$|\tau(\rho_A(\af(\boldsymbol{\bt}(w))))|<\sigma\min\{\tau'(\iota(1_{C^i}))/m(i),~~1\leq i\leq k, ~\tau'\in T(A)\},$$
for all $w\in {\cal P}_1$ and $\tau\in T(A)$. Then there exists a unitary $v\in \iota(1_C)A\iota(1_C)$ such that
$${\rm Bott}(\iota, v)|_{{\cal P}_0\cup {\cal P}_1}= \af\circ \boldsymbol{\bt}|_{{\cal P}_0\cup {\cal P}_1}.$$

\end{cor}

\begin{proof} Let $e_{11}^i\in C^i=M_{m(i)}(C(\T))$ be the rank one projection of
the upper left corner of $C^i$ and $u^i\in K_1(C^i)$ be the standard generator given by $ze_{11}^i+ (1_{C_i}-e_{11}^i)$, where $z\in C(\T)$ is the identity function from $\T$ to $\T\subset \C$. Without loss of generality, we may assume that
${\cal P}_0=\{[e_{11}^i], 1\leq i\leq k\}$ and ${\cal P}_1=\{u^i, 1\leq i\leq k\}$.
Let $\sigma$ be as in Lemma \ref{Vpair2}. For each $i\in \{1,2,\cdots, k\}$, applying Lemma \ref{Vpair2} to $\iota(e_{11}^i)A\iota(e_{11}^i)$ (in place of $A$), $\iota(ze_{11}^i)\in \iota(e_{11}^i)A\iota(e_{11}^i)$ (in place of $u$) with $x=\af(\boldsymbol{\bt}(u^i))$, $y=\af(\boldsymbol{\bt}([e_{11}^i]))$, one obtains a unitary $v_{11}^i\in \iota(e_{11}^i)A\iota(e_{11}^i)$ in place of $v$. Identifying
$\iota(1_{C^i})A\iota(1_{C^i})\cong \big(\iota(e_{11}^i)A\iota(e_{11}^i)\big)\otimes M_{m(i)}(\C)$,
we define $v^i=v_{11}^i\otimes 1_{m(i)}$. Finally, choose $v=v^1\oplus v^2 \oplus \cdots \oplus v^k\in \iota(1_C)A\iota(1_C)$ to finish the proof.

\end{proof}

\section{More existence theorems for Bott elements}

Using Lemma \ref{Vpair2},
\ref{preBot1}, Corollary 21.11 of \cite{GLN-I}, 
Lemma 18.11 of \cite{GLN-I}, 
and Theorem 12.11 of \cite{GLN-I}, 
we can show the following  result:

\begin{lem}\label{Extbot1}
Let  $A=A_1\otimes U_1,$ where $A_1$ is as in Theorem 14.10 of \cite{GLN-I}
and $B=B_1\otimes U_2,$ where $B_1\in {\cal B}_0$ 
and $U_1, U_2$ are two UHF-algebras of infinite type.
For any $\ep>0,$ any finite subset ${\cal F}\subset A,$ and any finite subset ${\cal P}\subset \underline{K}(A),$ there exist $\dt>0$ and a finite subset ${\cal Q}\subset K_1(A)$ satisfying the following condition:
Let a unital \hm\, $\phi: A\to B$
and $\af\in KL(A\otimes C(\T), B)$  be such that
\beq\label{Extbot1-1}
|\tau\circ \rho_B(\af(\boldsymbol{\bt}(x)))|<\dt \tforal x\in {\cal Q}\tand \tforal \tau\in T(B).
\eneq
Then there exists a unitary $u\in B$ such that
\beq\label{Extbot1-2}
\|[\phi(x), \, u]\|&<&\ep \tforal x\in {\cal F}\andeqn\\
{\rm Bott}(\phi, u)|_{\cal P}&=&\af(\boldsymbol{\bt})|_{\cal P}{{.}}
\eneq
\end{lem}

\begin{proof}
Let $\ep_1>0$ and let ${\cal F}_1\subset A$ be a finite subset satisfying the following condition:
If $$L, L': A\otimes C(\T)\to B$$ are
two {{unital}} ${\cal F}_1'$-$\ep_1$-multiplicative \cp s
such that
\beq\label{Extbot1-4}
\|L(f)-L'(f)\|<\ep_1\tforal f\in {\cal F}_1',
\eneq
where
$$
{\cal F}'_1=\{a\otimes g: a\in {\cal F}_1\andeqn g\in \{z, z^*,1_{C(\T)}\}\},
$$
then
\begin{equation}\label{Extbot1-3}
[L]|_{\boldsymbol{\bt}({\cal P})}=[L']|_{\boldsymbol{\bt}({\cal P})}.
\end{equation}

Let $B_{1,n}=M_{m(1,n)}(C(\T))\oplus M_{m(2,n)}(C(\T))\oplus\cdots \oplus M_{m(k_1(1),n)}(C(\T)),$
$B_{2,n}=PM_{r_1(n)}(C(X_n))P,$ where
$X_n$ is a finite disjoint union of copies of
$S^2, T_{2,k},$ and $T_{3,k}$ (for various $k\ge 1$). Let $B_{3,n}$
be a finite direct sum of \CA s in ${\cal C}_0$ (with trivial $K_1$ and
${\rm ker}\rho_{B_{3,n}}=\{0\}$---see  Proposition 3.5 of \cite{GLN-I}),
$n=1,2,....$
Put $C_n=B_{1,n}\oplus B_{2,n}\oplus B_{3,n},$ $n=1,2,....$
We may write that $A=\lim_{n\to\infty}(C_n,\imath_n)$ as in Theorem 14.10 of \cite{GLN-I}.
with the maps
$\imath_n$ injective (applying Theorem 14.10 of \cite{GLN-I}
to $A_1$),
\beq\label{Extbot1-5}
&&{\rm ker}\rho_A\subset (\imath_{n, \infty})_{*0}({\rm ker}\rho_{C_n}),\andeqn\\
&&\lim_{n\to\infty}\sup\{\tau(1_{B_{1,n}}\oplus 1_{B_{2,n}}): \tau\in T(B)\}=0.
\eneq

Let $\ep_2=\min\{\ep_1/4, \ep/4\}$ and
let ${\cal F}_2={\cal F}_1\cup {\cal F}.$

Let ${\cal P}_{1,1}\subset \underline{K}(B_{1,n_1}),$
${\cal P}_{2,1}\subset \underline{K}(B_{2,n_1})$ and ${\cal P}_{3,1}\subset \underline{K}(B_{3,n_1})$ be finite subsets
such that
$$
{\cal P}\subset [\imath_{n_1, \infty}]({\cal P}_{1,1}){{\oplus}} [\imath_{n_1, \infty}]({\cal P}_{2,1}){{\oplus}}
[\imath_{n_1, \infty}])({\cal P}_{3,1})
$$
for some $n_1\ge 1.$ To simplify the notation, without lose of generality, we may  assume $
{\cal P}\subset [\imath_{n_1, \infty}]({\cal P}_{1,1})\cup [\imath_{n_1, \infty}]({\cal P}_{2,1})\cup
[\imath_{n_1, \infty}])({\cal P}_{3,1})
$.
Let  ${\cal Q}'$ be a finite set of generators of $K_1(C_{n_1})$ and
let ${\cal Q}=[\imath_{n_1,\infty}]({\cal Q}').$ Since $K_i(B_{1,n_1}),~i=0,1,$ are  finitely generated free abelian groups, without loss
of generality, we may assume that ${\cal P}_{1,1}\subset K_0(B_{1,n_1})\cup K_1(B_{1,n_1})$ and generates $K_0(B_{1,n_1})\oplus K_1(B_{1,n_1})$.

Without loss of generality, we may assume that ${\cal F}_1\cup {\cal F}\subset
\imath_{n_1,\infty}(C_{n_1}).$ Let
${\cal F}_{1,1}\subset B_{1, n_1},$ ${\cal F}_{2,1}\subset B_{2, n_2},$ and
${\cal F}_{3,1}\subset B_{3, n_1}$ be finite subsets such that
\beq\label{Extbot1-6}
{\cal F}_1\cup {\cal F}\subset \imath_{n_1, \infty}({\cal F}_{1,1}\cup{\cal F}_{2,1}\cup {\cal F}_{3,1}).
\eneq
Let $e_1=\imath_{n_1, \infty}(1_{B_{1, n_1}}),$ $e_2=\imath_{n_1, \infty}(1_{B_{2, n_1}}),$ and $e_3=1-e_1-e_2.$
 Note that $B_{1, n_1}=\oplus_{i=1}^{s(n_1)}B_{1, n_1}^i$, where $s(n_1)$ is an integer depending on $n_1$ and $B_{1, n_1}^i=M_{m(i, n_1)}(C(\T))$.  We may write
$e_1=\oplus_{i=1}^{s(n_1)}e_1^i$ with $e_1^i=\imath_{n_1, \infty}(1_{B_{1, n_1}}^i)$.
Let $\Delta_1: (B_{2,n_1})_+^{q, {\bf 1}}\setminus \{0\}\to (0,1)$ be defined by
$$
\Delta_1(\hat{h})=(1/2)\inf\{\tau(\phi(\imath_{n_1, \infty}(h)): \tau\in T(B)\}\tforal h\in (B_{2, n_1})_+^{{\bf 1}}\setminus \{0\}.
$$
Let $\Delta_2: B_{{{3}}, n_1}^{q, {\bf 1}}\setminus \{0\}\to (0,1)$ be defined by
$$
\Delta_2(\hat{h})=(1/2)\inf\{\tau(\phi(\imath_{n_1, \infty}(h)): \tau\in T(B)\}\tforal h\in (B_{3, n_1})_+^{{\bf 1}}\setminus \{0\}.
$$
Note that $B_{2, n_1}$ has the form $C$ of Theorem 12.7 of \cite{GLN-I}. 
So we will apply Theorem 12.7 of \cite{GLN-I}. 
Let ${\cal H}_{2,1}\subset (B_{2, n_1}^{\bf 1})_+\setminus \{0\}$ (in place of ${\cal H}_1$),
$\gamma_{2,1}>0$
(in place of $\gamma_1$),
$\dt_{2,1}>0$ (in place of $\dt$),
${\cal G}_{2,1}\subset B_{2, n_1}$ (in place of ${\cal G}$),
${\cal P}_{2,2}\subset \underline{K}(B_{2, n_1})$
(in place of ${\cal P}$),  and ${\cal H}_{2,2}\subset (B_{2, n_1})_{s.a.}$ (in place of ${\cal H}_2$) be the
constants and  finite subsets
 provided
 by Theorem 12.7 of \cite{GLN-I} 
 for  $\ep_2/16,$ ${\cal F}_{2,1},$ and
$\Delta_1$ (we do not need the set ${\cal U}$ in Theorem 12.7 of \cite{GLN-I} 
since $K_1(B_{2,n_1})$ is torsion or zero; see Corollary 12.8 of \cite{GLN-I}). 

Recall that $B_{1, n_1}=\bigoplus_{i=1}^{s(n_1)}B_{1, n_1}^i$ with $B_{1, n_1}^i=M_{m(i, n_1)}(C(\T))$, and $e_1=\bigoplus_{i=1}^{s(n_1)}e_1^i$ with $e_1^i=\imath_{n_1, \infty}(1_{B_{1, n_1}}^i)$. Now let $\sigma>0$ be as provided  by Corollary  \ref{Vpair2+1} (see \ref{Vpair2} also) for ${\cal P}_{1,1}$ and $\ep_2/4$ (in place of $\ep$).
Let  $\dt=\sigma\cdot \inf\{\tau(e_1^i)/m(i,n_1): 1\leq i\leq s(n_1),\, \tau\in T(A)\}.$
It follows from \ref{Vpair2+1} that if $|\tau\circ\rho_B(\af(\bt(x)))|<\dt$ for all $x\in [\imath_{n_1, \infty}])({\cal P}_{1,1})$ then there is a unitary $v_1\in e_1Be_1$ such that
\beq\label{Extbot1-7}
{\rm Bott}(\phi\circ \imath_{n_1, \infty}, v_1)|_{{\cal P}_{1,1}}=\af\circ{\bt}\circ[\imath_{n_1, \infty}]|_{({\cal P}_{1,1})}.
\eneq
 Note that $K_1(B_{2, n_1})$
 is a finite group.
Therefore,
\beq\label{Extbot1-8}
\af(\boldsymbol{\bt}([\imath_{n_1, \infty}])(K_1(B_{2, n_1}))\subset {\rm ker}\rho_B.
\eneq
Define $\kappa_1\in KK(B_{2, n_1}\otimes C(\T){{, A}})$ by
$\kappa_1|_{\underline{K}(B_{2, n_1})}=[\phi\circ \imath_{n_1, \infty}|_{B_{2, n_1}}]$ and
$\kappa_1|_{\boldsymbol{\bt}(\underline{K}(B_{2, n_1})})=
\af|_{ \boldsymbol{\bt}(\underline{K}(B_{2, n_1}))}.$
Since $\imath_{n_1, \infty}$ is injective, by (\ref{Extbot1-8}),
$\kappa_1\in KK_e(B_{2, n_1}\otimes C(\T), e_2Be_2)^{++}.$

Let
$$
\sigma_0=\min\{\gamma_{ 2,1}/2,  \min\{\Delta_1(\hat{h}): h\in {\cal H}_{2,1}\}\cdot \inf\{\tau(e_2): \tau\in T(A)\}.
$$
Define $\gamma_0: T(e_2Ae_2)\to T_f(B_{2,n_1}{{\otimes C(\T)}})$ by
$\gamma_0(\tau)(f\otimes 1_{C(\T)})=\tau\circ \phi\circ \imath_{n_1, \infty}(f)$ for all $f\in B_{2,n_1}$ and
$\gamma_0(1\otimes g)=\int_{\T}g(t)dt$ for all $g\in C(\T).$
It follows from \ref{Next2014/09}, {{applied to the space $X_{n_1}\times \T$,}} that there
is a unital monomorphism $\Phi: B_{2,n_1}\otimes C(\T)\to e_2Ae_2$
such that $[\Phi]=\kappa_1$ and $\Phi_T=\gamma_0.$  Put $L_2=\Phi|_{B_{2,n_1}}${{(identifying $B_{2,n_1}$ with ${ B_{2,n_1}\otimes 1_{C(\T)}}$)}}  and $v_2'=\Phi(1\otimes z),$
where $z\in C(\T)$ is the identity function on the unit circle.
Then $L_2$ is a unital monomorphism  from $ B_{2, n_1}$ to $ e_2Ae_2.$ We also have the following facts:
\beq\label{Extbot-9}
&&[L_2]=[\phi\circ \imath_{n_1, \infty}],\,\,\,\|[L_2(f),\, v_2']\|=0,\\
&&{\rm Bott}(L_2,\, v_2')|_{{\cal P}_{2,2}}=\af(\boldsymbol{\bt}( [\imath_{n_1, \infty}]))|_{{\cal P}_{2,2}},\andeqn\\\label{Extbot-9++}
&&|\tau\circ L_2(f)-\tau\circ \phi\circ \imath_{n_1, \infty}(f)|=0 \tforal f\in {\cal H}_{2,1}\cup {\cal H}_{2,2}
\eneq
and for all $\tau\in T(e_2Ae_2).$
It follows from (\ref{Extbot-9++}) that
\beq\label{Extbot-10}
\tau(L_2(f))\ge \Delta_1(\hat{f})\cdot \tau(e_2)
\tforal f\in {\cal H}_{2,1}\andeqn \tau\in  T(A).
\eneq
By Theorem 12.7 of \cite{GLN-I} 
(see also Corollary 12.8 of \cite{GLN-I}), 
there exists a unitary $w\in e_2Ae_2$ such that
\beq\label{Extbot-11}
\|{\rm Ad}\, w\circ L_2(f)-\phi\circ \imath_{n_1, \infty} (f)\|<\ep_2/16\tforal  f\in {\cal F}_{2,1}.
\eneq
Define $v_2=w^*v_2'w.$
Then, for all $f\in {\cal F}_{2,1}$,
\beq\label{Extbot-12}
\|[\phi\circ \imath_{n_1,\infty}(f),\,v_2]\|<\ep_2/8\andeqn
{\rm Bott}(\phi\circ \imath_{n_1, \infty},\, v_2)|_{{\cal P}_{2,1}}=
\af(\boldsymbol{\bt}([\imath_{n_2, \infty}]))|_{{\cal P}_{2,1}}.
\eneq

Note that $B_{3, n_1}$ has the form $C$ of Theorem 12.7 of \cite{GLN-I}. 
  Let ${\cal H}_{3,1}\subset (B_{3, n_1})_+^{\bf 1}\setminus \{0\}$
(in place of ${\cal H}_1$),
$\gamma_{3,1}>0$ (in place of $\gamma_1$),
$\dt_{3,1}>0$ (in place of $\dt$), ${\cal G}_{3,1}\subset B_{3, N_1}$ (in place of ${\cal G}$),
${\cal P}_{3,2}\subset \underline{K}(B_{3,n_1})$ (in place of ${\cal P}$),
and ${\cal H}_{3,2}\subset
(B_{3, n_1})_{s.a.}$ (note that $K_1(B_{3, n_1})=\{0\}$) be  constants and
finite subsets as provided by Theorem  12.7 of \cite{GLN-I} 
for $\ep_2/16,$ ${\cal F}_{3,1},$ and $\Delta_2$ (see also Corollary 12.8 of \cite{GLN-I}). 

Let
$$
\sigma_1=(\gamma_{3,1}/2)\min\{\tau(e_3): \tau\in T(A)\}\cdot \min\{ \Delta_1(\hat{f}): f\in {\cal H}_{3,1}\}.
$$

Note that ${\rm ker}\rho_{B_{3, n_1}}=\{0\}$ and $K_1(B_{3, n_1})=\{0\}$ (see Proposition 3.5 of \cite{GLN-I})
Therefore, ${\rm ker}\rho_{B_{3, n_1}\otimes C(\T)}={\rm ker}\rho_{B_{3, n_1}}=\{0\}.$
Define
$\kappa_2\in KK(B_{3, n_1}\otimes C(\T), A)$ as follows:
$$
\kappa_2|_{\underline{K}(B_{3, n_1})}=[\phi\circ \imath_{n_1, \infty}]|_{B_{3, n_1}}\andeqn
\kappa_2|_{\boldsymbol{\bt}(\underline{K}(B_{3, n_1})}=\af(\boldsymbol{\bt}(\imath_{n_1, \infty})|_{\underline{K}(B_{3, n_1})}.
$$
Thus $\kappa_2\in KK_e(B_{3,n_1}\otimes C(\T), e_3Ae_3)^{++}.$
It follows from \ref{CtimesText0} that there is  a unital ${\cal G}_{3,1}$-$\min\{\ep_2/16,\dt_{3,1}/2\}$-multiplicative
\cp\, $L_3: B_{3, n_1}\to e_3Ae_3$ and a unitary $v_3'\in e_3Ae_3$ such that
\beq\label{Extbot-13}
&&\hspace{-1.4in}[L_3]=[\phi],\,\,\, \|[L_3(f),\, v_3']\|<\ep_2/16\tforal f\in {\cal G}_{3,1},\\
&&\hspace{-1.4in}{\rm Bott}(L_3,\, v_3')|_{{\cal P}_{3,1}}=\kappa_2|_{\boldsymbol{\bt}({{\cal P}_{3,2}})},\andeqn\\\label{Extbot-13++}
&&\hspace{-1.4in}|\tau\circ L_3(f)-\tau\circ \phi\circ \imath_{n_1, \infty}(f)|<\sigma_1\tforal f\in {\cal H}_{3,1}\cup {\cal H}_{3,2}
\eneq
and
 for all $ \tau\in T(e_3Ae_3).$
It follows that  (\ref{Extbot-13++}) that
\beq\label{Extbot-14}
\tau(L_3(f))\ge \Delta_1(\hat{f})\tau(e_3)\tforal f\in {\cal H}_{3,1}\andeqn \tforal \tau\in T(A).
\eneq
It follows from Theorem 12.7 of \cite{GLN-I}, 
and its corollary (see part (2) of Corollary 12.8 of \cite{GLN-I}), that there exists a unitary $w_1\in e_3Ae_3$ such that
\beq\label{Extbot-15}
\|{\rm Ad}\, w_1\circ L_2(f)-\phi\circ \imath_{n_1, \infty}(f)\|<\ep_2/16\tforal f\in {\cal F}_{3,1}.
\eneq
Define $v_3=w_1^*v_3'w_1.$ Then
\beq\label{Extbot-16}
\|[\phi\circ \imath_{n_1, \infty}(f), \, v_3]\|<\ep_2/8\andeqn {\rm Bott}(\phi\circ \imath_{n_1, \infty},\, v_3)|_{{\cal P}_{3,1}}
={\rm Bott}(L_3,\, v_3')|_{{\cal P}_{3,1}}.
\eneq
Let $v=v_1+v_2+v_3.$ Then
\beq\label{Extbot-17}
\|[\phi(f), \, v]\|<\ep\tforal f\in {\cal F}.
\eneq
Moreover, we compute that
\beq\label{Extbot-18}
{\rm Bott}(\phi, \, v)|_{\cal P}=\af|_{\boldsymbol{\bt}({\cal P})}.
\eneq
\end{proof}

We have actually  proved the following result:
\begin{lem}\label{Extbot2}
Let  $A=A_1\otimes U_1,$ where $A_1$ {{is}}  as in Theorem 14.10 of \cite{GLN-I}
and $B=B_1\otimes U_2,$ where $B_1\in {\cal B}_0$ is a  unital simple \CA\, and where $U_1, U_2$ are two UHF-algebras of infinite type. Write
$A=\lim_{n\to\infty} (C_n, \imath_n)$ as described in Theorem 14.10 of \cite{GLN-I}.
For any $\ep>0,$ any finite subset ${\cal F}\subset A,$ and any finite subset ${\cal P}\subset \underline{K}(A),$ there exists an integer $n\ge 1$
such that ${\cal P}\subset [\imath_{n,\infty}](\underline{K}(C_n)) $
and  there is a finite subset
${\cal Q}\subset K_1(C_n)$ which generates $K_1(C_n)$
and there exists $\dt>0$ satisfying the following condition:
Let $\phi: A\to B$ be a unital \hm\, and let $\af\in KK(C_n\otimes C(\T), B)$ such that
\begin{equation*}
|\tau\circ \rho_B(\af(\boldsymbol{\bt}(x)))|<\dt \tforal x\in {\cal Q}\tand \tforal \tau\in T(B).
\end{equation*}
Then there exists a unitary $u\in B$ such that
\begin{equation*}
\|[\phi(x), \, u]\|<\ep \tforal x\in {\cal F}
\andeqn
{\rm Bott}(\phi\circ [\imath_{n,\infty}], u)=\af(\boldsymbol{\bt}).
\end{equation*}
\end{lem}

\begin{proof}
Note that, in the proof of Lemma \ref{Extbot1},
$K_i(B_{j,1})$ is finitely generated, $i=0,1,$ and $j=1,2,3.$
Then $KK(B_{j,1}, A)=KL(B_{j,1},B)$ (for any unital \CA\, $B$), $j=1,2.$
Moreover (see \cite{DL}),  there exists an integer $N_0>1$ such that elements in
${\rm Hom}_{\Lambda}(\underline{K}(B_{j,1}), \underline{K}(B))$
are  determined  by their restrictions to $K_i(B_{j,1})$ and
$K_i(B_{j,1},\Z/m\Z),$ $m=2,3,...,N_0.$ In particular, we may assume, in the proof
of  \ref{Extbot1}, that ${\cal P}_{j,1}$ generates $K_i(B_{j,1})\oplus \bigoplus_{m=2}^{N_0} K_i(B_{j,1}, \Z/m\Z),$
$j=1,2,3.$
\end{proof}
\begin{rem}\label{RemB2} Note that, in the statement above, if an integer $n$ works, any integer $m\ge n$ also
works.
In the terminology of Definition 3.6 of \cite{L-N}, the statement above also implies that
$B$ has properties (B1) and (B2) associated with $C.$
\end{rem}

\begin{cor}\label{L86}
Let $B\in {\cal B}_{{0}},$
let $A_1\in {\cal B}_{{0}},$ let $C=B\otimes U_1,$ and let $A=A_1\otimes U_2,$ where $U_1$ and $U_2$ are
UHF-algebras of infinite type. Suppose that $B$ satisfies the UCT and suppose that $\kappa\in KK_e(C,A)^{++},$
$\gamma: T(A)\to T(C)$ is a continuous affine map, and $\af: U(C)/CU(C)\to
U(A)/CU(A)$ is a continuous \hm\, for which $\gamma,\, \af,$ and $\kappa$ are compatible. Then, there exists a unital monomorphism $h: C\to A$ such that
\begin{enumerate}
\item $[h]=\kappa$ in $KK_e(C,A)^{++}$,
\item $h_T=\gamma$ and $h^{\ddag}=\af.$
\end{enumerate}
\end{cor}

\begin{proof}
The proof follows the same lines as that of Theorem 8.6 of \cite{Lnclasn}, following
the proof of Theorem 3.17 of \cite{L-N}. Denote by $\overline{\kappa}\in KL(C, A)$
 the image of $\kappa$. It follows from Lemma \ref{L85} that there is a unital monomorphism $\phi: C\to A$ such that
$$[\phi]=\overline{\kappa}, \quad \phi^\ddag=\alpha,\quad\textrm{and}\quad (\phi)_T=\gamma.$$
Note that it follows from the UCT that (as an element  of $KK(C,A)$)
$$\kappa-[\phi]\in\textrm{Pext}(K_*(C), K_{*+1}(A)).$$
By Lemmas \ref{Extbot2} (see also Remark \ref{Remark202011-1}) and \ref{Vpair2}, the \CA\,  $A$ has Property (B1) and Property (B2) associated with C
in the sense of \cite{L-N}.
Since $A$ contains a unital copy of $U_2,$  it is infinite dimensional, simple and antiliminal. It follows
from a result in \cite{AS} that $A$ contains an element $b$ with $sp(b)=[0,1].$ Moreover,
$A$ is approximately divisible.
It follows from Theorem 3.17 of \cite{L-N}
that  there is a unital monomorphism $\psi_0: A\to A$ which is approximately inner and such that
$$[\psi_0\circ\phi]-[\phi]=\kappa-[\phi]\quad\textrm{in $KK(C, A)$}.$$ Then the map
$$h:=\psi_0\circ\phi$$
satisfies the requirements of the corollary.
\end{proof}

\begin{lem}\label{Extbot3}
Let  $A=A_1\otimes U_1,$ where $A_1$ is as in Theorem 14.10 of \cite{GLN-I}
and $B=B_1\otimes U_2,$ where $B_1\in {\cal B}_{{0}}$ {{is}} a unital simple \CA\, and where $U_1, U_2$ are two UHF-algebras of infinite type. Let
$A=\lim_{n\to\infty} (C_n, \imath_n)$ be as described in Theorem 14.10 of \cite{GLN-I},
For any $\ep>0,$ any $\sigma>0,$ any finite subset ${\cal F}\subset A,$ any finite subset ${\cal P}\subset \underline{K}(A),$
and any projections $p_1, p_2,...,p_k,q_1,q_2,...,q_k\in A$
such that $\{x_1, x_2,...,x_k\}$
generates a free abelian subgroup  $G$ of $K_0(A),$  where $x_i=[p_i]-[q_i],$ $i=1,2,...,k,$
there exists an integer $n\ge 1$
such that ${{x_j\in}} {\cal P}\subset [\imath_{n,\infty}](\underline{K}(C_n)) $
and  there is a finite subset
${\cal Q}\subset K_1(C_n)$ which generates $K_1(C_n)$
and there exists $\dt>0$ satisfying the following condition:
Let $\phi: A\to B$ be a unital \hm,  let
$\Gamma: G\to {{U}}(B)/CU(B)$ be  a \hm\, and  let $\af\in KK(C_n\otimes C(\T), B)$
 such that
\beq
&&{{\af(\boldsymbol{\bt}(g)))=\kappa_1^B(\Gamma(g))\rforal g\in {\imath_{n, \infty}}_{*0}^{-1}(G)\tand}}\\
&&|\tau\circ \rho_B(
\af(\boldsymbol{\bt}(x)))|<\dt \tforal x\in {\cal Q}\tand \tforal \tau\in T(B).
\eneq
Then there exists a unitary $u\in B$ such that
\begin{equation*}
\|[\phi(x), \, u]\| < \ep \tforal x\in {\cal F},\,\,\,
{\rm Bott}(\phi\circ [\imath_{n,\infty}], u) = \af(\boldsymbol{\bt}),
\end{equation*}
and
\begin{equation*}
{\rm dist}(\overline{\langle ((1-\phi(p_i))+\phi(p_i)u)((1-\phi(q_i))+\phi(q_i)u^*)\rangle}, \Gamma(x_i))<\sigma,\,i=1,2,...,k.
\end{equation*}
\end{lem}

\begin{proof}
This follows from Lemma \ref{Extbot2} and  Theorem \ref{BB-exi+}.
In fact,  for any $0<\ep_1<\ep/2$ and finite subset ${\cal F}_1\supset {\cal F},$
by \ref{Extbot2}, there exists
an integer $n\ge 1,$   a finite subset ${\cal Q}\subset K_1(C_n),$ and $\dt>0$ as described above, and a unitary $u_1\in U_0(B),$  such that
\begin{equation*}
\|[\phi(x), \, u_1]\|<\ep_1\tforal x\in {\cal F}_1
\end{equation*}
and
\beq\label{Add2021-n1}
{\rm Bott}(\phi\circ \imath_{n,\infty}, u_1)  = \af(\boldsymbol{\bt})|_{\cal P}.
\eneq

Choosing a smaller $\ep_1$ and a larger ${\cal F}_1,$ {{if}} necessary, we may assume that the class
$$\overline{\langle ((1-\phi(p_i))+\phi(p_i)u_1)((1-\phi(q_i))+\phi(q_i)u_1^*)\rangle}\in {{U(B)}}/CU(B)$$
is well defined for all $1\le i\le k.$
Define {{a}}  map $\Gamma_1: G\to {{U(B)}}/CU(B)$ by
\begin{equation}\label{Extbot2-4}
\Gamma_1(x_i)=\overline{\langle ((1-\phi(p_i))+\phi(p_i)u_1)(1-\phi(q_i))+\phi(q_i)u_1^*)\rangle},\quad i=1, 2, ..., k.
\end{equation}

%
Choosing a  large  enough $n,$ \wilog,    we may assume that there are projections $p_1', p_2',...,p_k',$
$q_1',q_2,'...,q_k'\in C_n$ such that $\imath_{n, \infty}(p_i')=p_i$ and
$\imath_{n, \infty}(q_i')=q_i,$ $i=1,2,...,k.$ 
 Moreover, we may assume that ${\cal F}_1\subset \imath_{n, \infty}(C_n).$

Let $\Gamma_2: G\to U_0(B)/CU(B)$ {{be defined}} by
$\Gamma_2(x_i)=\Gamma_1(x_i)^*\Gamma(x_i),$ $i=1,2,...,k$ (see also \eqref{Add2021-n1}).
It follows  by Theorem \ref{BB-exi+} that is a unitary $v\in U_0(B)$ such that
\beq\label{Extbot2-5}
\|[\phi(x), \, v]\|<\ep/2\rforal x\in {\cal F},\\
{\rm Bott}(\phi\circ \imath_{n,\infty}, v)=0,\andeqn\\
{\rm dist}(\overline{\langle ((1-\phi(p_i))+\phi(p_i)v)((1-\phi(q_i))+\phi(q_i)v^*)\rangle}, \Gamma_2(x_i))&<&\sigma,
\eneq
$i=1,2,...,k.$ Define $u=u_1v,$
\beq
X_i&=&\overline{\langle ((1-\phi(p_i))+\phi(p_i)u_1)((1-\phi(q_i))+\phi(q_i)u_1^*)\rangle},\andeqn\\
Y_i&=&\overline{\langle ((1-\phi(p_i))+\phi(p_i)v)((1-\phi(q_i))+\phi(q_i)v^*)\rangle},
\eneq
 $i=1,2,...,k.$
We then compute that
\beq
\label{Extbot2-6}
&&\|[\phi(x), \, u]\|<\ep_1+\ep/2<\ep\tforal x\in {\cal F},\\\nonumber
&&{\rm Bott}(\phi\circ \imath_{n,\infty}, u)={\rm Bott}(\phi\circ \imath_{n,\infty}, u_1)=\af(\boldsymbol{\bt}),\,\,\andeqn\\\nonumber
&&{\rm dist}(\overline{\langle ((1-\phi(p_i))+\phi(p_i)u)((1-\phi(q_i))+\phi(q_i)u^*)\rangle}, \Gamma(x_i))\\\nonumber
&\le & {\rm dist}(X_iY_i, \Gamma_1(x_i)Y_i)+{\rm dist}(\Gamma_1(x_i)Y_i, \Gamma(x_i))\\\nonumber
&=&{\rm dist}(X_i, \Gamma_1(x_i))+{\rm dist}(Y_i, \Gamma_2(x_i))< \sigma,
\eneq
for $i=1,2,...,k$.
\end{proof}

\section{Another Basic Homotopy Lemma}

\begin{lem}\label{densesp}
Let $A$  be a unital C*-algebra and let $U$ be an infinite dimensional  UHF-algebra. Then there is a unitary $w\in U$ such that for
any unitary $u\in A$, one has
\begin{equation}\label{densesp1909}
\tau(f(u\otimes w))=\tau(f(1_A\otimes w))=\int_{\T}f dm,\quad f\in\mathbb C(\T),\ \tau\in T(A\otimes U),
\end{equation}
where $m$ is normalized Lebesgue
 measure on $\T.$
Furthermore, for any $a\in A$  and $\tau\in T(A\otimes U)$, $\tau (a\otimes w^j)=0$ if $j\not=0$.
\end{lem}

\begin{proof}
Denote by $\tau_U$ the unique trace of $U$. Then any trace $\tau\in T(A\otimes U)$ is a product trace, i.e.,
$$\tau(a\otimes b)=\tau(a\otimes 1)\otimes\tau_U(b),\quad a\in A, b\in U.$$

Pick a unitary $w\in U$ such that the spectral measure of $w$ is
Lebesgue measure (a Haar unitary). Such a unitary always exists. (It can be constructed directly; or, one can consider a strictly ergodic Cantor system $(\Omega, \sigma)$ such that $K_0(\mathrm{C}(\Omega)\rtimes_\sigma\Z)\cong K_0(U).$   One then notes that the canonical unitary in $\mathrm{C}(\Omega)\rtimes_\sigma\Z$ is a Haar unitary. Embedding $\mathrm{C}(\Omega)\rtimes_\sigma\Z$ into $U$, one obtains a Haar unitary in $U.$)
Then one has, for each $n\in \Z,$
$$
\tau_U(w^n)=
\left\{
\begin{array}{ll}
1, & \textrm{if $n=0$},\\
0, & \textrm{otherwise}.
\end{array}
\right.
$$
Hence, for any $\tau\in T(A\otimes U)$, one has, for each $n\in \Z,$
$$\tau((u\otimes w)^n)=\tau(u^n\otimes w^n)=\tau(u^n\otimes 1)\tau_U(w^n)=
\left\{
\begin{array}{ll}
1, & \textrm{if $n=0$},\\
0, & \textrm{otherwise};
\end{array}
\right.
$$
and therefore,
$$\tau(P(u\otimes w))=\tau(P(1\otimes w))=\int_{\T}P(z)dm$$
for any polynomial $P$. Similarly, $\tau(P(u\otimes w)^*)=\int_TP({\bar z})dm$
for any polynomial $P.$
Since polynomials in $z$ and $z^{-1}$ are dense in $\mathrm{C}(\T)$, one has
$$\tau(f(u\otimes w))=\tau(f(1\otimes w))=\int_{\T}fdm,\quad f\in \mathrm{C}(\T),$$
as desired.
\end{proof}

\begin{lem}\label{Defuv}
Let $A$ be a unital separable amenable \CA\, and let
$L: A\otimes C(\T)\to B$ be a unital completely positive linear map, where $B$ is another
unital amenable \CA. Suppose that $C$ is a unital \CA\, and $u\in C$ is a
unitary. Then, there is  a unique  pair of
unital completely positive linear maps  $\Phi_1, \Phi_2: A\otimes C(\T)\to B\otimes C$
such that
\beq
&&\Phi_i|_{A\otimes 1_{C(\T)}}=\imath\circ L|_{A\otimes 1_{C(\T)}}\, (i=0,1)\,\,  \textrm{and}\quad \Phi_1(a\otimes z^j)=L(a\otimes z^j)\otimes u^j\tand\\
&&\Phi_2(a\otimes z^j)=L(a\otimes 1_{C(\T)})\otimes u^j
\eneq
for any $a\in A$ and any  integer $j,$ where $\imath: B\to B\otimes C$ is the standard inclusion.

Furthermore,
if $\dt>0$ and ${\cal G}\subset A\otimes C(\T)$ is a finite subset, there are  $\dt_1>0$ and finite set  ${\cal G}_1\subset A\otimes C(\T)$  (which do not  depend on $L$) such that  if
$L$ is {{${\cal G}_1$-$\dt_1$-}}multiplicative, then 
$\Phi_i$  
is ${\cal G}$-$\dt$-multiplicative.


\end{lem}

\begin{proof}
Considering the map $L': A\otimes C(\T)\to B$ by
$L'(a\otimes f)=L(a\otimes f(1))$ for all $a\in A$ and $f\in C(\T),$ where $f(1)$ is the evaluation of $f$ at
1 (a point on the unit circle), we see that it suffices to prove the statement for $\Phi_1$ only.
%

Denote by $C_0$ the unital C*-subalgebra of $C$ generated by $u.$  Then the tensor product map
$$L \otimes\id_{C_0}: A\otimes C(\T)\otimes C_0 \to B\otimes C_0$$
is unital and completely positive (see, for example, Theorem 3.5.3 of \cite{BO-Book}). Define the homomorphism $\psi: C(\T) \to C(\T)\otimes C_0$ by
$$\psi(z)=z\otimes u.$$ By Theorem 3.5.3 of \cite{BO-Book} again, the tensor product map
$$\id_A\otimes\psi: A\otimes C(\T) \to A\otimes C(\T)\otimes C_0$$
is unital and completely positive. Then, the map
$$\Phi:=(L\otimes\id_{C_0})\circ(\id_A\otimes \psi)$$
satisfies the requirement of the first part of the lemma.

Let us consider the second part of the lemma. Let $\dt>0$ and ${\cal G}\subset A\otimes C(\T)$ be a finite subset.
Without loss of generality, we may assume that
elements in ${\cal G}$ have the form $\sum_{-n\le i\le n} a_i\otimes z^i.$
Let $N=\max\{  n: \sum_{-n\le i\le n} a_i\otimes z^i\in {\cal G}\},$
let $\dt_1=\dt/2N^2,$ and let
${\cal G}_1\supset \{a_i\otimes z^i: -n\le i\le n: \sum_{-n\le i\le n} a_i\otimes z^i\in {\cal G}\}.$

Then
\beq\label{Defuv-6}\nonumber
\Phi((\sum_{-n\le i\le n} a_i\otimes z^i)(\sum_{-n\le i\le n} b_i\otimes z^i))
&=&\sum_{i,j} \Phi(a_ib_j\otimes z^{i+j})\\\nonumber
&=& \sum_{i,j}L(a_ib_j\otimes z^{i+j})\otimes u^{i+j}\\\nonumber
&\approx_{\dt}& (\sum_{-n\le i\le n}L(a_i\otimes z^i)\otimes u^i)(\sum_{-n\le i\le n}L(b_i\otimes z^i)\otimes u^i)\\\nonumber
&=&\Phi(\sum_{-n\le i\le n} a_i\otimes z^i)\Phi(\sum_{-n\le i\le n} b_i\otimes z^i),
\eneq
if $\sum_{-n\le i\le n}a_i\otimes z^i, \sum_{-n\le i\le n}b_i\otimes z^i\in {\cal G}.$
It follows that $\Phi$ is {{${\cal G}$-$\dt$}}-multiplicative.

Let $P(\T)=\{\sum_{i=-n}^n c_iz^i, c_i\in \C\}$ denote the algebra of  Laurent  polynomials.  Uniqueness of $\Phi$ follows from
the fact that $A\otimes C(\T)$ is the closure of the algebraic tensor product $A\otimes_{alg} P(\T)$.
\end{proof}


The following corollary follows immediately from \ref{Defuv} and \ref{densesp}.

\begin{cor}\label{densesp3}
Let $C$ be a unital \CA\, and let $U$ be an infinite dimensional UHF-algebra.
{{For any $\dt>0$ and any finite subset ${\cal G}\subset
C\otimes C(\T),$ there exist $\dt_1>0$ and a finite subset
 ${\cal G}_1\subset C\otimes C(\T)$ satisfying the following condition:}}
For any $1>\sigma_1,\,\sigma_2>0,$ any finite subset ${\cal H}_1\subset C(\T)_+\setminus \{0\},$ any finite subset ${\cal H}_2\subset  (C\otimes C(\T))_{s.a.},$   
 and any unital
 ${\cal G}_1$-$\dt_1$-multiplicative \cp\,  $L: C\otimes C(\T)\to A,$
 where $A$ is another unital \CA,
 there exists  a unitary $w\in U$ satisfying the following conditions:
 \beq\label{densesp3-1}
 |\tau(L_1(f))-\tau(L_2(f))|<\sigma_1\tforal f\in {\cal H}_2,\ \tau\in T(B),\ \textrm{and}\ \\ \label{densesp3-2}
 \tau(g(1_A\otimes w))\ge \sigma_2(\int g dm)\tforal g\in {\cal H}_1,\ \tau\in T(B),
 \eneq
where $B=A\otimes U$ and $m$ is  normalized Lebesgue
 measure
 on $\T,$ and $L_1, L_2: C\otimes C(\T)\to A\otimes U$ are $\mathcal G$-$\dt$-multiplicative \morp s as
 given by Lemma {\rm \ref{Defuv}}  (as $\Phi_1, \Phi_2$) such that
 $L_i(c\otimes 1_{C(\T)})=L(c\otimes 1_{C(\T)})\otimes 1_U$ ($i=1,2$),  $L_1(c\otimes z^j)=L(c\otimes z^j)\otimes w^j,$
 and $L_2(c\otimes z^j)=L(c \otimes 1_{C(\T)})\otimes w^j$ for all $c\in C$ and all
 integers {{$j.$}}

 \begin{proof}
 Fix a $ \dt>0$ and a finite subset ${\cal G}.$ Let $\dt_1>0$
 and ${\cal G}_1
\subset  C\otimes C(\T)$
 be as given by Lemma \ref{Defuv} for $A$ (in place of $B$).

 Let $0<\sigma_1, \sigma_2<1,$ ${\cal H}_1,$ and ${\cal H}_2$ be as given in the statement.
 There are a finite subset ${\cal F}_C\subset C$ and an integer $N>0$ such that,
 for any $h\in {\cal H}_2,$
 \beq\label{190903}
\|h-\sum_{j=-N}^N a_{h,i}\otimes z^i\|<\sigma_1/2,
 \eneq
 where $a_{h,i}\in {\cal F}_C\cup\{0\}.$
 Set  ${\cal H}_2'=\{\sum_{i=-N}^N a_{h,i}\otimes z^j: h\in {\cal H}_2\}.$

 Now assume that $L$ is as stated for ${\cal G}_1$ and $\dt_1$ mentioned above.

 Choose $w\in U$ as in Lemma \ref{densesp}. Let $L_1, L_2:  C\otimes C(\T)\to A\otimes U$ be as described in the corollary.  In other words, let
$L_1: C\otimes C(\T)\to A\otimes U$ be the map as $\Phi_1$ given by  Lemma \ref{Defuv}
 (with $C$ in place of $A,$ $A$ in place of $B,$  $U$ in place of $C,$ and $w\in U$ in place  $u\in U$),
 and let $L_2: C\otimes C(\T)\to A\otimes U$ be defined by $L_2(c\otimes f)=L(c\otimes 1_{C(\T)})\otimes f(w)$
 for all $c\in C$ and $f\in C(\T)$ (as $\Phi_2$ in Lemma \ref{Defuv}).
By the choice of ${\cal G}_1$ and $\dt_1,$
 $L_1$ and $L_2$ are ${\cal G}$-$\dt$-multiplicative (as in \ref{Defuv}).

 By Lemma \ref{densesp} (and by Lemma \ref{Defuv}),  for $h\in {\cal H}_2',$
 \beq
 \tau(L_1(h))&=&\tau(\sum_{i=-N}^N L_1(a_{h,i}\otimes z^i))=\sum_{i=-N}^N \tau(L(a_{h,i}\otimes z^i)\otimes w^i)\\
 &=&\tau(L(a_{h,0}\otimes 1_{C(\T)}))=
 \tau(L_2(a_{h,0}\otimes 1_{C(\T)}))\\
 &=&\tau(\sum_{i=-N}^N L(a_{h,i}\otimes 1_{C(\T)})\otimes w^i)=\tau(L_2(h)).
 \eneq
 Thus, combining \eqref{190903}, inequality \eqref{densesp3-1} holds.
 By \eqref{densesp1909},  
 \eqref{densesp3-2} also holds.

 \end{proof}

%
\end{cor}


\begin{lem}\label{BHfull}
Let $A=A_1\otimes U_1,$
where $A_1\in {\cal B}_0$ and satisfies the UCT and $U_1$ is a UHF-algebra
of infinite type.
For any $1>\ep>0$ and any finite subset ${\cal F}\subset A,$ there exist  $\dt>0,$ $\sigma>0$, a finite subset
${\cal G}\subset A,$ a finite subset $\{p_1,p_2,...,p_k, q_1,q_2,...,q_k\}$ of projections of $A$ such that
$\{[p_1]-[q_1],[p_2]-[q_2],...,[p_k]-[q_k]\}$ generates a free abelian subgroup $G_u$ of $K_0(A),$
and a finite subset ${\cal P}\subset \underline{K}(A),$
satisfying the following condition:

Let $B=B_1\otimes U_2,$
where $B_1$ is in ${\cal B}_0$  and satisfies the UCT and  $U_2$ is a  UHF-algebra of infinite type.
Suppose that $\phi: A\to B$ is a unital \hm.

If $u\in U(B)$  is a unitary such that
\beq\label{BHfull-1}
&&\|[\phi(x),\, u]\|<\dt\tforal x\in {\cal G},\\\label{BHfull-1n}
&&{\rm Bott}(\phi,\, u)|_{\cal P}=0{{,}}\\\label{BHfull-1n1}
&&{\rm dist}(\overline{\langle ((1-\phi(p_i))+\phi(p_i)u)(1-\phi(q_i))+\phi(q_i)u^*)\rangle}, {\bar 1})<\sigma,\tand\\\label{BHfull-1n2}
&&\mathrm{dist}(\bar{u}, \bar{1})<\sigma,
\eneq
then there exists a continuous path of unitaries $\{u(t): t\in [0,1]\}\subset U(B)$ such
that
\beq\label{BHTL-3}
&&u(0)=u,\,\,\, u(1)=1_B,\\
&&{\rm dist}(u(t), CU(B))<\ep {\tforal} t\in [0,1],\\
&&\|[\phi(a),\, u(t)]\|<\ep\tforal a\in {\cal F}\tand for\,\, all\,\, t\in [0,1],\tand\\
&&{\rm length}(\{u(t)\})\le 2\pi+\ep.
\eneq
\end{lem}

\begin{proof}
It is enough
 to prove the statement under the  assumption that $u\in CU(B)$.

Recall that every C*-algebra in ${\cal B}_0$ has stable rank one (see Theorem 9.7 of \cite{GLN-I}).
Define
$$
\Delta(f)=(1/2)\int f dm \tforal f\in C(\T)_+^{\bf 1}\setminus \{0\},
$$
where $m$ is normalized Lebesgue measure on the unit circle $\T.$
Let $A_2=A\otimes C(\T).$ Let ${\cal F}_1=\{x\otimes f: x\in {\cal F}, f=1, z, z^*\}.$
We may assume that ${\cal F}$ is a subset of the unit ball of $A.$
Let $1>\dt_1>0$ (in place of $\dt$), ${\cal G}_1\subset A_2$
(in place of ${\cal G}$), $1/4>\sigma_1>0, \,1/4>\sigma_2>0,$ ${\cal \widetilde{P}}\subset \underline{K}(A_2),$
${\cal H}_1\subset C(\T)_+^{\bf 1}\setminus \{0\},$
${\cal H}_2\subset (A_2)_{s.a.},$
and ${\cal U}\subset U(M_2(A_2))/CU(M_2(A))$ be the  constants and
finite subsets provided  by Theorem 12.11 part (b) of \cite{GLN-I} 
for $\ep/4$ (in place of $\ep$), ${\cal F}_1$ (in place of ${\cal F}$),  $\Delta,$ and $A_2$ (in place of $A$).

We may assume ${\cal H}_2\subset A_{s.a.}^{\bf 1}.$
that
$$
{\cal G}_1=\{a\otimes f: a\in {\cal G}_2\andeqn  f=1,z,z^*\},
$$
where ${\cal G}_2\subset A$ is a finite subset, and
$
{\cal \widetilde{P}}={\cal P}_1\cup \boldsymbol{\bt}({\cal P}_2),
$
where ${\cal P}_1, {\cal P}_2\subset \underline{K}(A)$ are finite subsets. Define ${\cal P}={\cal P}_1\cup {\cal P}_2$.

We may assume that $(2\dt_1, {\cal \tilde{P}}, {\cal G}_1)$ is a $KL$-triple for $A_2$, $(2\dt_1, {\cal P}_1, {\cal G}_2)$ is a $KL$-triple for $A$, and $1_{A_1} \otimes \mathcal H_1\subset \mathcal H_2$.


We may
choose $\sigma_1$ and $\sigma_2$ such that
\beq\label{Bufull-9}
\max\{\sigma_1, \sigma_2\}<(1/4)\inf\{\Delta(f): f\in {\cal H}_1\}.
\eneq

Let $\dt_2$ (in place of $\dt_1$) and a finite subset  ${\cal G}_3$ (in place of ${\cal G}_1$)
be as provided by  \ref{densesp3} for $A$ (in place of $C$), $\dt_1/4$ (in place of $\dt$), and
${\cal G}_1$ (in place of ${\cal G}$).
Choosing
smaller $\dt_2,$
\wilog, we may assume that
${\cal G}_3=\{a\otimes f: g\in {\cal G}_2'\andeqn f=1,z,z^*\}$ for a large finite subset
${\cal G}_2'\supset {\cal G}_2.$
We may assume that $\dt_2<\dt_1.$

We may further assume
that
\beq\label{BUfull-8}
{\cal U}={\cal U}_1\cup \{\overline{1\otimes z}\}\cup  {\cal U}_2,
\eneq
where ${\cal U}_1=\{\overline{a\otimes 1}: a\in {\cal U}_1'\subset U(A)\},$
${\cal U}_1'$ is a finite subset, and
${\cal U}_2\subset U(A_2)/CU(A_2)$ is a finite subset
whose elements represent a finite subset of $\boldsymbol{\bt}(K_0(A)).$
So we may assume that ${\cal U}_2\in J_c(\boldsymbol{\bt}(K_0(A))).$
As in Remark 12.12 of \cite{GLN-I}, 
we may assume that the subgroup of $J_c(\boldsymbol{\bt}(K_0(A)))$
generated by ${\cal U}_2$ is free abelian.  Let ${\cal U}_2'$ be a finite subset of unitaries such
that $\{{\bar x}: x\in {\cal U}_2'\}={\cal U}_2.$ We may also
assume that the unitaries in ${\cal U}_2'$ have the form
\beq\label{Bufull-10-1}
((1-p_i)+p_i\otimes z)((1-q_i)+q_i\otimes z^*),\,\,\, i=1,2,...,k.
\eneq

We may further assume that $p_i\otimes z, q_i\otimes z \in {\cal G}_1,$ $i=1,2,...,k.$
Choose $\dt_3>0$ and a finite subset ${\cal G}_4'\subset A$ (and write  ${\cal G}_4:=\{g\otimes f: g\in {\cal G}_4', f=1,z,z^*\}$) such that, for any two
unital ${\cal G}_4$-$\dt_3$-multiplicative
\cp s $\Psi_1, \Psi_2: A\otimes C(\T)\to C$ (any unital \CA\, $C$),
any ${\cal G}_4'$-$\dt_3$-multiplicative \morp\, $\Psi_0: A \to C$ and unitary $V\in C$ ($1\le i\le k$), { {if
\beq\label{Bf-10+1}
&&\|\Psi_0(g)-\Psi_1(g\otimes 1)\|<\dt_3\rforal g\in {\cal G}_4',\\\label{Bf-10+2}
 &&\|\Psi_1(z)-V\|<\dt_3,\andeqn
\|\Psi_1(g)-\Psi_2(g)\|<\dt_3\rforal g\in {\cal G}_4,
\eneq then}}
\beq\label{Bufull-10-1+}
&&\langle (1-\Psi_0(p_i)+\Psi_0(p_i)V)(1-\Psi_0(q_i)+\Psi_0(q_i) V^*)\rangle\\
&&\hspace{0.3in}\approx_{\sigma_2/2^{10}}\langle \Psi_1(((1-p_i)+p_i\otimes z)((1-q_i)+q_i\otimes z^*))\rangle,\\
&&\|\langle \Psi_1(x)\rangle -\langle \Psi_2(x)\rangle\|<{{\sigma_2/2^{10}}}\rforal x\in {\cal U}_2',\andeqn\\
&&\Psi_1(((1-p_i)+p_i\otimes z)((1-q_i)+q_i\otimes z^*))\\\label{Bf-1510}
&&\hspace{0.3in}\approx_{\sigma_2/2^{10}}
\Psi_1((1-p_i)+p_i\otimes z)\Psi_1((1-q_i)+q_i\otimes z^*),
\eneq
and,  furthermore, for $d_i^{(1)}=p_i,$ $d_i^{(2)}=q_i,$ there are projections ${\bar d}_i^{(j)}\in C$ and unitaries  ${\bar z}^{(j)}_i\in {\bar d}_i^{(j)}C {\bar d}_i^{(j)}$  such that
\beq\label{Bufull-10-1++}
&&\Psi_1(((1-d_i^{(j)})+d_i^{(j)}\otimes z))\approx_{{\sigma_2\over{2^{12}}}}(1-{\bar d}_i^{(j)})+{\bar z}^{(j)}_i,\\\label{Bf-15107+1-1}
&&{\bar d}_i^{(j)}\approx_{\sigma_2\over{2^{12}}} \Psi_1(d_i^{(j)}),\,\,{\bar z}_i^{(1)}\approx_{\sigma_2\over{2^{12}}} \Psi_1(p_i\otimes z),\andeqn
{\bar z}_i^{(2)}\approx_{\sigma_2\over{2^{12}}}  \Psi_1(q_i\otimes z^*),
\eneq
where
$1\le i\le k,$\, $j=1,2.$
Choose $\sigma>0$ so it is smaller than $\min\{\sigma_1/16, \ep/16, \sigma_2/16, \dt_2/16, \dt_3/16\}.$

Choose $\dt_5>0$ and a finite subset ${\cal G}_5\subset A$ satisfying the following condition:
there is a unital  ${{{\cal G}_4}}$-$\sigma/8$-multiplicative \cp\,
$L: A\otimes C(\T)\to B'$ such that
\beq\label{Bufull-10}
&&\|L(a\otimes 1)-\phi'(a)\|<\sigma/8
\tforal a\in {{{\cal G}_4'}}\andeqn
\|L(1\otimes z)-u'\|<\sigma/8
\eneq
for any unital \hm\, $\phi': A\to B'$ and any unitary $u'\in B'$ such that
$$
\|\phi'(g)u'-u'\phi'(g)\|<\dt_5\tforal  g\in {\cal G}_5.
$$
Let $\dt=\min\{\dt_5/4, \sigma\}$ and ${\cal G}={\cal G}_5\cup {\cal G}_4'\cup {\cal G}_2'.$

Now suppose that $\phi: A\to B$ is a unital \hm\, and $u\in CU(B)$
satisfies the assumptions (\ref{BHfull-1}) to (\ref{BHfull-1n1}) for the above mentioned
$\dt,$ $\sigma,$ ${\cal G},$  ${\cal P},$ $p_i,$ and $q_i.$
Choose an isomorphism $s: U_2\otimes U_2\to U_2.$
Note that $s\circ\imath$ (since it is unital) is approximately unitarily equivalent to the identity map on $U_2,$
where   ${\imath}: U_2\to U_2\otimes U_2$ is defined by $\imath(a)=a\otimes 1$ (for all $a\in U_2$).
To simplify notation,
let us assume that
$\phi(A)\subset B\otimes 1\subset B\otimes U_2.$
Suppose that $u\in U(B)\otimes 1_{U_2}$ is a unitary which satisfies the assumption of the lemma. {{As mentioned at the beginning,  we may assume that $u\in CU(B)\otimes 1_{U_2}.$
\Wlog, we may further assume that $u=\prod_{j=1}^{m_1}c_jd_jc_j^*d_j^*,$
where $c_j, d_j\in U(B)\otimes 1_{U_2},$ $1\le j\le m_1.$
Let ${\cal  F}_2= \{c_j, d_j: 1\le j\le m_1\}.$ }}

Let $L: A\otimes C(\T)\to B$ be a unital ${\cal G}_4$-$\dt_2/8$-multiplicative \cp\,
such that
\beq\label{BUfull11}
\|L(a\otimes 1)-\phi(a)\|<\sigma/8
\tforal a\in {{{\cal G}_4'}}\andeqn
\|L(1\otimes z)-u\|<\sigma/8.
\eneq
Since ${\rm Bott}(\phi, u)|_{\cal P}=0$, we may also assume
that
\beq\label{Bufull11+}
[L]|_{{\cal P}_1}=[\phi]|_{{\cal P}_1}\andeqn
[L]|_{\boldsymbol{\bt}({\cal P}_2)}=0.
\eneq

Since $B$ is in ${\cal B}_0,$ there is a projection $p\in B$ and a unital \SCA\,
$C\in {\cal C}_0$ with $1_C=p$ satisfying the following condition:
\beq\label{BUfull-12}
&&\|L(g)-[(1-p)L(g)(1-p)+L_1(g)]\|<{{\sigma^2/32(m_1+1)}}
\tforal g\in {{ {\cal G}_4}}\\\label{BUfull-12+}
&&{{\andeqn \|(1-p)x-x(1-p)\|<\sigma^2/32(m_1+1)
\rforal x\in {\cal F}_2,}}
\eneq
where $L_1: A{{\otimes C(\T)}}\to C$ is a unital ${{{\cal G}_4}}$-$\min\{\dt_2/8, \ep/8\}$-multiplicative
\cp,
\beq\label{BUfull-13}
\tau(1-p)<\min\{\sigma_1/16, \sigma_2/16\}\tforal \tau\in T(B),
\eneq
and,  using (\ref{BHfull-1n1}), (\ref{BHfull-1n2}), (\ref{BUfull11}), and (\ref{Bufull-10-1+}) to
(\ref{Bf-1510})
we have that
\beq\label{Bufull-14}
&&{\rm dist}(L_2^{\ddag}(x), {\bar 1})<\sigma_2/4\tforal x\in \{1\otimes {\bar z}\}\cup {\cal U}_2\andeqn \\
&&{\rm dist}(L_2^{\ddag}(x), \overline{\phi(x')\otimes 1_{C(\T)}})<\sigma_2/4
\tforal  x\in {\cal U}_1,
\eneq
where $\overline{x'\otimes 1_{C(\T)}}=x$ and
$L_2(a)=(1-p)L(a)(1-p)+L_1(a)$ for all $a\in A\otimes C(\T).$
Note that  we also have
\beq\label{Bufull-15}
\|\phi(g)-L_2(g\otimes 1)\|<\sigma/2
\rforal g\in {{{\cal G}_4'}}\andeqn
[L_2|_A]|_{{\cal P}_1}=[\phi]|_{{\cal P}_1}.
\eneq
By  (\ref{BUfull-12+}) and the choice of ${\cal F}_2,$ there are a unitary $v_0\in CU(C)$ and a unitary\\ $v_{00}\in CU((1-p)B(1-p))$
 such that
\beq\label{Bufull-17}
&&\|L_1(1\otimes z)-v_0\|<\min\{\dt_2/2, \ep/8\}\andeqn \\\label{Bufull-17+}
&&\|(1-p)L(1\otimes z)(1-p)-v_{00}\|<\min\{\dt_2/2, \ep/8\}.
\eneq

By the choice of $\dt_2$ and ${\cal G}_3,$ applying Corollary \ref{densesp3}, we obtain  a unitary $w\in U(U_2)$
for example) such that
\beq\label{nBufull-18}
&&|t(L_3'(g)))-t(L_1'(g))|<\sigma_1/128,\quad g\in\mathcal H_2, 
\andeqn\\ \label{nBufull-18-+}
&&t(g(p\otimes w))\ge 
\frac{1}{2(1-\sigma_1/2)}\int_{\T} g dm\tforal g\in {\cal H}_1
\eneq
for all $t\in T(pBp\otimes U_2),$
where
$L_1', L_3': A\otimes C(\T)\to pBp\otimes U_2$ is the unital ${\cal G}_1$-$\dt_1/4$-multiplicative completely positive linear map defined by
\beq\label{Bufull-18-}
&&L_1'(a\otimes 1)=L_1(a\otimes 1)\otimes 1_{U_2},\,\, L_1'(a\otimes z^j)=L_1(a\otimes 1)\otimes w^j,\\
&&L_3'(a\otimes 1)=L_1(a\otimes 1)\otimes 1_{U_2} \andeqn
L_3'(a\otimes z^j)=L_1(a\otimes z^j)\otimes (w)^j
\eneq
for all $a\in A$ and all integers $j$ as given by Lemma \ref{Defuv}.

Let $L_3'': A\otimes C(\T)\to B\otimes U_2$ be defined by
$L_3''(a\otimes z^j)=(1-p)L(a\otimes z^j)(1-p)\otimes w^j$
for all $a\in A$ and for all $j\in \Z$ as described in  Lemma \ref{Defuv} which is ${\cal G}_1$-$\dt_1/4$-multiplicative completely positive linear map.

Define  $L_3: A\otimes C(\T)\to B\otimes U_2$ by $L_3(c)=L_3'(c)+L_3''(c)$
for all $c\in A\otimes C(\T).$  Thus 
$L_3(a\otimes z^j)=L_2(a\otimes z^j)\otimes w^j$ for all $a\in A$ and integer $j\in \Z.$
Define 
$\Phi: A\otimes C(\T)\to B\otimes U_2$   by
$\Phi(a\otimes 1)=\phi(a)\subset B\otimes 1_{U_2}$ for all $a\in A$ and
$\Phi(1\otimes f)=1_A\otimes f(w)$
for all $f\in C(\T)$  and for some $\lambda\in \T.$

One then checks that, for all $t\in T(B\otimes U_2),$
\beq\label{Bufull-18}
&&|t(L_3(g)))-t(\Phi(g))|<\sigma_1,\quad g\in\mathcal H_2, 
\andeqn\\ \label{Bufull-18-+}
&&t(g(1\otimes w))\ge 
\frac{1}{2}\int_{\T} g dm\tforal g\in {\cal H}_1.
\eneq


Note that $CU(U_2)=U(U_2)$ (see Theorem 4.1 of \cite{GLX}).
It is also known (by working in matrices, for example)  that there is a continuous
path of unitaries in $CU(U_2)$ connecting $1_{U_2}$ to $w$ with length
no more than $\pi+\ep/256.$ Therefore
one obtains a continuous path of unitaries
$\{v(t): t\in [1/4, 1/2]\}\subset CU(U_2)$ such that
\beq\label{Bufull-1409}
v(1/4)=1_{U_2},\,\,\, v(1/2)=w,\andeqn
{\rm length}(\{v(t): t\in [1/4, 1/2]\})\le \pi+\ep/256.
\eneq
Note that $\phi(a)\Phi(1\otimes z)=\Phi(1\otimes z)\phi(a)$
for all $a\in A.$ So, in particular, $\Phi$ is a unital \hm\, and
\beq\label{Bf-moved}
[\Phi]|_{\boldsymbol{\bt}(\underline{K}(A))}=0.
\eneq
Define a unital \cp\,
$L_t: A_2\to C([2,3], B\otimes U_2)$ by
$$
L_t(f\otimes 1)=L_2(f\otimes 1)\andeqn L_t(a\otimes z^j)=L_2(a\otimes z^j)\otimes (v((t-2)/4+1/4))^j
$$
for all $a\in A$ and integers $j$ and $t\in [2,3]$.
Note that $L_t(1\otimes z)\approx_{\min(\dt_2/2,\ep/8)}(v_0\oplus v_{00})\otimes v((t-2)/4+1/4)),$
and,  since $v(s)\in CU(U_2),$
$L_t(1\otimes z) \in_{\min(\dt_2/2,\ep/8)} CU(B\otimes U_2)$ for all $t\in [2, 3].$
Note also that,
$L_t$ is ${\cal G}_1$-$\dt_1/4$-multiplicative.
Note that at $t=2,$ $L_t=L_2$ and at $t=3,$ $L_t=L_3.$
It follows that
\beq\label{Bufull-18+}
&&[L_3]|_{{\cal P}_1}=[L_2]|_{{\cal P}_1}=[\phi]|_{{\cal P}_1},\,\,\,
[L_3]|_{{\boldsymbol{\bt}}({\cal P}_2)}=0,\andeqn\\\label{Bufull-18++}
&&L_3^{\ddag}(x)=L_2^{\ddag}(x)\tforal x\in {\cal U}_1.
\eneq
If $v=(e\otimes z)+(1-e)$ for some projection $e\in A,$ then
\beq\label{Bufull-19-1}
L_3(v)=L_2(e\otimes z)\otimes w+L_2((1-e)).
\eneq
Since $w\in CU(U_2),$
one computes  from (\ref{Bf-1510}) 
 that
that, with $x=((1-p_i)+p_i\otimes z)((1-q_i)+q_i\otimes z^*),$
\beq\label{Bf-15107+1}
&&\hspace{-0.4in}{{\overline{\langle L_3(x)\rangle}\approx_{\sigma_2/2^{10}} \overline{({\bar z}_i^{(1)}\otimes w+(1-{\bar p}_i))({\bar z}_i^{(2)}\otimes w+(1-{\bar q}_i))}}}\\
&&\hspace{-0.2in}{{=\overline{({\bar z}_i^{(1)}+(1-{\bar p}_i))({\bar p}_i\otimes w+(1-{\bar p}_i)\otimes 1_{U_2})({\bar z}_i^{(2)}+(1-{\bar q}_i))
({\bar q}_i\otimes w+(1-{\bar q}_i))}}}\\
&&{{=\overline{({\bar z}_i^{(1)}+(1-{\bar p}_i))({\bar z}_i^{(2)}+(1-{\bar q}_i))}={\overline{\langle L_2(x)\rangle}}}},
\eneq
where ${\bar p}_i, {\bar q}_i, {\bar z}_i^{(1)}, {\bar z}_i^{(2)}$ are as above (see the lines
following (\ref{Bf-1510})),
with $\Psi_1$ replaced by $L_2.$  It follows that
\beq\label{Bufull-19-2}
{\rm dist}(L_3^{\ddag}(x), {\bar 1})<\sigma_2/2 \tforal x\in \{\overline{1\otimes z}\}\cup {\cal U}_2.
\eneq
Note that, since $w\in CU(U_2)$ and $\phi(q)\in B\otimes 1_{U_2},$
\beq\label{Bufull-19+3}
\Phi(q\otimes z+(1-q)\otimes 1)=
\phi(q)\otimes w+\phi(1-q)\in CU(B\otimes U_2)
\eneq
for any projection $q\in A.$
It follows  that
\beq\label{Bufull-19}
&&\Phi^{\ddag}(x)\in CU(B\otimes U_2)\tforal x\in \{\overline{1\otimes z}\}\cup {\cal U}_2.
\eneq
Therefore (see also (\ref{Bf-moved}))
\beq\label{Bufull-20}
[L_3]|_{{\cal P}}=[\Phi]|_{\cal P}\andeqn
{\rm dist}(\Phi^{\ddag}(x), L_3^{\ddag}(x))<\sigma_2\tforal x\in {\cal U}.
\eneq
It follows from \eqref{Bufull-18-+} that
\begin{equation}\label{Bufull-21}
\tau(\Phi(f))\ge \Delta(f),\quad f\in {\cal H}_1,\ \tau\in T(B\otimes U_2),
\end{equation}
and it follows from \eqref{Bufull-18} that
\begin{equation}
|\tau(\Phi(f))-\tau(L_3(f))|<\sigma_1,\quad f\in\mathcal H_2,\ \tau\in T(B\otimes U_2).
\end{equation}
Applying Theorem 12.11 of \cite{GLN-I}, 
we obtain a unitary $w_1\in B\otimes U_2$ such that
\beq\label{Bufull-22}
\|w_1^*\Phi(f)w_1-L_3(f)\|<\ep/4\tforal f\in {\cal F}_1.
\eneq
Since $w\in U_2,$ there is a continuous path of unitaries $\{w(t):t\in [3/4, 1]\}\subset CU(U_2)$
(recall that $CU(U_2)=U_0(U_2)$) such that
\beq\label{Bufull-23}
\hspace{-0.2in}w(3/4)=w,
\,\,\, w(1)=1_{U_2}\andeqn {\rm length}(\{w(t):t\in [3/4,1]\})\le \pi+\ep/256.
\eneq
Note that
\begin{equation}\label{Bufull-24}
\Phi(a\otimes 1)w(t)=w(t)\Phi(a\otimes 1)\tforal a\in A\andeqn t\in [3/4,1].
\end{equation}
It follows from (\ref{Bufull-22}) that there exists
a continuous path of unitaries $\{u(t): t\in [1/2, 3/4]\}\subset B\otimes U_2$ such that  (see \eqref{Bufull-17},
\eqref{Bufull-17+},
and \eqref{Bufull-18-})
\beq\label{Bufull-25}
u(1/2)=(v_{00}+(v_0))\otimes w,\,\,\,
u(3/4)=w_1^*\Phi(1\otimes z)w_1,\andeqn\\\label{Bufull-25+}
\|u(t)-u(1/2)\|<\ep/4\tforal t\in [1/2, 3/4].
\eneq
It follows from \eqref{BUfull11}, \eqref{Bufull-17},
and (\ref{Bufull-17+}) that
there exists a continuous path of unitaries
$\{u(t): t\in [0, 1/4]\}\subset B$ such that
\beq\label{Bufull-26}
u(0)=u,\,\,\,u(1/4)=v_{00}+v_0,\andeqn\\\label{Bufull-26+}
\|u(t)-u\|<\ep/4\tforal t\in [0, 1/4].
\eneq
Also, define $u(t)=(v_{00}\oplus v_0)\otimes v(t)$ for all $t\in [1/4, 1/2].$
It follows that
\beq\label{Bufull-26-1604}
\|\phi(g)u(t)-u(t)\phi(g)\|<\ep/4+\dt<5\ep/16 \rforal g\in {\cal G}.
\eneq
Then define
\beq\label{Bufull-27}
u(t)=w_1^*(p\otimes w(t)+(1-p)\otimes 1_{U_2})w_1\tforal t\in [3/4,1].
\eneq
Then $\{u(t): t\in [0,1]\}\subset B\otimes U_2$ is a continuous path of unitaries such that $u(0)=u$ and $u(1)=1.$ Moreover, by (\ref{Bufull-26+}), \eqref{Bufull-26-1604}, \eqref{Bufull-25}, \eqref{Bufull-25+},
 (\ref{Bufull-22}),
\eqref{Bufull-27}, \eqref{BUfull11},
(\ref{Bufull-1409}), and  (\ref{Bufull-23}),
\beq\label{Bufull-28}
\|\phi(f)u(t)-u(t)\phi(f)\|<\ep\tforal f\in {\cal F}\andeqn {\rm length}(\{u(t)\})\le 2\pi+\ep.
\eneq

\end{proof}

\begin{rem}\label{RBHfull}
Note, in the statement of Theorem \ref{BHfull}, if $[1_A]\in {\cal P}$ (as an element of $K_0(A)$),
by 2.14 of \cite{GLN-I}, 
condition \eqref{BHfull-1n} implies $[u]=0$ in $K_0(B).$  In other words, by making $[1_A]\in {\cal P},$
\eqref{BHfull-1n} implies $[u]=0.$

One also notices that if, for some $i,$ $p_i=1_A$ and $q_i=0,$ then \eqref{BHfull-1n1}
implies \eqref{BHfull-1n2}.  In fact, \eqref{BHfull-1n2} is redundant.
To see this, let $A$ be a unital simple separable amenable \CA\, with stable rank one.
Let $G_0\subset K_0(A)$ be a finitely generated subgroup containing $[1_A].$
Let $G_r=\rho_A(G_0).$ Then $\rho_A([1_A])\not=0$ {and} $G_r$ is a finitely generated free abelian
group. Then we may write $G_0=G_0\cap {\rm ker}\rho_A\oplus G_r',$
where $\rho_A(G_r')=G_r$ and $G_r'\cong G_r.$
Note that $G_0\cap {\rm ker}\rho_A$ is {{a}} finitely generated group.
We may  therefore  write $G_0\cap {\rm ker}\rho_A=G_{00}\oplus G_{01},$ where
$G_{00}$ is a torsion group and $G_{01}$ is free abelian.
Note that $G_{01}\oplus G_r'$ is  free abelian.
Therefore  $G_0={\rm Tor}(G_0)\oplus F,$ where $F$ is a finitely generated free abelian subgroup.
Note that there is an integer $m\ge 1$ such that $m[1_A]\in F.$
Let $z\in C(\T)$ be the standard unitary generator. Consider $A\otimes C(\T).$
Then $\boldsymbol{\bt}(G_0)\subset \boldsymbol{\bt}(K_0(A))$ is a subgroup
of $K_1(A\otimes C(\T)).$  Moreover, $\boldsymbol{\bt}([1_A])$ may be identified with
$[1\otimes z].$

If we choose ${\cal U}_2$ in the proof of \ref{BHfull} to generate $\boldsymbol{\bt}(F),$ then
$m[1_A]$ is in the subgroup generated by
$\{[p_i]-[q_i]: 1\le i\le k\}$ (see the last paragraph).
Thus,
 for any $\sigma_1>0,$ we may assume that
\beq\label{RBH-1}
{\rm dist}({\overline{ u^m}}, {\bar 1})<\sigma_1
\eneq
provided that (\ref{BHfull-1n1}) holds for a sufficiently small $\sigma.$
Recall  that  $B$ has stable rank one (see  Theorem 9.7 of \cite{GLN-I})
and $u\in U_0(B)$ (see the beginning of this remark).
We may write $u=\exp(ih)v$ for some $h\in B_{s.a.}$ {{and $v\in CU(B)$.}}
Recall, in this case,  $U_0(B)/CU(B)=\Aff(T(B))/\overline{\rho_B(K_0(B))},$
where $\overline{\rho_B(K_0(B))}$ is a closed vector subspace of $\Aff(T(B))$ (see the proof
of Lemma 11.5 of \cite{GLN-I}).
The image of $\overline{u^m}$ in $\Aff(T(B))/\overline{\rho_B(K_0(B))}$
is the same as $m$ times the image of $\overline{u}$
in $\Aff(T(B))/\overline{\rho_B(K_0(B))}.$
It follows
from (\ref{RBH-1}) that
\beq\label{RBH-2}
{\rm dist}({\bar u}, {\bar 1})<\sigma_1.
\eneq
This implies that (with sufficiently small $\sigma$) the condition (\ref{BHfull-1n2}) is redundant
and therefore can be  {{omitted.}}

\end{rem}


\section{Stable results}

\begin{lem}\label{stablehomtp}
Let $C$ be a unital amenable separable \CA\,  which
is residually finite dimensional and satisfies the UCT.  For any $\ep>0,$ any finite subset ${\cal F}\subset C,$ any finite subset ${\cal P}\subset \underline{K}(C),$ any unital \hm\, $h: C\to A,$ where $A$ is any unital \CA, and any ${\kappa}\in {\rm Hom}_{\Lambda}(\underline{K}({SC}), \underline{K}(A)),$ there exists an integer $N\ge 1,$ a unital \hm\,
$h_0: C\to M_N(\C)\subset M_N(A)$,  and  a unitary $u\in U(M_{N+1}(A))$ such that
\beq\label{sthomp-1}
\|H(c),\, u]\|<\ep\tforal c\in {\cal F}\quad\textrm{and}\quad {\rm Bott}(H,\, u)|_{\cal P}=\kappa\circ{\bf \beta}|_{\cal P},
\eneq
where $H(c)={\rm diag}(h(c), h_0(c))$ for all $c\in C.$

\end{lem}

\begin{proof}
Define $S=\{z, 1_{C(\T)}\},$ where $z$ is  the identity function on the unit circle. Define $x\in {\rm Hom}_{\Lambda}(\underline{K}(C\otimes C(\T)), \underline{K}(A))$ as follows:
\beq\label{sthomp-2}
x|_{\underline{K}(C)}=[h]\andeqn x|_{\boldsymbol{\bt}(\underline{K}(C))}=\kappa.
\eneq
Fix a finite subset ${\cal P}_1\subset \boldsymbol{\bt}(\underline{K}(C)).$
Choose $\ep_1>0$ and a finite subset ${\cal F}_1\subset C$ satisfying the following condition:
\beq\label{sthomp-3}
[L']|_{{\cal P}_1}=[L'']|_{{\cal P}_1}
\eneq
for any pair of (${\cal F}_1\otimes S)$-$\ep_1$-multiplicative \morp s $L',L'':C\otimes C(\T)\to B$ (for any unital \CA\, $B$), whenever
\beq\label{sthomtp-4}
L'\approx_{\ep_1} L''\,\,\,{\rm on} \,\, {\cal F}_1\otimes S.
\eneq

Let a positive number $\ep>0,$  a finite subset ${\cal F}$  and a finite subset ${\cal P}\subset \underline{K}(C)$ be given.
We may assume, without loss of generality,  that
\beq\label{sthomtp-4+1}
{\rm Bott}(H',\, u')|_{\cal P}={\rm Bott}(H',\, u'')|_{\cal P}
\eneq
whenever $\|u'-u''\|<\ep$ for any unital \hm\, {{$H'$}} from $C.$
Put
$\ep_2=\min\{\ep/2, \ep_1/2\}$ and ${\cal F}_2={\cal F}\cup {\cal F}_1$
(choosing ${\cal P}_1=\boldsymbol{\bt}({\cal P})$ above).

Let $\dt>0,$ a finite subset ${\cal G}\subset C$, and a finite subset ${\cal P}_0\subset
\underline{K}(C)$ (in place of ${\cal P}$) be as provided by Lemma 4.17 of \cite{GLN-I} 
for $\ep_2/2$ (in place of $\ep$) and ${\cal F}_2$ (in place of ${\cal F}$).
We may assume that ${\cal F}_2$ and ${\cal G}$ are in the unit ball of $C$ and $\dt<\min\{1/2, \ep_2/16\}.$
Fix another finite subset ${\cal P}_2\subset \underline{K}(C)$ and define
${\cal P}_3={\cal P}_0\cup {\boldsymbol{\bt}}({\cal P}_2)$ (as a subset of
$\underline{K}(C\otimes C(\T))$). We may assume that ${\cal P}_1\subset
{\boldsymbol{\bt}}({\cal P}_2).$

It follows from Theorem 18.2 of \cite{GLN-I} 
that there are integers
$k_1, k_2,...,k_m$ and $K_1,$ a \hm\,
$h_1': C\otimes C(\T)\to \bigoplus_{j=1}^m M_{k_j}(\C)\to M_K(A),$  and a
unital $({\cal G}\otimes S)$-$\dt/2$-multiplicative \cp\,
$L': C\otimes C(\T)\to M_{K+1}(A)$ such that
\beq\label{sthomtp-5-}
[L']|_{{\cal P}_3}=(x+[h_1'])|_{{\cal P}_3}.
\eneq
Write $h_1'=\oplus_{j=1}^m H_j',$
where $H_j'=\psi_j\circ \pi_j,$ $\pi_j: A\otimes C(\T)\to M_{k_j}(\C)$ is a finite dimensional
representation and $\psi_j: M_{k_j}(\C)\to M_K(A)$ is a \hm.
Let $e_j$ be a minimal projection of $M_{k_j}(\C)$ and $q_j=\psi_j(e_j)\in M_{K}(A),$
and $Q_j=\psi_j(1_{M_{k_j}(\C)})\in M_{K}(A).$
Set $p_j=1_{M_K(A)}-q_j,$ $j=1,2,...,m.$ Then $M_{k_jK}(\C)$ can be identified with $M_{k_j}(M_K(A))=M_{k_j}\big((q_j\oplus p_j)M_K(A)(q_j\oplus p_j)\big)$ (since $q_j+p_j=1_{M_K(A)}$) in such a way that $Q_j(M_K(A))Q_j$ is identified with $M_{k_j}(q_j M_K(A)q_j)$. Define $\psi_j': M_{k_j}(\C)\to M_{k_j}(\C\cdot 1_{M_{K}(A)})\subset  M_{k_jK}(A)$
by sending $e_j$ to $p_j.$ Define $H_j'': C\otimes C(\T)\to M_{k_jK}(A)$
by $H_j''(c)=\psi_j\circ \pi_j(c)\oplus \psi_j'\circ \pi_j$ (conjugating a  unitary).
Note we require $H_j''$ maps into the scalar matrices of $M_{k_jK}(A).$
Let $H'=\oplus_{j=1}^m \psi_j': C\otimes C(\T)\to M_{k_j}(p_jM_K(A)p_j) \subset M_{(\sum_j k_j)K}(A)$ (conjugating a suitable unitary).
Let $N_1=(\sum_{j=1}^m k_j)K.$
Define $h_1=h_1'\oplus H'$ and $L=L'\oplus H'.$ Then $h_1$ maps $C\otimes C(\T)$ into $M_{N_1}(\C\cdot 1_A)\subset M_{N_1}(A).$

In other words,
 there are an integer $N_1\ge 1,$ a unital \hm\, $h_1':C\otimes C(\T)\to M_{N_1}(\C)\subset M_{N_1}(A)$, and a
unital $({\cal G}\otimes S)$-$\dt/2$-multiplicative \cp\,
$L: C\otimes C(\T)\to M_{N_1+1}(A)$ such that
\beq\label{sthomtp-5}
[L]|_{{\cal P}_3}=(x+[h_1])|_{{\cal P}_3}.
\eneq

We may assume that there is a unitary $v_0\in M_{N_1+1}(A)$ such that
\beq\label{sthomtp-5+}
\|L(1\otimes z)-v_0\|<\ep_2/2.
\eneq

Define $H_1: C\to M_{N_1+1}(A)$ by
\beq\label{sthomtp-6}
H_1(c)=h(c)\oplus h_1(c\otimes 1)\tforal c\in C.
\eneq
Define $L_1: C\to M_{N_1+1}(A)$ by $L_1(c)=L(c\otimes 1)$ for all $c\in C.$ Note that
\beq\label{sthomtp-7}
[L_1]|_{{\cal P}_0}=[H_1]|_{{\cal P}_0}.
\eneq
It follows from Lemma 4.17 of \cite{GLN-I} 
that there exists an integer $N_2\ge 1,$ a unital \hm\, $h_2: C\to M_{N_2(N_1+1)}(\C){{\subset M_{N_2(N_1+1)}(A)}},$ and a unitary $W\in M_{(N_2+1)(1+N_1)}(A)$ such that
\beq\label{sthomtp-8}
W^*(L_1(c)\oplus h_2(c))W\approx_{\ep/4} H_1(c)\oplus h_2(c)\rforal c\in {\cal F}_2.
\eneq
Put $N=N_2(N_1+1)+N_1.$ Now define $h_0: C\to M_N(\C)$ and $H: C\to M_{N+1}(A)$ by
\beq\label{sthomtp-9}
h_0(c)=h_1(c\otimes 1)\oplus h_2(c)\andeqn H(c)=h(c)\oplus h_0(c)
\eneq
for all $c\in C.$
Define
$u=W^*(v_0\oplus 1_{M_{N_2(N_1+1)}})W.$
Then, by (\ref{sthomtp-8}), and as $L_1$ is $({\cal G}\otimes S)$-$\dt/2$- multiplicative, we have
\beq\label{sthomtp-11}
&&\hspace{-0.2in}\|[H(c),\, u]\| \le \|(H(c)-{\rm Ad}\, W\circ (L_1(c)\oplus h_2(c))) u]\|\\
&&+\|{\rm Ad}\, W\circ (L_1(c)\oplus h_2(c)),\, u]\|+
\|u(H(c)-{\rm Ad}\, W\circ (L_1(c)\oplus h_2(c)))\|\\
&&<\ep/4+\dt/2+\ep/4<\ep\rforal c\in {\cal F}_2.
\eneq
Define $L_2: C\to M_{N+1}(A)$ by $L_2(c)=L_1(c)\oplus h_2(c)$ for all $c\in C.$ Then, we compute that
\beq\label{sthomtp-12}
{\rm Bott}(H,\, u)|_{\cal P}&=& {\rm Bott}({\rm Ad}\, W\circ L_2,\,u)|_{\cal P}=
 {\rm Bott}(L_2,\, v_0\oplus 1_{M_{N_2(N_1+1)}})|_{\cal P}\\
&=& {\rm Bott}(L_1,\, v_0)|_{\cal P}+{\rm Bott}(h_2,\, 1_{M_{N_2(N_1+1)}})|_{\cal P}\\
&=& [L]|_{{\boldsymbol{\bt}({\cal P})}}+0=
(x+[h])|_{{\boldsymbol{\bt}}({\cal P})}=\kappa|_{\cal P}.
\eneq

\end{proof}

\begin{thm}\label{STHOM}
Let $C$ be a unital amenable separable \CA\,  which  is residually finite dimensional and satisfies the UCT.  For any $\ep>0$ and any finite subset ${\cal
F}\subset C,$ there are $\dt>0,$ a finite subset ${\cal G}\subset
C,$ and a finite subset ${\cal P}\subset \underline{K}(C)$
 satisfying the following condition:

Suppose that $A$ is a unital \CA, suppose that $h: C\to A$ is a unital
\hm\, and suppose that $u\in U(A)$ is a unitary such that
\beq\label{Shomp1}
\|[h(a), u]\|<\dt\tforal a\in {\cal G}\tand
{\rm{Bott}}(h,u)|_{\cal P}=0.
\eneq
Then there exist an integer $N\ge 1$, a unital \hm\, $H_0:C\otimes C(\T)\to M_N(\C)$ ($\subset M_N(A)$)  (with finite dimensional range), and a continuous path of
unitaries $\{U(t): t\in [0,1]\}$ in $M_{N+1}(A)$ such that
\beq\label{Shomp2}
U(0)=u',\,\,\, U(1)=1_{M_{N+1}(A)},\tand \|[h'(a),
U(t)]\|<\ep\tforal a\in {\cal F},
\eneq
where
$$
u'={\rm diag}(u, H_0(1\otimes z))
$$
and $h'(f)=h(f)\oplus H_0(f\otimes 1)$ for $f\in C,$
and $z\in C(\T)$ is the identity function on the unit circle.

Moreover,
\beq\label{Shomp2+}
{\rm Length}(\{U(t)\})\le \pi+\ep.
\eneq

\end{thm}

\begin{proof}
Let $\ep>0$ and ${\cal F}\subset C$  be given. We may assume that ${\cal F}$ is in the unit ball of
$C.$

Let $\dt_1>0,$  ${\cal G}_1\subset C\otimes C(\T),$ and ${\cal
P}_1\subset \underline{K}(C\otimes C(\T))$  be as provided by Lemma 4.17 of \cite{GLN-I}
for $\ep/4$ and ${\cal F}\otimes S.$  We may assume that ${\cal G}_1={\cal G}_1'\otimes S,$
where ${\cal G}_1'$ is in the unit ball of $C$ and
$S=\{1_{C(\T)}, z\}\subset C(\T).$ Moreover, we may assume that ${\cal P}_1={\cal P}_2\cup{\cal
P}_3,$ where ${\cal P}_2\subset \underline{K}(C)$ and ${\cal
P}_3\subset {\boldsymbol{\bt}}(\underline{K}(C)).$ Let ${\cal P}= {\cal P}_2 \cup \beta^{-1}({\cal P}_3) \subset \underline{K}(C)$. Furthermore, we
may assume that any $\dt_1$-${\cal G}_1$-multiplicative \morp\,
$L'$ from $C\otimes C(\T)$ to a unital \CA\, gives rise  to a well defined map $[L']|_{{\cal P}_1}.$

Let $\dt_2>0$ and a finite subset ${\cal G}_2\subset C$  be as provided by 2.8 of \cite{LnHomtp} for $\dt_1/2$ and ${\cal G}_1'$
above.

Let $\dt=\min\{\dt_2/2, \dt_1/2, \ep/2\}$ and ${\cal G}={\cal
F}\cup {\cal G}_2.$

Suppose that $h$ and $u$ satisfy the assumption with $\dt,$
${\cal G}$ and ${\cal P}$ as above.  Thus, by 2.8 of \cite{LnHomtp}, there is a
$\dt_1/2$-${\cal G}_1$-multiplicative \morp\, $L: C\otimes
C(\T)\to A$ such that
\beq\label{Shomp4}
&&\|L(f\otimes 1)-h(f)\|<\dt_1/2\rforal f\in {\cal G}_1'\andeqn\\
&& \|L(1\otimes z)-u\|<\dt_1/2.
\eneq

Define $y\in {\rm Hom}_{\Lambda}(\underline{K}(C\otimes C(\T)),
\underline{K}(A))$ as follows:
$$
y|_{\underline{K}(C)}=[h]|_{\underline{K}(C)}\andeqn
y|_{\boldsymbol{\bt}(\underline{K}(C))}=0.
$$
{It follows from ${\rm{Bott}}(h,u)|_{\cal P}=0$ that $[L]|_{\bt({\cal P})}=0$.}

Then
\beq\label{Shomp4+1}
[L]|_{{\cal P}_{ 1}}=y|_{{\cal P}_{ 1}}.
\eneq

Define $H: C\otimes C(\T)\to A$ by
$$
H(c\otimes g)=h(c)\cdot g(1)\cdot 1_A
$$
for all $c\in C$ and $g\in C(\T),$ where $\T$ refers to
the unit circle (and $1\in \T$).

It follows that
\beq\label{Shomp3}
[H]|_{{\cal P}_{ 1}}= y|_{{\cal P}_{ 1}}=[L]|_{{\cal P}_{ 1}}.
\eneq

It follows from Lemma 4.17 of \cite{GLN-I} 
that there are an integer $N\ge 1,$ a
unital \hm\, $H_0: C\otimes C(\T)\to M_N(\C)$ {{($\subset M_N(A)$)}} with finite
dimensional range, and a unitary $W\in U(M_{1+N}(A))$ such that
\beq\label{Shomp5}
W^*(H(c)\oplus H_{0}(c))W\approx_{\ep/4} L(c)\oplus H_{0}(c)\rforal c\in {\cal F}\otimes S.
\eneq

Since $H_0$ has finite dimensional range {{and since $H_0(1\otimes z)$ is in the center of range($H_0$) $\subset M_N(\C)$, }} it is easy to construct
a continuous path $\{V'{ (t)}: t\in [0,1]\}$ in a finite dimensional
\SCA\, of $M_N(\C)$  such that
\beq\label{Shomp6}
&&V'(0)=H_0(1\otimes z),\,\,\,
V'(1)=1_{M_{N}(A)}\andeqn\\
&&H_0(c\otimes 1)V'(t)=V'(t)H_0(c\otimes 1)
\eneq
for all $c\in C$ and $t\in [0,1].$ Moreover, we may ensure that
\beq\label{Shomp6+}
\text{Length}(\{V'(t)\})\le \pi.
\eneq

Now define $U(1/4+3t/4)=W^*{\rm diag}(1, V'(t))W$ for $t\in [0,1]$
and
$$
u'=u\oplus  H_0(1_A\otimes z)\andeqn h'(c)=h(c)\oplus H_0(c\otimes
1)
$$
for $c\in C$ for $t\in [0,1].$  Then, by (\ref{Shomp5}),
\beq\label{Shomp7}
\|u'-U(1/4)\|<\ep/4\andeqn \|[U(t),\, h'(a)]\|<\ep/4
\eneq
for all $a\in {\cal F}$ and $t\in [1/4,1].$ The desired conclusion follows by
connecting $U(1/4)$ with $u'$ with a short path as follows: There
is a self-adjoint element $a\in M_{1+N}(A)$ with $\|a\|\le {\ep
\pi/8}$ such that
\beq\label{Shomp-8}
\exp(i a)=u'U(1/4)^*
\eneq
Then the path of unitaries $U(t)=\exp(i (1-4t) a)U(1/4)$ for $t\in [0,1/4)$ satisfies the requirements.
\end{proof}

\begin{lem}\label{VuV}
Let $C$ be a unital separable \CA\,  whose irreducible representations have bounded dimension and let $B$ be a unital \CA\, with $T(B)\not=\emptyset.$
Suppose that $\phi_1, \, \phi_2:C\to B$ are two unital monomorphisms
such that
$$
[\phi_1]=[\phi_2]\,\,\,{\textrm in} \,\,\, KK(C,B),\\
$$
Let $\theta: \underline{K}(C)\to \underline{K}(M_{\phi_1, \phi_2})$ be the splitting map defined in Equation (e 2.46)
in Definition 2.21 of \cite{GLN-I}.

For any $1/2>\ep>0,$ any finite subset ${\cal F}\subset C$ and any
finite subset ${\cal P}\subset \underline{K}(C),$ there are
integers $N_1\ge 1,$  a unital $\ep/2$-${\cal F}$-multiplicative  \cp\,
$L: C\to M_{1+N_1}(M_{\phi_1,\phi_2}),$  a unital \hm\,
$h_0: C\to M_{N_1}(\C)$ (later, $M_{N_1}(\C)$ can be regarded as unital subalgebra of $M_{N_1}(B)$ and also of $M_{N_1}(M_{\phi_1,\phi_2})$),
and a continuous path of unitaries $\{V(t): t\in [0,1-d]\}$ in
$M_{1+N_1}(B)$ for some $1/2>d>0,$ such that $[L]|_{\cal P}$ is
well defined, $V(0)=1_{M_{1+N_1}(B)},$
\beq\label{VuV1}
[L]|_{\cal P}&=&(\theta+[h_0])|_{\cal P},\\
\label{VuV2}
\pi_t\circ L&\approx_{\ep}&{\rm Ad}\, V(t)\circ
(\phi_1\oplus h_0)\,\,\,\text{on}\,\,\,{\cal F} \tforal t\in (0,1-d],\\
\label{VuV3}
\pi_t\circ L&\approx_{\ep}&{\rm Ad}\, V(1-d)\circ
(\phi_1\oplus h_0)\,\,\,\text{on}\,\,\,{\cal F} \tforal t\in (1-d, 1],\tand\\
\label{VuV33}
\pi_1\circ L&\approx_{\ep}&\phi_2\oplus h_0\,\,\,\text{on}\,\,\,{\cal F},
\eneq
where $\pi_t: M_{\phi_1,\phi_2}\to B$ is the point evaluation at
$t\in (0,1).$

\end{lem}

\begin{proof}
Let $\ep>0$ and let ${\cal F}\subset C$ be a finite subset.
Let $\dt_1>0,$ a finite subset ${\cal G}_1\subset C$, and a finite subset ${\cal
P}\subset \underline{K}(C)$ be as provided by
\ref{STHOM} for $\ep/4$ and ${\cal F}$ above.
In particular, we assume that $\dt_1<\dt_{\cal P}$ (see Definition 2.14 of \cite{GLN-I}).
By
Lemma 2.15  of \cite{Lnclasn},
we may further assume that $\dt_1$ is sufficiently small  that
\beq\label{1510-n1}
{\rm Bott}(\Phi, U_1U_2U_3)|_{\cal P}=\sum_{i=1}^3{\rm Bott}(\Phi, U_i)|_{\cal P}
\eneq
whenever
$\|[\Phi(a), U_i]\|<\dt_1\rforal a\in {\cal G}_1,\,\,\,i=1,2,3.$

Let $\ep_1=\min\{\dt_1/2, \ep/16\}$ and ${\cal F}_1={\cal F}\cup
{\cal G}_1.$ We may assume that ${\cal F}_1$ is in the unit ball
of $C.$ We may also assume that $[L']|_{\cal P}$ is well defined
for any $\ep_1$-${\cal F}_1$-multiplicative \morp\, $L'$ from $C$ to
any unital C*-algebra.

Let $\dt_2>0$, ${\cal G}\subset C$, and ${\cal P}_1\subset \underline{K}(C)$ be a constant and finite subsets as provided by
Lemma 4.17 of \cite{GLN-I}
for $\ep_1/2$ and ${\cal F}_1.$ We may
assume that
$\dt_2<\ep_1/2,$ ${\cal G}\supset {\cal F}_1$, and ${\cal
P}_1\supset {\cal P}.$ We also assume that ${\cal G}$ is in the
unit ball of $C.$

It follows from Theorem 18.2 of \cite{GLN-I} 
that there exist an integer $K_1\ge
1,$ a unital \hm\, $h_0': C\to M_{K_1}(\C)$ (see also lines around \eqref{sthomtp-5}),
and a
$\dt_2/2$-${\cal G}$-multiplicative \morp\, $L_1:  C\to
M_{K_1+1}(M_{\phi_1, \phi_2})$ such that
\beq\label{VuV-1}
[L_1]|_{{\cal P}_1}=(\theta+[h_0'])|_{{\cal P}_1}.
\eneq

%
Note that  $[\pi_0]\circ \theta=[\phi_1]$ and $[\pi_1]\circ \theta =[\phi_2]$ and, for each $t\in (0,1),$
\beq\label{VuV-4}
[\pi_t]\circ \theta=[\phi_1]=[\phi_2].
\eneq
By Lemma 4.17 of \cite{GLN-I}, 
we obtain an integer $K_0,$ a unitary $V\in
U(M_{1+K_1+K_0}((C)))$, and a unital \hm\, $h: C\to
M_{K_0}(\C)$ such that
\beq\label{VuV-5--1}
{\rm Ad}\, V\circ (\pi_e\circ L_1\oplus h)\approx_{\ep_1/2}
(\id\oplus h_0'\oplus h)\,\,\,\text{on}\,\,\,{\cal F}_1,
\eneq
where $\pi_e: M_{\phi_1,\phi_2} \to C$ is the canonical projection.

(Here and below, we will identify a homomorphism mapping to $M_k(\C)$ with a homomorphism to $M_k(A)$ for any unital $C^*$ algebra $A$, without introducing new notation.)

Write $V_{00}=\phi_1(V)$ and
$V_{00}'=\phi_2(V).$ The assumption that
$[\phi_1]=[\phi_2]$ implies that $[V_{00}]=[V_{00}']$ in $K_1(B).$
By adding another homomorphism to $h$ in (\ref{VuV-5--1}), replacing $K_0$ by $2K_0$,  and replacing $V$
by $V\oplus 1_{M_{K_0}},$ if necessary, we may assume that
$V_{00}$ and $V_{00}'$ are in the same connected component of
$U(M_{1+K_1+K_0}(B)).$ (Note that $[V_{00}]=[V'_{00}]$.)

One obtains a continuous path of unitaries $\{Z(t): t\in [0,1]\}$
in $M_{1+K_1+K_0}(B)$ such that
\beq\label{VuV-5--2}
Z(0)=V_{00}\andeqn Z(1)=V_{00}'.
\eneq
It follows that $Z\in M_{1+K_1+K_0}(M_{\phi_1,\phi_2}).$ By
replacing $L_1$ by ${\rm ad}\, Z\circ (L_1\oplus h)$ and using a
new $h_0',$ we may assume that
\beq\label{VuV-5--3}
\pi_0\circ L_1\approx_{\ep_1/2}\phi_1\oplus
h_0'\,\,\,\text{on}\,\,\,{\cal F}_1\andeqn
\pi_1\circ L_1\approx_{\ep_1/2} \phi_2\oplus
h_0'\,\,\,\text{on}\,\,\, {\cal F}_1.
\eneq

Define $\lambda: C\to M_{1+K_1+K_0}(C)$ by $\lambda(c)={\rm diag}(c, h_0'(c)),$ where
we also identify $M_{K_0+K_1}(\C)$ with the scalar matrices in $M_{K_0+K_1}(C).$
In particular, since $\phi_i$ is  unital, $\phi_i\otimes {\rm id}_{M_{K_1+K_0}}$ is the identity on  $M_{K_0+K_1}(\C),$
$i=1,2.$ {{Consequently, $(\phi_i\otimes {\rm id}_{M_{K_0+K_1}})\circ h_0'=h_0'$.}}
Therefore, one may write
\begin{equation*}
\phi_i(c)\oplus h_0'(c)=(\phi_i\otimes {\rm id}_{M_{K_0+K_1{{+1}}}})\circ \lambda(c)\rforal c\in C.
\end{equation*}

There is a partition
$0=t_0<t_1<\cdots <t_n=1$
such that
\beq\label{VuV-6}
\pi_{t_{i}}\circ L_1\approx_{\dt_2/8}\pi_t\circ
L_1\,\,\,\text{on}\,\,\,{\cal G} \rforal t_i\le t\le t_{i+1},\,\,\,i=1,2,...,n-1.
\eneq
Applying Lemma 4.17 of \cite{GLN-I} 
again, we obtain an integer $K_2\ge 1, $
a unital \hm\, $h_{00}: C\to M_{K_2}(\C),$
 and  a unitary $V_{t_i}\in M_{1+K_0+K_1+K_2}(B)$ such that
\beq\label{VuV-7}
{\rm Ad}\, V_{t_i}\circ (\phi_1\oplus h_0'\oplus
h_{00})\approx_{\ep_1/2} (\pi_{t_i}\circ L_1 \oplus
h_{00})\,\,\,\text{on}\,\,\, {\cal F}_1.
\eneq

Note that, by (\ref{VuV-6}), (\ref{VuV-7}), and (\ref{VuV-5--3}),
$$
\|[\phi_1\oplus h_0'\oplus h_{00}(a), V_{t_i}V_{t_{i+1}}^*]\|<\dt_2/4+\ep_1\rforal
a\in {\cal F}_1.
$$

Define  $\eta_{-1}=0$ and
$$
\eta_k=\sum_{i=0}^k {\rm Bott}(\phi_1\oplus h_0'\oplus h_{00}, V_{t_i}V_{t_{i+1}}^*)|_{\cal P},\,\,k=0,1,...,n-1.
$$

Now we will construct, for each $i,$ a unital \hm\, $F_i: C\to M_{J_i}(\C)\subset M_{J_i}(B)$ and
a unitary $W_i\in M_{1+K_0+K_1+K_2+\sum_{k=1}^iJ_i}(B)$ such that
\beq\label{150110-L2}
\|[H_i(a),\, W_i]\|<\dt_2/4\tforal a\in {\cal F}_1\andeqn {\rm Bott}(H_i,\, W_i)=\eta_{i-1},
\eneq
where $H_i=\phi_1\oplus h_0'\oplus h_{00}\oplus\bigoplus_{k=1}^i F_i,$ $i=1,2,...,n-1.$

Let $W_0=1_{M_{1+K_0+K_1+K_2}}.$
It follows from Lemma \ref{stablehomtp} that there are an integer $J_1\ge 1,$
a unital \hm\, $F_1: C\to M_{J_1}(\C)$, and a unitary $W_0\in
U(M_{1+K_0+K_1+K_2+J_1}(B))$ such that
\beq\label{VuV-8}
\|[H_1(a), \, W_1]\|<\dt_2/4\rforal a\in {\cal F}_1\andeqn
\text{Bott}(H_1, W_1)=\eta_0,
\eneq
where $H_1=\phi_1\oplus h_0'\oplus h_{00}\oplus F_1.$

Assume that we have constructed the required $F_i$ and $W_i$ for $i=0,1,...,k<n-1.$
It follows from Lemma \ref{stablehomtp} that there are an integer $J_{k+1}\ge 1,$
a unital \hm\, $F_{k+1}: C\to M_{J_{k+1}}(\C),$ and a unitary $W_{k+1}\in
U(M_{1+K_0+K_1+K_2+\sum_{i=1}^{k+1}J_i}(B))$ such that
\beq\label{VuV-9}
\|[H_{k+1}(a), \, W_{k+1}]\|<\dt_2/{ 4}\rforal a\in {\cal F}_1\andeqn
\text{Bott}(H_{k+1}, W_{k+1})=\eta_{k},
\eneq
where $H_{k+1}=\phi_1\oplus h_0'\oplus h_{00}\oplus \bigoplus_{i=1}^{k+1}F_i.$ This finished the construction of $F_i$, $W_i$ and $H_i$ for $i=0,1,..., n-1.$

Now define $F_{00}=h_{00}\oplus \bigoplus_{i=0}^{n-1}F_i$ and define
$K_3=1+K_0+K_1+K_2+\sum_{i=1}^{n-1}J_i.$
Define
$$
v_{t_k}= {\rm diag}(W_k{\rm diag}(V_{t_k}, {\rm id}_{1_{M_{\sum_{i=1}^kJ_i}}}), 1_{M_{\sum_{i=k+1}^{n-1} J_i}}),
$$
$k=1,2,...,n-1$ and
$v_{t_0}=1_{M_{1+K_0+K_1+K_2+\sum_{i=1}^{n-1}J_i}}.$
Then
\beq\label{VuV-14}
&&{\rm Ad}\, v_{t_i}\circ (\phi_1\oplus h_0'\oplus
F_{00})\approx_{\dt_2+\ep_1} \pi_{t_i}\circ (L_1\oplus
F_{00})\,\,\,\text{on}\,\,\,{\cal F}_1,\\
&&\|[\phi_1\oplus  h_0'\oplus F_{00}(a),\,v_{t_i}v_{t_{i+1}}^*]\|<\dt_2/2+2\ep_1\rforal a\in {\cal F}_1, \andeqn\\
\label{VuV-15}
&&\text{Bott}(\phi_1\oplus h_0'\oplus F_{00}, v_{t_i}v_{t_{i+1}}^*)\\
&=&\text{Bott}(\phi_1', W_i')+
\text{Bott}(\phi_1', V_{t_i}'(V_{t_{i+1}}')^*) +\text{Bott}(\phi_1', (W_{i+1}')^*)\\
&=& \eta_{i-1}+\text{Bott}(\phi_1', V_{t_i}V_{t_{i+1}}^*)-\eta_{i}=0,
\eneq
where $\phi'_{{1}}=\phi_1\oplus h_0'\oplus F_{00},$   $W_i'={\rm diag}(W_i, 1_{M_{\sum_{j=i+1}^{n-1}J_i}})$
and $V_{t_i}'={\rm diag}(V_{t_i}, 1_{M_{\sum_{i=1}^{n-1}J_i}}),$
$i=0,1,2,...,n-2.$

It follows by Lemma \ref{STHOM} that there are an integer $N_1\ge 1,$
a unital \hm\, $F_0': C\to M_{N_1}(\C)$, and a continuous path of
unitaries $\{w_i(t): t\in [t_{i-1}, t_i]\}$ in $M_{K_3}(B)$ such that
\beq\label{VuV-17}
&&
w_i(t_{i-1})=v_{i-1}'(v_{i}')^*, w_i(t_i)=1,
\andeqn\\
&&\|[\phi_1\oplus h_0'\oplus F_{00}\oplus F_{0}'(a), \,w_i(t)]\|<\ep/4\rforal a\in {\cal F},
\eneq
where $v_i'={\rm diag}(v_i, 1_{M_{N_1}}(B)),$
$i=1,2,...,n-1.$
Define $V(t)=w_i(t)v_i'$ for $t\in [t_{i-1}, t_i],$ $i=1,2,...,n-1.$
Then $V(t)\in C([0,t_{n-1}], M_{K_3+N_1}(B)).$ Moreover,
\beq\label{VuV-18}
{\rm Ad}\, V(t)\circ (\phi_1\oplus h_0'\oplus F_{00}\oplus
F_0')\approx_{\ep} \pi_t\circ L_1\oplus F_{00}\oplus F_0'
\,\,\,\text{on}\,\,\,{\cal F}.
\eneq

Define  $h_0=h_0'\oplus F_{00}\oplus F_0',$ $L=L_1\oplus
F_{00}+F_0'$, and $d=1-t_{n-1}.$ Then, by (\ref{VuV-18}),
(\ref{VuV2}) and (\ref{VuV3}) hold. From (\ref{VuV-5--3}), it follows that
(\ref{VuV33}) also holds.

\end{proof}

\section{Asymptotic unitary equivalence}

\begin{lem}\label{MUN2}
Let $C_1$ and {$A_1$} be two unital separable simple \CA s in {${\cal B}_1,$} let
$U_1$ and $U_2$ be two UHF-algebras of infinite type and
consider the \CA s $C=C_1\otimes U_1$ and $A=A_1\otimes U_2.$  Suppose that
$\phi_1, \phi_2: C\to A$ are two unital monomorphisms.
Suppose also that
\beq\label{MUN2-1}
[\phi_1]=[\phi_2]\,\,\,{\textrm in}\,\,\, KL(C,A),\\
(\phi_1)_T=(\phi_2)_T\tand \phi_1^{\ddag}=\phi_2^{\ddag}.
\eneq
Then $\phi_1$ and $\phi_2$ are approximately unitarily equivalent.
\end{lem}

\begin{proof}
This follows  immediately from Theorem 12.11 part (a) of \cite{GLN-I}. 
Note that both $A$ and $C$ are
in $\mathcal B_1.$
\end{proof}

\begin{lem}\label{botgroup}
Let $B$ be a unital \CA\, and let $u_1, u_2,...,u_n\subset U(B)$ be unitaries. Suppose that $v_1, v_2,...,v_m\subset U(B)$ are also unitaries
such that $[v_j]\subset G,$ $j=1, ..., m$, where $G$ is the subgroup of $K_1(B)$ generated by $[u_1], [u_2],...,[u_n].$ There exist  $\dt>0$  and a finite subset ${\cal F}\subset B$ satisfying the following condition:
For any unital \CA\, $A$ and any unital monomorphisms $\phi_1, \phi_2: B\to A${{,}}
if $\tau\circ \phi_1=\tau\circ \phi_2$ for all $\tau\in T(A)$ and
if there is a unitary $w\in U(B)$ such that
\beq\label{botgroup-1}
\| w^*\phi_1(b)w-\phi_2(b)\|<\dt \tforal b\in {\cal F},
\eneq
then there exists a group \hm\, $\af: G\to \Aff(T(A))$ such that
\beq\label{botgroup-2}
{1\over{2\pi i}}\tau(\log(\phi_2(u_k)w^*\phi_1(u_k^{ *})w)&=&\af([u_k]){{(\tau)}}\tand\\
{1\over{2\pi i}}\tau(\log(\phi_2(v_j)w^*\phi_1(v_j^{ *})w)&=&\af([v_j]){{(\tau)}},
\eneq
for any $\tau\in T(A),$ $k=1,2,...,n$ and $j=1,2,...,m.$
\end{lem}

\begin{proof}
The proof is essentially contained in the proofs of 6.1, 6.2, and 6.3 of \cite{Lnind}. Note that there is a typo in Lemma 6.2 and Lemma 6.3 in \cite{Lnind}: ``$\tau(\af(a))=a$" should be ``$\tau(\af(a))=\tau(a)$". Here the condition $\tau\circ \phi_1=\tau\circ \phi_2$ plays the role of condition $\tau(\af(a))=\tau(a)$ there.

\end{proof}

\begin{lem}\label{inv71}
Let $C_1$ be a unital simple \CA\, as in Theorem 14 .10 of \cite{GLN-I},
let $A_1$ be a unital separable simple \CA\, in {${\cal B}_0,$} and let $U_1$ and $U_2$ be UHF-algebras of infinite type. Let $C=C_1\otimes U_1$ and $A=A_1\otimes U_2.$ Suppose that $\phi_1, \phi_2: C\to A$ are unital monomorphisms. Suppose also that
\beq\label{71-1}
&&[\phi_1]=[\phi_2]\,\,\,{\textrm in}\,\,\, KL(C, A),\\\label{71-2}
&&\phi_1^{\ddag}=\phi_2^{\ddag}, \,\,\, (\phi_1)_T=(\phi_2)_T,\tand\\\label{71-3}
&&R_{\phi, \psi}(K_1(M_{\phi_1,\phi_2})){ \subset} \rho_A(K_0(A)).
\eneq
Then, for any increasing sequence of finite subsets $\{{\cal F}_n\}$ of $C$ whose union is dense in $C,$ any increasing sequence of finite
subsets ${\cal P}_n$ of $K_1(C)$ with
$\bigcup_{n=1}^{\infty} {\cal P}_n=K_1(C),$ and any decreasing sequence of positive numbers $\{\dt_n\}$ with $\sum_{n=1}^{\infty} \dt_n<\infty,$ there exists a sequence of unitaries $\{u_n\}$ in $U(A)$ such that
\beq\label{71-4}
{\rm Ad}u_n\circ \phi_1\approx_{\dt_n} \phi_2\,\,\,{\textrm on}\,\,\,{\cal F}_n\tand\\
\rho_A({\rm bott}_1(\phi_2, u_n^*u_{n+1})(x))=0\ \textrm{for all}\  x\in {\cal P}_n (\subset K_1(C))
\eneq
and for all sufficiently large $n.$

\end{lem}

\begin{proof}
Note that $A\cong A\otimes U_2.$ Therefore, as $U_2$ is of infinite type, there is a unital \hm\, $s: A\otimes U_2\to A$ such that
$s\circ \imath$ is approximately unitarily equivalent to the identity map on $A,$ where
 ${\imath}: A\to A\otimes U_2$ is defined by $a\to a\otimes 1_{U_2}$ {for all $a\in A$}.
 Therefore, we
may assume that $\phi_1(C),\, \phi_2(C)\subset A\otimes 1_{U_2}.$
By Lemma \ref{MUN2}, there exists a sequence of unitaries
$\{v_n\}\subset A$ such that
\beq\label{2019-sept-10}
\lim_{n\to\infty} {\rm Ad}\, v_n\circ \phi_1(c)=\phi_2(c)\tforal c\in C.
\eneq

We may assume that the set ${\cal F}_n$ are in the unit ball of $C,$ with dense union.
For the next four paragraphs of the proof, fix $n=1,2,....$

Put $\ep_n'=\min\{1/2^{n+1}, \dt_n/2\}.$ Let
$C_n\subset C$ be a unital \SCA\, (in place of ${ C_n}$) such that $K_i(C_n)$ is finitely generated ($i=0,1$),
and let ${\cal Q}_n$ be a finite set of generators of $K_1(C_n),$
 let $\dt_n'>0$ (in place of $\dt$) be as in Lemma \ref{Extbot2} for $C$ (in place of $A$), $\ep_n'$ (in
place of $\ep$), ${\cal F}_n$ (in place of ${\cal F}$), and $[\imath_n]({\cal Q}_{n-1})$
(in place of ${\cal P}$), where $\imath_n: C_n\to C$ is the embedding.
Note that we assume that
\beq\label{71-9}
[\imath_{n+1}]({\cal Q}_{n+1})\supset {\cal P}_{n+1}\cup [\imath_n]({\cal Q}_n).
\eneq

Write $K_1(C_n)=G_{n,f}\oplus {\rm Tor}(K_1(C_n)),$ where $G_{n,f}$ is a finitely generated free abelian group.
Let $z_{1,n}, z_{2,n},...,z_{f(n),n}$ be independent generators of $G_{n,f}$ and $z'_{1,n},z'_{2,n},...,z_{t(n),n}'$ be generators of ${\rm Tor}(K_1(C_n)).$
We may assume that
$$
{\cal Q}_n=\{z_{1,n}, z_{2,n},...,z_{f(n),n},z'_{1,n},z'_{2,n},...,z_{t(n),n}'\}.
$$

Choose $1/2>\ep_n''>0$ so that ${\text{bott}}_1(h', u')|_{K_1(C_n)}$ is a well defined group \hm, ${\text{bott}}_1(h', u')|_{{\cal Q}_n}$ is well defined, and
$(\text{bott}_1(h',u')|_{K_1(C_n)})|_{{\cal Q}_n}={\text{bott}}_1(h', u')|_{{\cal Q}_n}$ for any unital \hm\, $h': C\to A$ and any unitary $u'\in A$ for which
\beq\label{71-10}
\|[h'(c), u']\|<\ep_n''\tforal c\in {\cal G}_n'
\eneq
for some finite subset ${\cal G}_n'\subset C$ which contains ${\cal F}_n.$

Let $w_{1,n},w_{2,n},...,w_{f(n),n}, w_{1,n}',w_{2,n}',...,w_{t(n),n}'\in C$ be unitaries (note that, by
Theorem 9.7 of \cite{GLN-I},
$C$ has stable rank one)
such that $[w_{i,n}]=(\imath_n)_{*1}(z_{i,n})$ and
$[w_{j,n}']=(\imath_n)_{*1}(z_{j,n}'),$ $i=1,2,...,f(n),$
$j=1,2,...,t(n)$, and  $n=1,2,....$
Since we may choose larger ${\cal G}_n',$
without loss of generality, we may assume that $w_{i,n}\in {\cal G}_n'.$

Let $\dt_1''=1/2$ and, for $n\ge 2,$ let $\dt_n''>0$ (in place of $\dt$)
and ${\cal G}_n''$ (in place of ${\cal F}$) be as in Lemma \ref{botgroup}  associated
with $w_{1,n},w_{2,n},...,w_{f(n),n}, w_{1,n}',w_{2,n}',...,w_{t(n),n}'$
(in place of $u_1, u_2, ... ,u_n$) and
$$
\{w_{1,{n-1}}, w_{2,n-1},...,w_{f(n-1), n-1}, w_{1,n-1}',w_{2,n-1}',...,w_{t(n-1), n-1}'\}
$$
(in place of $v_1,v_2,...,v_m$).

Now consider all $n=1,2,....$
Put $\ep_n=\min\{\ep_n''/2, \ep_n'/2,\dt'_n, \dt_n''/2\}$ and
${\cal G}_n={\cal G}_n'\cup {\cal G}_n''.$
By (\ref{2019-sept-10}), we may assume that
\beq\label{71-11}
{\rm Ad}\, v_n\circ \phi_1\approx_{\ep_n} \phi_2\,\,\,{\rm on\,\,\,} {\cal G}_n,\,\,\, n=1,2,....
\eneq
Thus, ${\rm bott}_1(\phi_2\circ \imath_n, v_n^*v_{n+1})$ is well defined.
Since $\Aff(T(A))$ is torsion free,
\beq\label{71-12-1}
\tau\big({\rm bott}_1(\phi_2\circ \imath_n, v_n^*v_{n+1})|_{{\rm Tor}(K_1(C_n))}\big)=0.
\eneq

From (\ref{71-11}), we have
\beq\label{71-12}
\|\phi_2(w_{j,n}){\rm Ad}\, v_n(\phi_1(w_{j,n})^*)-1\|<(1/4)\sin(2\pi \ep_n)<\ep_n,\,\,\, n=1,2,....
\eneq
Define
\beq\label{71-13}
h_{j,n}={1\over{2\pi i}}\log(\phi_2(w_{j,n}){\rm Ad}\, v_n(\phi_1(w_{j,n})^*)),\,\,\,j=1,2,...,f(n),\  n=1,2,....
\eneq
Then, for any $\tau\in T(A),$
$|\tau(h_{j,n})|<\ep_n<\dt_n',\,\,\, j=1,2,...,f(n)$, $n=1,2,....$
Since $\Aff(T(A))$ is torsion free, and the classes $[w'_{j, n}]$ are torsion, it follows from Lemma
\ref{botgroup} that
\beq\label{71-13+}
\tau({1\over{2\pi i}}\log(\phi_2(w_{j,n}'){\rm Ad}\, v_n(\phi_1(w_{j,n}'^{*}))))=0,
\eneq
$j=1,2,...,t(n)$ and $n=1,2,....$
By the assumption that
$R_{\phi_1, \phi_2}({ K_1}(M_{\phi_1,\phi_2})){ \subset} \rho_A(K_0(A)),$ by
Exel's formula (see \cite{HL}), and by Lemma 3.5 of \cite{Linajm}, we conclude that
\begin{equation*}
\widehat{h_{j,n}}(\tau)=\tau(h_{j,n})\in R_{\phi_1, \phi_2}({ K_1}(M_{\phi_1,\phi_2})){ \subset} \rho_A(K_0(A)).
\end{equation*}
Now define $\af_n': K_1(C_n)\to \rho_A(K_0(A))$ by
\beq\label{71-15}
\af_n'(z_{j,n})(\tau)=\widehat{h_{j,n}}(\tau)=\tau(h_{j,n}),\,\,\, j=1,2,...,f(n)\andeqn
\af_n'(z_{j,n}')=0,\,\,\,j=1,2,...,t(n),
\eneq
$n=1,2,....$
 Since $\af_n'(K_1(C_n))$ is free abelian,
it follows that there is a \hm\, $\af_n^{(1)}: K_1(C_n)\to K_0(A)$ such that
\beq\label{71-16}
{{(}}\rho_A\circ \af_n^{(1)}(z_{j,n}))(\tau)=\tau(h_{j,n}),\,\,\, j=1,2,...,f(n),\,\,{{\tau\in T(A)}}, \,\,\andeqn\\
\af_{n}^{(1)}(z_{j,n}')=0,\,\,\, j=1,2,...,t(n).
\eneq
Define
$\af_n^{(0)}: K_0(C_n)\to K_1(A)$ by $\af_n^{(0)}=0.$ By the UCT,
there is $\kappa_n\in KL({ S}C_n, A)$ such that $\kappa_n|_{K_i(C_n)}=\af_n^{(i)},$ $i=0,1,$
where $SC_n$ is the suspension of $C_n$ (here, we identify $K_i(C_n)$ with $K_{i+1}(SC_n)$).

By the UCT again, there is $\af_n\in KL(C_n\otimes C(\T), A)$ such that
$\af_n\circ \boldsymbol{\bt}|_{\underline{K}(C_n)}=\kappa_n.$ In particular,
$\af_n\circ \boldsymbol{\bt}|_{K_1(C_n)}=\af_n^{(1)}.$ It follows from Lemma \ref{Extbot2} that there exists a unitary $U_n\in U_0(A)$ such that
\beq\label{71-17}
\|[\phi_{ 2}(c),\, U_n]\|<\ep_n''\rforal c\in {\cal F}_n\andeqn\\\label{71-17+}
\rho_A({\text{bott}_1}(\phi_2,\, U_n)(z_{j,n}))=-\rho_A\circ \af_n^{(1)}(z_{j,n}),
\eneq
$j=1,2,...,f(n).$ We also have
\beq\label{71-18}
\rho_A({\text{bott}_1}(\phi_2,\, U_n)(z_{j,n}'))=0,\,\,\, j=1,2,...,t(n),
\eneq
as the elements $z_{j, n}$ are torsion.
By the Exel trace formula (see \cite{HL}), (\ref{71-16}), and (\ref{71-17+}),
 we have
\beq\label{71-19}
\tau(h_{j,n})&=&-\rho_A({\rm bott}_1(\phi_2, U_n)(z_{j,n})(\tau)
    =-\tau({1\over{2\pi i}}\log(U_n\phi_2(w_{j,n})U_n^*\phi_2(w_{j,n}^*)))
\eneq
for all $\tau\in T(A),$ $j=1,2,...,f(n).$
Define $u_n=v_nU_n,$ $n=1,2,....$ By 6.1 of \cite{Lnind}, (\ref{71-19}), and
(\ref{71-17+}), we compute that
\beq\label{71-21}
&&\hspace{-0.2in}\tau({1\over{2\pi i}}\log(\phi_2(w_{j,n}){\rm Ad} u_n(\phi_{ 1}(w_{j,n}^*)))))\\
&=&\tau({1\over{2\pi i}}\log(U_n\phi_2(w_{j,n})U_n^*v_n^*\phi_1(w_{j,n}^*)v_n)))\\
&=&\tau({1\over{2\pi i}}\log(U_n\phi_2(w_{j,n})U_n^*\phi_2(w_{j,n}^*)\phi_2(w_{j,n})v_n^*
\phi_1(w_{j{{,n}}}^*)v_n)))\\
&=&\tau({1\over{2\pi i}}\log(U_n\phi_2(w_{j,n})U_n^*\phi_2(w_{j,n}^*))))
+\tau({1\over{2\pi i}}\log(\phi_2(w_{j,n})v_n^*\phi_1(w_{j,n}^*)v_n)))\\
&=&\rho_A({\rm bott}_1(\phi_2, U_n)(z_{j,n}))(\tau)+\tau(h_{j,n})=0
\eneq
for all $\tau\in T(A),$ $j=1,2,...,f(n)$ and $n=1,2,....$
By (\ref{71-13+}) and (\ref{71-18}),
\begin{equation}\label{71-22}
\tau({1\over{2\pi i}}\log(\phi_2(w_{j,n}'){\rm Ad} u_n(\phi_{1}((w_{j,n}')^*))))=0,
\end{equation}
$j=1,2,...,t(n)$ and $n=1,2,....$
Let
\beq\label{71-23}
b_{j,n}&=&{1\over{2\pi i}}\log(u_n\phi_2(w_{j,n})u_n^*\phi_1(w_{j,n}^*)),\\
b_{j,n}'&=&{1\over{2\pi i}}\log(\phi_2(w_{j,n})u_n^*u_{n+1}\phi_2(w_{j,n}^*)u_{n+1}^*u_n),\andeqn\\
b_{j,n+1}''&=&{1\over{2\pi i}}\log(u_{n+1}\phi_2(w_{j,n})u_{n+1}^*\phi_1(w_{j,n}^*)),
\eneq
$j=1,2,...,f(n)$ and $n=1,2,....$ We have, by (\ref{71-21}),
\beq\label{71-24}
\tau(b_{j,n})&=&\tau({1\over{2\pi i}}\log(u_n\phi_2(w_{j,n})u_n^*\phi_1(w_{j,n}^*)))\\
&=&\tau({1\over{2\pi i}}\log(\phi_2(w_{j,n})u_n^*\phi_1(w_{j,n}^*)u_n))=0
\eneq
for all $\tau\in T(A),$ $j=1,2,..., f(n)$, and $n=1,2,....$ Note that
$\tau(b_{j, n+1})=0$ for all $\tau\in T(A),$ $j=1,2,...,f(n+1).$
It follows from Lemma \ref{botgroup} {and (\ref{71-9})} that
\begin{equation*}
\tau(b_{j,n+1}'')=0\tforal \tau\in T(A),\,\,\,j=1,2,...,f(n),\,\,\,n=1,2,....
\end{equation*}
Note
that
\begin{equation*}
u_ne^{2\pi i b_{j,n}'}u_n^*=e^{2\pi i b_{j,n}}\cdot e^{-2\pi i b_{j,n+1}''},\,\,\,j=1,2,...,f(n).
\end{equation*}
Hence, using 6.1 of \cite{Lnind}, we compute that
\beq\label{71-26}
\tau(b_{j, n}')=\tau(b_{j,n})-\tau(b_{j,n+1}'')=0\tforal \tau\in T(A).
\eneq
By the Exel formula {{(see \cite{HL})}} and (\ref{71-26}),
\beq\label{71-27}
&&\hspace{-0.3in}\rho_A({\rm bott}_1(\phi_2, u_n^*u_{n+1}))(w_{j,n}^*)(\tau)
=\tau({1\over{2\pi i}}\log(u_n^*u_{n+1}\phi_2(w_{j,n})u_{n+1}^*u_n\phi_2(w_{j,n}^*)))\\
&=&\tau({1\over{2\pi i}}\log(\phi_2(w_{j,n})u_{n}^*u_{n+1}\phi_2(w_{j,n}^*)u_{n+1}^*u_n))=0
\eneq
for all $\tau\in T(A)$ and $j=1,2,...,f(n).$
Thus,
\beq\label{71-28}
\rho_A({\rm bott}_1(\phi_2,u_n^*u_{n+1})(w_{j,n}))(\tau)=0\tforal \tau\in T(A),
\eneq
$j=1,2,...,f(n)$, and $n=1,2,....$ We also have
\beq\label{71-29}
\rho_A({\rm bott}_1(\phi_2, u_n^*u_{n+1})(w_{j,n}'))(\tau)=0\tforal \tau\in T(A),
\eneq
$j=1,2,...,f(n)$, and $n=1,2,....$ By \ref{botgroup}, we have that
\beq\label{71-30}
\rho_A({\rm bott}_1(\phi_2, u_n^*u_{n+1})(z))=0\tforal z\in {\cal P}_n,
\eneq
$n=1,2,....$
\end{proof}

\begin{rem}\label{gamma}
Let $C$ be a unital separable amenable C*-algebra satisfying the UCT with finitely generated $K_i(C)$ ($i=0, 1$), let $A$ be a unital separable C*-algebra and let $\phi_1, \phi_2: C \to A$ be two unital homomorphisms. In what follows, we will continue to use $\phi_1$ and $\phi_2$ for the induced homomorphisms from $M_k(C)$ to $M_k(A)$. Suppose that $v\in U(A)$ and $$\|v^*\phi_1(a)v - \phi_2(a)\| < \ep < 1/8,\quad a\in\{z_1, z_2, ..., z_n\}\cup\mathcal F$$ for a finite subset $\mathcal F\subseteq M_k(C)$ and some $z_1, z_2, ..., z_n \in U(M_k(C))$ such that $[z_1], [z_2], ..., [z_n]$ generate $K_1(C)$. Define $W_j(t) \in U(M_2(C([0, 1], M_k(A))))$ as follows
$$W_j(t) = (T_tVT_t^{-1})^*\mathrm{diag}(\phi_1(z_j), 1_{M_k})T_tVT_t^{-1},
$$
where
$$V=\mathrm{diag}(v, 1_{M_k})\quad\mathrm{and}\quad T_t = \left( \begin{array}{cc} \cos(\frac{\pi}{2} t) & -\sin(\frac{\pi}{2} t) \\ \sin(\frac{\pi}{2} t) & \cos(\frac{\pi}{2} t)  \end{array} \right).$$

Note that $W_j(0) = \mathrm{diag}(v^*\phi_1(z_j)v, 1)$ and $W_j(1) = \mathrm{diag}(\phi_1(z_j), 1)$. Connecting $W_j(0)$ with $\mathrm{diag}(\phi_2(z), 1)$ by a {{continuous}} path, we obtain a continuous path of {{unitaries}} $Z_j(t)$ such that $Z_j(0) = \mathrm{diag}(\phi_2(z_j), 1)$, $Z(1/4) = W(0)$ and $Z_j(1) = \mathrm{diag}(\phi_1(z_j), 1)$ and $\|Z_j(t) - Z_j(1/4)\| < 1/8$ for $t\in [0, 1/4)$. Thus $Z_j \in M_{2k}(M_{\phi_1, \phi_2})$. 
With sufficiently small $\ep>0$, since $K_1(C)$ is finitely generated, the map $$K_1(C) \ni [z_j] \mapsto [Z_j] \in K_1(M_{\phi_1, \phi_2}),\quad j=1, 2, ..., n,$$ induces a homomorphism.

Set
\begin{equation}\label{define-h}
h_j = \frac{1}{i}\mathrm{diag}(\log(\phi_2(z_j)^*V^*\phi_1(z_j)V), 1),\quad j=1, 2, ..., n.
\end{equation}
We may specifically use
$$Z_j(t) = \mathrm{diag}(\phi_2(z_j), 1)\exp(i4th_j),\quad t\in [0, 1/4].$$

Still use $\phi_1$ and $\phi_2$ for the induced homomorphisms from $M_k(C \otimes C')$ to $M_k(A\otimes C')$, where $C'$ is a commutative C*-algebra $C'$ with finitely generated $K_i(C')$ ($i=0, 1$). Fix a finite set of unitaries $z_1, ..., z_n \in M_\infty(C\otimes C')$ which generates $K_1(C\otimes C')$. We also obtain a homomorphism $K_1(C\otimes C') \to K_1(M_{\phi_1, \phi_2}\otimes C')$ provided that $\ep$ is small. 

Let $\mathcal F \subset C$ be a finite subset and $\ep>0$. Suppose that there is a unitary $v\in U(A)$ such that $$\mathrm{ad}v\circ\phi_1 \approx_\ep \phi_2\quad{\mathrm{on}}\ \mathcal F.$$
Let $U'(t) = T_tVT_t^{-1}$. Define
\beq\label{913-1}
L(c)(t) = (U'(\frac{4t-1}{3}))^*\mathrm{diag}(\phi_1(c), 1)U'(\frac{4t-1}{3}),\quad t\in[1/4, 1]
\eneq
and
$$L(c)(t) = 4tL(c)(1/4) + (1 - 4t)\mathrm{diag}(\phi_2(c), 1),\quad t\in [0, 1/4].$$
Note that $L$ maps $C$ into $M_2(M_{\phi_1, \phi_2})$. Thus, since $K_i(C)$ ($i=0, 1$) is finitely generated, by Corollary 2.11 of \cite{DL}, there is $N_1>0$ such that any element of $\mathrm{Hom}_\Lambda(\underline{K}(C), \underline{K}(A))$ is determined by its restriction to $K_i(A, \mathbb Z/n\mathbb Z)$, $i=0, 1$, $n=0, 1, ..., N_1$. Hence, if $\ep$ is sufficiently small and $\mathcal F$ is sufficiently large, there is
\beq\label{LN-0926-1}
\gamma_{\phi_1,\phi_2, v} \in \mathrm{Hom}_\Lambda(\underline{K}(C), \underline{K}(M_{\phi_1, \phi_2}))
\eneq
such that
\beq\label{LN-0926-2}
[L]|_{\mathcal P} = \gamma_{\phi_1, \phi_2, v} |_{\mathcal P}
\eneq
for any given finite subset $\mathcal P \subset \underline{K}(C)$.

One computes that $$\int_0^1\tau(\frac{dZ_j(t)}{dt}Z_j(t))dt = \tau(h_j),\quad \tau \in T(A).$$ Therefore, if $R_{\phi_1, \phi_2}\circ\gamma_{\phi_1,\phi_2, v}(K_1(C)) = 0$, then
\beq\label{913-30}
\tau(h_j) = 0,\quad \tau\in T(A).
\eneq	

On the other hand, for any given $\eta>0$ and a finite set $\{z_1, z_2, ..., z_n\}$ of generators of $K_1(C)$, by \eqref{define-h},
\beq\label{913-11}
|\tau(h_i)| < \eta,\quad \tau\in T(A),
\eneq
provided that $\ep$ is sufficiently small and $\mathcal F$ is sufficiently large.

Now, assume that $\phi_1=\phi_2$. Then, with sufficiently large $\mathcal F$ and sufficiently small $\ep$, the element $\mathrm{Bott}(\phi_1, v): \underline{K}(C) \to \underline{K} (SA)$ is well defined.

We have the following splitting short exact sequence:
\begin{displaymath}
\xymatrix{
0\ar[r]&\mathrm{S}A\ar[r]^{\imath}&M_{\phi_1,\phi_1}\ar[r]^{\pi_e}&C
\ar[r]&0}.
\end{displaymath}
Define $\theta: C \to M_{\phi_1, \phi_1}$ by $\theta(b)={{\phi_1(b)}}$ as {a}
constant element in $M_{\phi_1, \phi_1}.$   Then $\theta$ may be identified with a
splitting map and $\underline{K}(M_{\phi_1, \phi_1})$ may be written
as $\underline{K}(SA){\oplus} \underline{K}(C).$

Let $P: \underline{K}(M_{\phi_1, \phi_1})\cong \underline{K}(C)\oplus \underline{K}(SA) \to\underline{K}(SA)$ be the standard projection map. One can verify that for any two elements $x, y \in \underline{K}(C)$ if $\mathrm{Bott}(\phi_1, v)(x)=\mathrm{Bott}(\phi_1, v)(y)$, then $P\circ\gamma_{\phi_1, \phi_1,v} (x)=P \circ\gamma_{\phi_1,\phi_1, v}(y)$.
So we will also use $\Gamma(\mathrm{Bott}(\phi_1, v))$ to denote the map $P\circ\gamma_{\phi_1, \phi_1, v} \in\mathrm{Hom}_\Lambda(\underline{K}(C), \underline{K}(SA))$. 

By shifting the index, we {see} $\Gamma({\rm Bott}(\phi_1,\, v))|_{\cal P}$
maps ${\cal P}$ to $\underline{K}(A).$
One may identify $P$ with ${\rm id}_{M_{\phi_1, \phi_1}}-[\theta]\circ [\pi_e].$
Note that
$$
\pi_e\circ \theta={\rm id}_{C} \andeqn \pi_e\circ L={\rm id}_{C}.
$$
So
\beq\label{913-33}
\Gamma({\rm Bott}(\phi_1, v))|_{\cal P}=
\gamma_{\phi_1, \phi_1,v}|_{\cal P}-\theta|_{\cal P}.
\eneq
Furthermore, it is shown in 10.6 of \cite{Linajm} that $\Gamma(\mathrm{Bott}(\phi_1, v))=0$ if and only if $\mathrm{Bott}(\phi_1, v)=0$.

Note that since the K-theory of $C$ is finitely generated, by Corollary 2.11 of \cite{DL}, one has that any element of $\mathrm{Hom}_\Lambda(\underline{K}(C), \underline{K}(A))$ is determined by its restriction to $K_i(A, \mathbb Z/n\mathbb Z)$, $i=0, 1$, $n=0, 1, ..., N_1$. Fix separable commutative C*-algebras $C_0=\mathbb C$, $C_1$, ..., $C_{N_1},$ $C_{N_1+1},$ ..., $C_{2N_1+1}$ with
$$K_0(C_n) = \mathbb Z/n\mathbb Z \quad\mathrm{and}\quad K_1(C_n) = \{0\},\quad n=0, 1, ..., N_1,$$
and $C_{N_1+i}=SC_{i-1},$ $i=1,2,...,N_1+1.$
For each $C\otimes C'$, where $C'$ is one of the $C_0, C_1, ..., C_{2N_1+1}$, fix a finite set of unitaries $z_1^{(n)}, z^{(n)}_2, ..., z^{(n)}_{k(n)}$ of  $M_{N_2}({\widetilde{C\otimes C'}})\subset M_{N_2}(C\otimes {\widetilde C'})$
(for some $N_2\ge 1$) which generates $K_1(C\otimes C_n),$ $n=0,1,...,2N_1+1.$
Let $C'_i=M_{N_2}(C_i),$ $i=0,1,...,2N_1+1.$

Let $1/4>\ep>0$ and $1/4>\eta>0.$
Choose $0<\dt<\ep/2$ sufficiently small and a finite set $\mathcal F\subseteq A$ sufficiently large such that if
$$u^*\phi_1u \approx_{\dt} \phi_2\quad\mathrm{on}\ \mathcal F$$
for some unitary $u\in A$,
then, for each $C_n,$
\beq\label{LN-0926}
u^*{\tilde \phi}_1u\approx_{\ep/2} {\tilde \phi}_2\,\, {\rm on}\,\, {\cal F}_{0,n},
\eneq
where ${\cal F}_{0,n}$ is a finite subset which contains $\{z_1^{(n)}, z_2^{(n)}, ..., z_{k(n)}^{(n)}\},$
$n=0,1,...,2N_1+1.$  We also assume that $\dt$ is sufficiently small and ${\cal F}$ is sufficiently large so
that \eqref{LN-0926-1}, \eqref{LN-0926-2}, \eqref{913-11}, \eqref{913-33} hold.

Suppose that there are  unitaries  $u_1, u_2\in A$ such  that
$$
u_i^*\phi_1u_i\approx_{\dt/2} \phi_2\,\,\,{\rm on}\,\,\, {\cal F},
$$
$i=1,2.$   Then, as in \eqref{LN-0926}, for each $C_n,$
\beq
u_i^*{\tilde \phi}_1u_i\approx_{\ep/2} {\tilde \phi}_2\,\, {\rm on}\,\, {\cal F}_{0,n},
\eneq
where ${\cal F}_{0,n}$ is a finite subset which contains $\{z_1^{(n)}, z_2^{(n)}, ..., z_{k(n)}^{(n)}\},$
$n=0,1,...,2N_1+1.$
Let $L_i: C\to M_{\phi_1, \phi_2}$ and $\gamma_{\phi_1, \phi_2, u_i} \in {\rm Hom}_{\Lambda}(\underline{K}(C),  \underline{K}(M_{\phi_1, \phi_2}){)}$
be the element defined by the pair $(\phi_1, u_i)$ ($i=1,2$) as above. 

On the other hand, one also has that
$$
u_2u_1^*\phi_1u_1u_2^*\approx_{\dt} \phi_1 \,\,\,{\rm on}\,\,\, {\cal F}.
$$
Note that  $\pi_e\circ (L_1-L_2)=0.$

Fix $n\in \{0,1,...,2N_1+1\}.$
Consider $z\in \{z_1^{(n)}, z_2^{(n)}, ..., z_{k(n)}^{(n)}\}$
and ${\tilde u}_i=u_i\otimes 1_{\widetilde{C_n'}},$ $i=1,2.$
We also write ${\tilde \phi}_i$ for $\phi_i\otimes {\rm id}_{{\widetilde{C_n'}}}.$

Define ${\tilde T}(t)=T_{2(t-1/4)}$ for $t\in [1/4, 3/4]$ and ${\tilde T}_t=T_1$ for $t\in [3/4,1].$
Let $W_i(t)=
{\tilde T}_t\left(\begin{array}{cc} \tilde{u}^*_i & 0 \\ 0 & 1 \end{array}\right){\tilde{T}}_t^*,$
$t\in [1/4, 1],$ and $W_i(t)=\diag(1,1)$ for $t\in [0,1/4],$
$i=0,1.$
Note
that
\beq
\|\diag({\tilde u}_1^*{\tilde \phi}(z){\tilde u}_1\tilde{u}_2^*{\tilde \phi}(z)^*{\tilde u}_2, 1)-\diag(1,1)\|<\ep<1/8.
\eneq
There is a continuous path $d(z)(t)$ (for $t\in [0,1/4]$)
such that
$d(z)(0)=\diag(1,1)$ and $d(z)(1/4)=\diag({\tilde u}_1^*{\tilde \phi(z)}{\tilde u}_1\tilde{u}_2^*{\tilde \phi}(z)^*{\tilde u}, 1)$ and
\beq
\|d(z)(t)-\diag(1,1)\|<\ep\rforal t\in [0,1/4].
\eneq
Define,
for $t\in [1/4, 1],$
\begin{eqnarray*}
d(z)(t)&=& \big(W_1(t)\diag({\tilde \phi}(z),1) W_1(t)^*\big)\big(W_2(t)\diag({\tilde \phi}(z)^*,1)W_2(t)^*\big)\\
& = &  {\tilde T}_t\left(\begin{array}{cc} \tilde{u}^*_1 & 0 \\ 0 & 1 \end{array}\right){\tilde T}_t^* \left(\begin{array}{cc} {\tilde \phi}_1(z) & 0 \\ 0 & 1 \end{array}\right)   {\tilde T}_t\left(\begin{array}{cc} \tilde{u}_1  \tilde{u}^*_2 & 0 \\ 0 & 1 \end{array}\right) {\tilde T}_t^* \left(\begin{array}{cc} {\tilde \phi}_1(z)^* & 0 \\ 0 & 1 \end{array}\right)  {\tilde T}_t\left(\begin{array}{cc} \tilde{u}_2 & 0 \\ 0 & 1 \end{array}\right){\tilde T}_t^*.
\end{eqnarray*}
On $[3/4,1],$ define $d(z)(t)=\diag(1,1).$
Define $U(t)=\diag(W_2(t)^*, W_2(t))$ for $t\in (1/4, 3/4),$
On $[0,1/4],$ there is a continuous path $U(t)$  of unitaries  in $M_4(A\otimes C_n')$ with $U(0)=\diag(1,1,1,1)$
and $U(1/4)=\diag(W_2(0)^*, W_2(0)).$
On $[3/4, 1],$ there is a continuous path $U(t)$ of unitaries in $M_4(A\otimes C_n')$ with $U(3/4)=\diag(W_2(3/4)^*, W_2(3/4))$
and $U(1)=\diag(1,1,1,1).$

For $n=0$ (so $z$ is represented by unitaries in $M_{N_2}(C)$), we may also assume that (see \eqref{913-11}),
\beq\label{913-10}
|\tau(h_z)|<\eta\rforal \tau\in T(A),
\eneq
where
$$
h_z= \mathrm{diag}(\log({\tilde \phi}_2(z)^*{\tilde  u}_1^*{\tilde \phi}_1(z){\tilde u}_1, 1).
$$

Note that $d(z)(0) = d(z)(1) = 1$ and
\beq\label{0927-1}
[d(z)] = \gamma_{\phi_1, \phi_2, u_1}(z) - \gamma_{\phi_1, \phi_2, u_2}(z),
\eneq
where $\gamma_{\phi_1, \phi_2, u_i}$, $i=1, 2$, are the maps defined above
(see \eqref{913-1}). 

One has, on $[1/4, 3/4],$
\beq\nonumber
&&U(t)^*\diag(d(z)(t), 1,1)U(t) \\\nonumber
 &=&\diag( \left({\tilde T}_t\left(\begin{array}{cc} \tilde{u}_2 & 0 \\ 0 & 1 \end{array}\right){\tilde T}_t^*\right) \left({\tilde T}_t\left(\begin{array}{cc} \tilde{u}^*_1 & 0 \\ 0 & 1 \end{array}\right){\tilde T}_t^* \left(\begin{array}{cc} {\tilde \phi}_1(z) & 0 \\ 0 & 1 \end{array}\right)  {\tilde T}_t\left(\begin{array}{cc} \tilde{u}_1  \tilde{u}^*_2 & 0 \\ 0 & 1 \end{array}\right) {\tilde T}_t^* \left(\begin{array}{cc} {\tilde \phi}_1(z)^* & 0 \\ 0 & 1 \end{array}\right) \right), \\
 && \hskip 7mm \left(\begin{array}{cc} 1 &0\\ 0&1\end{array}\right )) \\\label{913-4}
 &= & \diag({\tilde T}_t\left(\begin{array}{cc} \tilde{u}_2\tilde{u}_1^* & 0 \\ 0 & 1 \end{array}\right){\tilde T}_t^* \left(\begin{array}{cc} {\tilde \phi}_1(z) & 0 \\ 0 & 1 \end{array}\right)  {\tilde T}_t\left(\begin{array}{cc} \tilde{u}_1\tilde{u}_2^* & 0 \\ 0 & 1 \end{array}\right){\tilde T}_t^* \left(\begin{array}{cc} {\tilde \phi}_1(z)^* & 0 \\ 0 & 1 \end{array}\right) , \left(\begin{array}{cc} 1 &0\\ 0&1\end{array}\right )),
\eneq
on $ [0,1/4],$ and on $[3/4,1],$
\beq
\|U(t)^*\diag(d(z)(t), 1, 1)U(t)-\diag(1,1,1,1)\|<\ep.
\eneq
Moreover, $U(0)^*\diag(d(z)(0), 1,1)U(0)=U^*(1)\diag(d(z)(1),1,1)U(1)=\diag(1,1,1,1).$
Therefore
\beq\label{0927-2-}
[d(z)]=[U^*d(z)U]\,\,{\rm in}\,\, K_1(SA\otimes {\widetilde{C_n}}).
\eneq
Since the short exact sequence $0\to SA\otimes C_n\to SA\otimes {\widetilde{C_n}}\to SA\to 0$
splits, we conclude that
\beq\label{0927-2}
[d(z)]=[U^*d(z)U]\,\,{\rm in}\,\, K_1(SA\otimes C_n).
\eneq

On the other hand, the class $\Gamma(\mathrm{Bott}(\phi_1, u_1u_2^*))(z)$ is represented by the path
$$
r(t)= {\tilde T}_t\left(\begin{array}{cc} \tilde{u}_2\tilde{u}_1^* & 0 \\ 0 & 1 \end{array}\right){\tilde T}_t^* \left(\begin{array}{cc} {\tilde \phi}_1(z) & 0 \\ 0 & 1 \end{array}\right)  {\tilde T}_t\left(\begin{array}{cc} \tilde{u}_1\tilde{u}_2^* & 0 \\ 0 & 1 \end{array}\right){\tilde T}_t^* \left(\begin{array}{cc} {\tilde \phi}_1(z)^* & 0 \\ 0 & 1 \end{array}\right)  \rforal t\in [1/4,3/4],
$$
and
\beq
\|r(t)-\diag(1,1)\|<\ep\rforal t\in [0,1/4]\cup [3/4,1].
\eneq
Hence (see \eqref{913-4}), by \eqref{0927-1} and \eqref{0927-2},
\beq\label{913-20}
\Gamma({\rm Bott}(\phi_1, u_1u_2^*))= \gamma_{\phi_1, \phi_2, u_1} - \gamma_{\phi_1, \phi_2, u_2}.
\eneq

\end{rem}

\begin{thm}\label{Tm72}
Let $C_1$ be a unital simple \CA\, as in
Theorem 14.10 of \cite{GLN-I},
let $A_1$ be a unital separable simple \CA\, in ${{\cal B}_0},$ let $C=C_1\otimes U_1$ and let $A=A_1\otimes U_2,$ where $U_1$ and $U_2$ are UHF-algebras of infinite type. Suppose that $\phi_1,\, \phi_2: C\to A$ are two unital monomorphisms.
Then $\phi_1$ and $\phi_2$ are asymptotically unitarily equivalent if and only if
\beq\label{Tm72-1}
[\phi_1]=[\phi_2]\,\,\,{\text in}\,\,\, KK(C,A),\\
\phi^{\ddag}=\psi^{\ddag},\,\,\,(\phi_1)_T=(\phi_2)_T, \tand {\overline{R_{\phi_1,\phi_2}}}=0.
\eneq
\end{thm}

\begin{proof}
We will prove the ``if " part only. The ``only if" part follows
from 4.3 of \cite{Lnclasn}.
Note $C=C_1\otimes U_1$ can be also regarded as a \CA\, as in Theorem 14.10 of \cite{GLN-I}.
Let $C=\lim_{n\to\infty}(C_n, \imath_n)$ be as in Theorem 14.10 of \cite{GLN-I},
where each $\imath_n: C_n\to C_{n+1}$ is an injective \hm. Let ${\cal F}_n\subset C$ be
an increasing sequence of finite subsets of $C$ such that
$\bigcup_{n=1}^{\infty}{\cal F}_n$ is dense in $C.$
Put
$$
M_{\phi_1, \phi_2}=\{(f, {  c})\in C([0,1], A){ \oplus C}: f(0)=\phi_1(c)\andeqn f(1)=\phi_2(c)\}.
$$
Since $C$ satisfies the UCT, the assumption that $[\phi_1]=[\phi_2]$ in
$KK(C,A)$ implies that the following exact sequence splits:
\beq\label{72-1}
0\to \underline{K}(SA)\to \underline{K}(M_{\phi_1, \phi_2}) \rightleftharpoons_{\theta}^{\pi_e} \underline{K}(C)\to 0
\eneq
for some
$\theta\in {\rm Hom}(\underline{K}(C), \underline{K}(A)),$ where
$\pi_e: M_{\phi_1,\phi_2}\to C$ is the projection to $C$ defined in Definition 2.20 of \cite{GLN-I}.
Furthermore, since $\tau\circ \phi_1=\tau\circ \phi_2$ for all $\tau\in T(A)$, and ${\overline{R_{\phi_1,\phi_2}}}=0,$ we may also assume that
\beq\label{72-2}
R_{\phi_1,\phi_2}(\theta(x))=0\rforal x\in K_1(C).
\eneq
By \cite{DL}, we have
\beq\label{72-3}
\lim_{n\to\infty}(\underline{K}(C_n),[\imath_n])=\underline{K}(C).
\eneq
Since $K_i(C_n)$ is finitely generated, there exists $K(n)\ge 1$ such that
\beq\label{72-3+}
{\rm Hom}_{\Lambda}(F_{K(n)}\underline{K}(C_n),\, F_{K(n)}\underline{K}(A))={\rm Hom}_{\Lambda}(\underline{K}(C_n), \underline{K}(A))
\eneq
(see also \cite{DL} for the notation $F_m$ there).

Let $\dt_n'>0$ (in place of $\dt$), $\sigma_n'>0$ (in place of $\sigma$), ${\cal G}_n'\subset C$ (in place of ${\cal G}$), \linebreak ${\{p_{1,n}', p_{2,n}',...,p'_{I(n),n)}, q_{1,n}', q_{2,n}',...,q_{I(n),n}'\}}$ (in place of
$\{p_1, p_2,...,p_k, q_1, q_2,...,q_k\}),$ ${\cal P}_n'\subset  {\underline{K}(C)}$
(in place of ${\cal P}$) corresponding to $1/2^{n+2}$ (in place of $\ep$), and ${\cal F}_n$ (in place of ${\cal F}$) be as provided by Lemma \ref{BHfull} (see also Remark \ref{RBHfull}).
Note that, by the choice as in \ref{BHfull}, we may assume that $G_{u,n}',$ the subgroup generated by
${\{[p'_{i,n}]-[q'_{i,n}]: 1\le i\le I(n)\}}$ is free abelian.

Without loss of generality, we may assume that
${\cal G}_n'\subset \imath_{n, \infty}({\cal G}_n)$ and ${\cal P}_n'\subset [\imath_{n, \infty}]({\cal P}_n)$ for some finite subset ${\cal G}_n\subset C_n,$ and for some finite subset ${\cal P}_n\subset \underline{K}(C_n),$ and we may assume that $p_{i,n}'=\imath_{n, \infty}(p_{i,n})$ and $q_{i,n}'=\imath_{n,\infty}(q_{i,n})$
for some projections $p_{i,n}, q_{i,n}\in C_n,$ $i=1,2,...,I(n).$
We may also assume that the subgroup $G_{n,u}$  generated by $\{[p_{i,n}]-[q_{i,n}]: 1\le i\le I(n)\}$
is free abelian and $p_{i,n}, q_{i,n}\in {\cal G}_n,$ $n=1,2,...,I(n).$

 We may assume that ${\cal P}_n$ contains a set of generators
of $F_{K(n)}\underline{K}(C_n),$ ${\cal F}_n\subset {\cal G}_n'$, and
$\dt_n'<1/2^{n+3}.$
We may also assume that
${\rm Bott}(h', u')|_{{\cal P}_n}$ is well defined whenever $\|[h'(a),\, u']\|<\dt_n'$ for all $a\in {\cal G}_n'$ and for any unital \hm\, $h'$ from $C_n$ and unitary $u'$ in the target algebra. Note that ${\rm Bott}(h',\,u')|_{{\cal P}_n}$ defines ${\rm Bott}(h'\, u').$
We may further assume that
\beq\label{72-4}
{\rm Bott}(h,\, u)|_{{\cal P}_n}={\rm Bott}(h',u)|_{{\cal P}_n}
\eneq
provided that $h\approx_{\dt_n'}h'$ on ${\cal G}_n'.$
We may also assume
that $\dt_n'$ is smaller than {$\dt/16$} for the $\dt$ defined in 2.15 of \cite{Lnclasn} for $C_n$ (in place of $A$) and ${\cal P}_n$ (in place of ${\cal P}$).
{Let $k(n)\ge n$ (in place of $n$), $\eta_n'>0$ (in place of $\dt$), and ${\cal Q}_{k(n)}\subset K_1(C_{k(n)})$ be as provided by Lemma \ref{Extbot3}
for
$\dt_{k(n)}'/4$ (in place of $\ep$),
$\imath_{n, \infty}({\cal G}_{k(n)})$ (in place of ${\cal F}$),  ${\cal P}_{k(n)}$ (in place of ${\cal P}$),
$\{p_{i,n}, q_{i,n},: i=1,2,...,k(n)\}$ (in place of $\{p_i, q_i: i=1,2,...,k\}$),
{and $\sigma_{k(n)}'/16$ (in place of $\sigma$)}.
 We may assume that ${\cal Q}_{k(n)}$ generates the group $K_1(C_{k(n)}).$
Since ${\cal P}$ generates $F_{K(n)}\underline{K}(C_{k(n+1)}),$ we may assume that ${\cal Q}_n\subset {\cal P}_{k(n)}.$}

Since $K_i(C_n)$ ($i=0,1$) is finitely generated, by
(\ref{72-3+}), we may further assume that
$[\imath_{k(n), \infty}]$ is injective on $[\imath_{n, k(n)}](\underline{K}(C_n)),$ $n=1,2,....$
Passing to a subsequence, we may also assume that
$k(n)=n+1.$ Let $\dt_n=\min\{\eta_n, \sigma_n',\dt_n'/2\}.$ By Lemma \ref{inv71}, there are unitaries $v_n\in U(A)$ such that
\beq\label{72-5}
&&{\rm Ad}\, v_n\circ \phi_1\approx_{\dt_{n+1}/4} \phi_2\,\,\,{\rm on}\,\,\,
\imath_{n, \infty}({\cal G}_{n+1}),\\
&&\rho_A({\rm bott}_1(\phi_2, v_n^*v_{n+1}))(x)=0
\,\,\,\rforal x\in (\imath_{n, \infty})_{*1}(K_1(C_{n+1})), \andeqn\\
&&\|[\phi_2(c),\, v_n^*v_{n+1}]\|<\dt_{n+1}/2\,\,\rforal a\in \imath_{n, \infty}({\cal G}_{n+1})
\eneq
(Recall that $K_1(C_{n+1})$ is finitely generated).
Note that, by (\ref{72-4}),
we may also assume that
\beq\label{72-7}
{\rm Bott}(\phi_1,\, v_{n+1}v_n^*)|_{[\imath_{n, \infty}]({\cal P}_n)}&=&{\rm Bott}(v_n^*\phi_1v_n, \, v_n^*v_{n+1})|_{[\imath_{n,\infty}]({\cal P}_n)}\\
&=&{\rm Bott}(\phi_2,\, v_n^*v_{n+1}])|_{[\imath_{n, \infty}]({\cal P}_n)}.
\eneq
In particular,
\beq\label{72-8}
{\rm bott}_1(v_n^*\phi_1v_n,\, v_n^*v_{n+1})(x)={\rm bott}_1(\phi_2,\, v_n^*v_{n+1})(x)
\eneq
for all $x\in (\imath_{n, \infty})_{*1}(K_1(C_{n+1})).$

Applying 10.4 and 10.5 of \cite{Linajm} (see also  Remark  \ref{gamma}),
we may assume that the pair $(\phi_1, \phi_2)$ and $v_n$ define 
an element $\gamma_n:=\gamma_{\phi_1|_{C_{n+1}},\phi_1|_{C_{n+1}}, v_n}\in {\rm Hom}_{\Lambda}(\underline{K}(C_{n+1}),\underline{K}(M_{\phi_1, \phi_2}))$ and $[\pi_e]\circ \gamma_n=[{\rm id}_{C_{n+1}}]$ (see Remark  \ref{gamma} for the definition of $\gamma_n$). Moreover, 
we may assume  (see \eqref{913-10}) that
\beq\label{72-9}
|\tau(\log(\phi_2\circ \imath_{n, \infty}(z_j^*){\tilde v_n}\phi_1\circ \imath_{n, \infty}(z_j){\tilde v_n}))|<\dt_{n+1},
\eneq
$j=1,2,...,r(n),$ where $\{z_1,z_2,...,z_{r(n)}\} \subset U(M_k(C_{n+1})),$ and this set generates $K_1(C_{n+1})$, and where
${\tilde v_n}={\rm diag}(\overbrace{v_n, v_n,...,v_n}^k).$ {We may assume that
$z_j\in {\cal Q}_n\subset {\cal P}_n,$ $j=1,2,...,r(n).$}

Let $H_n=[\imath_{n+1}](\underline{K}(C_{n+1})) \subset \underline{K}(C_{n+2}).$ Since $\bigcup_{n=1}[\imath_{n+1, \infty}](\underline{K}(C_n))=\underline{K}(C)$
and $[\pi_e]\circ \gamma_n=[{{\rm id}}_{C_{n+1}}],$ we conclude that
\beq\label{72-10}
\underline{K}(M_{\phi_1, \phi_2})=\underline{K}(SA)+\bigcup_{n=1}^{\infty}\gamma_{n+1}(H_n).
\eneq
Thus, passing to a subsequence, we may further assume that
\beq\label{72-11}
\gamma_{n+1}(H_n)\subset \underline{K}(SA)+\gamma_{n+2}(H_{n+1}),\,\,\,n=1,2,....
\eneq

Identifying $H_n$ with $\gamma_{n+1}(H_n),$ let us write $j_n: \underline{K}(SA)\oplus H_n\to \underline{K}(SA)\oplus H_{n+1}$ {for the inclusion in  (\ref{72-11}).} By
(\ref{72-10}), the inductive limit is $\underline{K}(M_{\phi_1,\phi_2}).$
From the definition of $\gamma_n,$ we note that $\gamma_n-\gamma_{n+1}\circ [\imath_{n+1}]$ maps $\underline{K}(C_{n+1})$ into $\underline{K}(SA).$
By Remark \ref{gamma} (see \eqref{913-20}),
the map
$$
\Gamma({\rm Bott}(\phi_1,\, v_nv_{n+1}^*))|_{H_n}=(\gamma_{n+1}-\gamma_{n+2}\circ [\imath_{n+2}])|_{H_n}
$$
(see \ref{gamma} for the definition of $\Gamma({\rm Bott}(,))$)
is then a \hm\, $\xi_n: H_n\to \underline{K}(SA).$ Put $\zeta_n=\gamma_{n+1}|_{H_n}.$ Then
\beq\label{72-13}
j_n(x,y)=(x+\xi_n(y),[\imath_{n+2}](y))
\eneq
for all $(x,y)\in \underline{K}(SA)\oplus H_n.$ Thus we obtain the following diagram:
\beq\label{72-14}
\begin{array}{ccccccc}
0 \to  & \underline{K}(SA)  &\to & \underline{K}(SA)\oplus H_n &\to
& H_n &\to 0\\\nonumber
 &\| & &\hspace{0.4in}\| \hspace{0.15in}\swarrow_{\xi_n} \hspace{0.05in}\downarrow_{[\imath_{n+2,\infty}]} &&
 \hspace{0.2in}\downarrow_{[\imath_{n+2,\infty}]} &\\\label{72-15}
 0 \to  & \underline{K}(SA)  &\to & \underline{K}(SA)\oplus H_{n+1} &\to & H_{n+1} &\to 0\\\nonumber
 &\| & &\hspace{0.4in}\| \hspace{0.1in}\swarrow_{\xi_{n+1}}\downarrow_{[\imath_{n+3,\infty}]} &&
\hspace{0.2in} \downarrow_{[\imath_{n+3,\infty}]} &\\
 0 \to  & \underline{K}(SA)  &\to & \underline{K}(SA)\oplus H_{n+2} &\to & H_{n+2} &\to 0.\\
 \end{array}
\eneq
By the assumption that ${\bar R}_{\phi_1, \phi_2}=0,$ the map $\theta$ also gives the following decomposition:
\beq\label{72-16}
{\rm ker}R_{\phi_1,\phi_2}={\rm ker}\rho_A\oplus K_1(C).
\eneq
Define $\theta_n=\theta\circ [\imath_{n+2, \infty}]$ and $\kappa_n=\zeta_n-\theta_n.$
Note that
\beq\label{72-17}
\theta_n=\theta_{n+1}\circ [\imath_{n+2}].
\eneq
We also have that
\beq\label{72-18}
\zeta_n-\zeta_{n+1}\circ [\imath_{n+2}]=\xi_n.
\eneq
Since $[\pi_e]\circ (\zeta_n-\theta_n)|_{H_n}=0,$ $\kappa_n$ maps
$H_n$ into $\underline{K}(SA).$ It follows that
\beq\label{72-19}
\kappa_n-\kappa_{n+1}\circ [\imath_{n+2}] &=&
\zeta_n-\theta_n-\zeta_{n+1}\circ [ \imath_{n+2}]+\theta_{n+1}\circ
[\imath_{n+2}]\\{\label{72-19+}}
&=&\zeta_n-\zeta_{n+1}\circ[\imath_{n+2}]=\xi_n{{.}}
\eneq
It follows from Lemma \ref{VuV} that there are an integer
$N_1\ge 1,$  a unital ${\dt_{n+1}\over{4}}$-$\imath_{n+1}({\cal
G}_{n+1})$-multiplicative
\cp\, $L_n: \imath_{n,
\infty}(C_{n+1})\to M_{1+N_1}(M_{\phi_1,\phi_2}),$ a unital \hm\,
$h_0: \imath_{n+1, \infty}(C_{n+1})\to M_{N_1}(\C),$
and a continuous path of unitaries $\{V_n(t): t\in [0,3/4]\}$ in
$M_{1+N_1}(A)$ such that $[L_n]|_{{\cal P}_{n+1}'}$ is well
defined, $V_n(0)=1_{M_{1+N_1}(A)},$
\begin{equation*}
[L_n\circ \imath_{n,\infty}]|_{{\cal P}_n}=(\theta\circ
[\imath_{n+1,\infty}]+[h_0\circ \imath_{n+1,\infty}])|_{{\cal P}_n},
\end{equation*}
\begin{equation*}
\pi_t\circ L_n\circ
\imath_{n+1,\infty}\approx_{\dt_{n+1}/4} {\rm Ad}\, V_n(t)\circ
((\phi_1\circ \imath_{n+1,\infty})\oplus (h_0\circ\imath_{n+1, \infty}))
\end{equation*}
on $\imath_{n+1,\infty}({\cal G}_{n+1})$ for all $t\in (0,3/4],$
\begin{equation*}
\pi_t\circ L_n\circ
\imath_{n+1,\infty}\approx_{\dt_{n+1}/4} {\rm Ad}\, V_n(3/4)\circ
((\phi_1\circ \imath_{n+1,\infty})\oplus (h_0\circ
\imath_{n+1,\infty}))
\end{equation*}
on $\imath_{n+1,\infty}({\cal G}_{n+1})$ for all $t\in (3/4,1),$ and
\begin{equation*}
\pi_1\circ L_n\circ
\imath_{n+1,\infty}\approx_{\dt_{n+1}/4}\phi_2\circ
\imath_{n+1,\infty}\oplus h_0\circ \imath_{n+1,\infty}
\end{equation*}
on $\imath_{n+1,\infty}({\cal G}_{n+1}),$ where $\pi_t:
M_{\phi_1,\phi_2}\to A$ is the point evaluation at $t\in (0,1).$

Note that $R_{\phi_1,\phi_2}(\theta(x))=0$ for all $x\in
\imath_{n+1,\infty}(K_1(C_{n+1})).$  As  in \eqref{913-30} (see also
10.4 of
\cite{Linajm}),
\beq\label{72-24}
\tau(\log((\phi_2(x)\oplus h_0(x)^*V_n(3/4)^*(\phi_1 (x)\oplus
h_0(x))V_n(3/4)))=0
\eneq
for $x=\imath_{n+1, \infty}(y),$ where $y$ is in a set of generators
of $K_1(C_{n+1}),$ and for all $\tau\in T(A).$

Define $W_n'={\rm diag}(v_{n+1},1)\in M_{1+N_1}(A).$ Then
\beq\label{913-n-1}
{\tilde \kappa}_n:=\text{Bott}((\phi_1\oplus h_0)\circ \imath_{n+1, \infty},\,
W_n'(V_n(3/4)^*)
\eneq
 defines a \hm\,
 in
${\rm Hom}_{\Lambda}(\underline{K}(C_{n+1}),\underline{K}(SA)).$ By
(\ref{72-9})
\beq\label{72-25}
|\tau(\log((\phi_2\oplus h_0)\circ \imath_{n+1,\infty}(z_j)^*(W_n')^*(\phi_1\oplus h_0)\circ \imath_{n+1,\infty}(z_j)W_n'))|<\dt_{n+1},
\eneq
$j=1,2,...,r(n).$
  One computes  (see \eqref{913-33}) that
\beq\label{72-25+}
\Gamma(\text{Bott}(\phi_1\circ {\imath}_{n+1, \infty}\oplus h_0,\, W_n'V(3/4)^*)|_{{\cal P}_n}=(\gamma_{n+1}-\theta)[\imath_n]|_{{\cal P}_n}.
\eneq
Put ${\tilde V}_n=V_n(3/4).$
Let
\beq\label{72-26}
&&\hspace{-0.7in}b_{j,n}={1\over{2\pi i}}\log({\tilde V}_n^*(\phi_1\oplus h_0)\imath_{n+1,\infty}(z_j){\tilde V}_n (\phi_2\oplus h_0)\circ \imath_{n+1, \infty}(z_j)^*),\\
&&\hspace{-0.7in}b_{j,n}'={1\over{2\pi i}}\log((\phi_1\oplus h_0)\circ \imath_{n+1,\infty}(z_j){\tilde V}_n(W_n')^*(\phi_1\oplus h_0)\circ \imath_{n+1, \infty}(z_j)^*W_n'{\tilde V}_n^*), \andeqn \label{2019-sept-10+1}\\
&&\hspace{-0.7in}b_{j,n}''={1\over{2\pi i}}\log((\phi_2\oplus h_0)\imath_{n+1,\infty}(z_j)(W_n')^* (\phi_1\oplus h_0)\circ \imath_{n+1, \infty}(z_j)^*W_n'),
\eneq
$j=1,2,...,r(n).$
By (\ref{72-24}) and (\ref{72-25}),
\beq\label{72-27}
\tau(b_{j,n})=0 \andeqn |\tau(b_{j,n}'')|<\dt_{n+1}
\eneq
for all $\tau\in T(A).$
Note that
\beq\label{72-28}
{\tilde V}_n^*e^{2\pi i b_{j,n}'}{\tilde V}_n=
e^{2\pi i b_{j,n}} e^{2\pi i b_{j,n}''}{{.}}
\eneq
Then, by 6.1 of \cite{Lnind} and  by (\ref{72-27}),
\beq\label{72-29}
\tau(b_{j,n}')&=&\tau(b_{j,n})-\tau(b_{j,n}'')
=\tau(b_{j,n}'') \andeqn |\tau(b_{j,n}')|<\dt_{n+1}
\eneq
for all $\tau\in T(A).$
It follows from this, \eqref{913-n-1},  and \eqref{2019-sept-10+1}
that
\beq\label{72-30}
|\rho_A({\tilde \kappa}_n(z_j))(\tau)| <\dt_{n+1},\,\,\,j=1,2,...,
\eneq
for all $\tau\in T(A).$
It follows from {\ref{Extbot3}} that there is a unitary $w_n'\in U(A)$
such that
\beq\label{72-31}
&&\|[\phi_1(a), w_n']\|<\dt_{n+1}'/4\rforal a\in \imath_{n+1,\infty}({\cal G}_{n+1})\andeqn\\
&&\text{Bott}(\phi_1\circ \imath_{n+1,\infty},\, w_n')=-{\tilde
\kappa}_n\circ[\imath_{n+1}].
\eneq
By (\ref{72-4}),
\beq\label{72-32}
\text{Bott}(\phi_2\circ \imath_{n+1,\infty},\,v_n^*w_n'v_n)|_{{\cal
P}_n}=-{\tilde \kappa}_n\circ [\imath_{n+1}]|_{{\cal P}_n}.
\eneq
 It follows from  \eqref{913-33} (see also
 10.6 of \cite{Linajm}) and \eqref{72-25+} that
\beq\label{72-33}
\Gamma(\text{Bott}(\phi_1\circ \imath_{n+1, \infty}, w_n'))&=&-\kappa_n\circ [\imath_{n+1}]
\andeqn\\
 \Gamma(\text{Bott}(\phi_1\circ \imath_{n+2, \infty},
w_{n+1}'))&=&-\kappa_{n+1}\circ [\imath_{n+2}].
\eneq
We also have
\beq\label{72-34}
\Gamma(\text{Bott}(\phi_1\circ \imath_{n+1,\infty},
v_nv_{n+1}^*))|_{H_n}=\zeta_n-\zeta_{n+1}\circ [\imath_{n+2}]=\xi_n.
\eneq
But, by (\ref{72-19}) {and (\ref{72-19+})},
\beq\label{72-35}
(-\kappa_n +\xi_n+\kappa_{n+1}\circ [{\imath_{n+2}}])=0.
\eneq
By 10.6 of \cite{Linajm} (see also Remark \ref{gamma}), $\Gamma({\rm Bott}(.,.))=0$ if and only
if ${\rm Bott}(.,.)=0.$ Thus, by (\ref{72-32}), (\ref{72-33}), and (\ref{72-34}),
\beq\label{72-36}
{\hspace{-0.2in}} -\text{Bott}(\phi_1\circ \imath_{n+1, \infty},\,w_n')
+\text{Bott}(\phi_1\circ \imath_{n+1, \infty}, \,v_nv_{n+1}^*)
+\text{Bott}(\phi_1\circ\imath_{n+1,\infty}, {w_{n+1}'}) =0.
\eneq
Put $w_n=v_n^*(w_n')v_n$ and  $u_n=v_nw_n^*,$ $n=1,2,....$ Then, by (\ref{72-5}) and
(\ref{72-31}),
\beq\label{72-37}
{\rm Ad}\, u_n\circ \phi_1\approx_{\dt_n'/2} \phi_2\rforal a\in
\imath_{n+1, \infty}({\cal G}_{n+1}).
\eneq
From (\ref{72-7}), (\ref{72-4}), and (\ref{72-36}), we compute
that
\beq\label{72-38}
&&\hspace{-0.6in}\text{Bott}(\phi_2\circ \imath_{n+1, \infty},u_n^*u_{n+1})
= \text{Bott}(\phi_2\circ \imath_{n+1, \infty}, w_nv_n^*v_{n+1}w_{n+1}^*)\\
&=&\hspace{-0.1in} \text{Bott}(\phi_2\circ \imath_{n+1, \infty}, w_n)+\text{Bott}(\phi_2\circ \imath_{n+1, \infty}, v_n^*v_{n+1})\\
&&\hspace{0.3in}
+\text{Bott}(\phi_2\circ \imath_{n+1, \infty}, w_{n+1}^*)\\
&=&\text{Bott}(\phi_1\circ \imath_{n+1, \infty},w_n')+\text{Bott}(\phi_1\circ \imath_{n+1, \infty}, v_{n+1}v^*_n
)\\
&&\hspace{0.3in}+\text{Bott}(\phi_1\circ \imath_{n+1, \infty}, (w_{n+1}')^*)\\
&=&-[-\text{Bott}(\phi_1\circ \imath_{n+1, \infty}, w_n')+\text{Bott}(\phi_1\circ \imath_{n+1, \infty}, v_nv_{n+1}^*)\\\label{Add211}
&&\hspace{2.1in}+\text{Bott}(\phi_1\circ \imath_{n+1, \infty}, {w_{n+1}'})]=0.
\eneq
Let $x_{i,n}=[p_{i,n}]-[q_{i,n}], \,\, 1\le i\le I(n).$ Note that we assume that
$G_{u,n}$ is a free abelian group generated by  $\{x_{i,n}: 1\le i\le I(n)\}.$ \Wlog, we may assume that these generators are independent.
{{Put $e_{i,n}=\phi_2\circ \imath_{n+1,\infty}(p_{i,n}),$
$e'_{i,n}=\phi_2\circ \imath_{n+1,\infty}(q_{i,n}),$
$i=1,2,..., I(n).$
Put $s_1=1$  and ${\td u}_1=u_1s_1^*=u_1.$}}
{{Define a \hm\, $\Lambda_1: G_{u,1}\to U_0(A)/CU(A)$ (see \eqref{Add211}) by
\beq
\Lambda_1(x_{i,1})=\overline{\langle ((1-e_{i,1})+e_{i,1}u_1^*u_2)((1-e'_{i,1})+e'_{i,1}u_2^*u_1)\rangle}.
\eneq}}
{{Since $\Lambda_1$ factors through $G_{u,1}',$ applying Theorem \ref{Extbot3} 
(or just  Theorem \ref{BB-exi+}) to $\phi_2\circ\iota_{2, \infty}$, one obtains  a unitary $s_{2}\in {{A}}$ such that}}
\beq
||[\phi_2\circ\iota_{2, \infty}(f), {{s_{2}}}]||<\delta'_{2}/4\rforal f\in\mathcal G_{2},\\
\mathrm{Bott}(\phi_2\circ\iota_{2, \infty}, s_2)|_{\mathcal P_2}=0, \andeqn
\eneq
\beq\label{n-72-39+}
\hspace{-0.2in}{\rm dist}(\overline{\langle ((1-e_{i,1})+e_{i,1}{{s_{2}^*}})((1-e'_{i,1})+
e'_{i,1}s_{2})\rangle},
\Lambda_{1}(-x_{i,1}))<\sigma_2'/16.
\eneq
{{Define $\td u_2=u_2s_2^*.$}}
%
%
In what follows, we will construct unitaries $s_2, ..., {s_n,...}$ in $A$ such that
\beq\label{n-72-37}
||[\phi_2\circ\iota_{j+1, \infty}(f), s_{{j+1}}]||<\delta'_{j+1}/4\rforal f\in\mathcal G_{j+1},\\
\label{n-72-38}
\mathrm{Bott}(\phi_2\circ\iota_{+1, \infty}, s_{j+1})|_{\mathcal P_j}=0, \andeqn\\
\label{n-72-39}
\hspace{-0.2in}{\rm dist}(\overline{\langle ((1-e_{i,j})+e_{i,j}s^*_{n+1})((1-e'_{i,j})+e'_{i,j}s_{j+1})\rangle},
\Lambda_n(-x_{i,j}))<\sigma_n'/16,
\eneq
{{where $\Lambda_j: G_{u,j}\to U_0(A)/CU(A)$  is a \hm\, defined  by
\beq\label{2019-sept-10-2}\Lambda_j(x_{i,j})=\overline{\langle ((1-e_{i,j})+e_{i,}{{\widetilde{u_j}}}^*u_{j+1})((1-e'_{i,j})+e'_{i,j}{{u_{j+1}^*{\widetilde{u_j}})\rangle}}}\eneq
(see \eqref{Add211} and \eqref{n-72-38} for $j$), and ${\tilde u}_j=u_js_j^*,$
$j=1,2,....$}}

{{A}}ssume that $s_2, ..., s_n$ are already constructed. Let us construct $s_{n+1}$. Note that by 
\eqref{72-38} 
the $K_1$ class of the unitary $u_n^*u_{n+1}$ is trivial. In particular, the $K_1$ class of $s_nu_n^*u_{n+1}$ is trivial.  {Since ${{\Lambda_n}}$ factors through $G_{u,n}',$ applying Theorem \ref{Extbot3}} to $\phi_2\circ\iota_{{{n+1}}, \infty}$, one obtains  a unitary $s_{n+1}\in {{A}}$ such that
\beq
&&||[\phi_2\circ\iota_{{{n+1}}, \infty}(f), {{s_{n+1}}}]||<\delta'_{n+1}/4\rforal f\in\mathcal G_{n+1},\\
&&\mathrm{Bott}(\phi_2\circ\iota_{{{n+1}}, \infty}, s_{n+1})|_{\mathcal P_n}=0, \andeqn\\
\label{n-72-39+}
&&\hspace{-1in}{\rm dist}(\overline{\langle ((1-e_{i,n})+e_{i,n}{{s_{n+1}^*}})((1-e'_{i,n})+
e'_{i,n}s_{n+1})\rangle},
\Lambda_{n}(-x_{i,n}))<\sigma_n'/16,
\eneq
$i=1,2,...,I(n+1).$
Then  $s_1, s_2, ..., s_{n+1}$ {{satisfy}} \eqref{n-72-37}, \eqref{n-72-38}, and \eqref{n-72-39}.

Put $\widetilde{u_n+1}=u_{n+1}s_{n+1}^*$. Then by \eqref{72-37} and \eqref{n-72-37}, one has
\beq\label{n-72-40}
{\rm ad}\, \widetilde{u_n}\circ \phi_1\approx_{\dt_n'} \phi_2\rforal a\in
\imath_{n+1, \infty}({\cal G}_{n+1}).
\eneq
By \eqref{72-38} and \eqref{n-72-38}, one has
\begin{equation}\label{n-72-41}
\mathrm{Bott}(\phi_2\circ\iota_{n+1, \infty}, (\widetilde{u_{n}})^* \widetilde{u_{n+1}})|_{\mathcal P_n}=0.
\end{equation}
Note that
{{\beq\label{72-n2014103}
\overline{\langle (1-e_{i,n})+e_{i,n}\widetilde{u_n}^* \widetilde{u_{n+1}}\rangle \langle(1-e_{i,n}')+e_{i,n}'{\widetilde{u_{n+1}}^*}{\widetilde{u_{n}}}\rangle}
=\overline{c_1c_2c_4c_3}=
\overline{c_1c_3c_2c_4},
\eneq}}
where
{{\beq
c_1=\langle(1-e_{i,n})+e_{i,n}{\widetilde{u_n}}^*u_{n+1}\rangle,
\,\, c_2=\langle (1-e_{i,n})+e_{i,n}s_{n+1}^*\rangle.\\
c_3=\langle (1-e'_{i,n})+{e_{i,n}'}u_{n+1}^*\widetilde{u_n}\rangle,\,\, 
c_4=\langle (1-e'_{i,n})+e'_{i,n}s_{n+1}\rangle.
\eneq}}

%
Therefore,  by \eqref{n-72-39+} and (\ref{2019-sept-10-2}),
one has
{{\beq\label{n-72-42}
&&\hspace{-0.2in}\mathrm{dist}(\overline{\langle ((1-e_{i,n})+e_{i,n}{\widetilde{u_n}}^*{\widetilde{u_{n+1}}})((1-e_{i,n}')+e_{i,n}'{\widetilde{u_{n+1}}^*{\widetilde{u_n}})}\rangle}), {\bar 1})\\
&&<\sigma_n'/16+{\rm dist}(\Lambda_n(x_{i,n})\Lambda_n(-x_{i,n}), {\bar 1})
={{\sigma_{n}'/16,}}
\eneq}}
$i=1,2,...,I(n).$
Therefore, by Lemma \ref{BHfull} (and Remark \ref{RBHfull}), there exists a continuous and piecewise smooth
path of unitaries $\{z_n(t): t\in [0,1]\}$ of $A$ such
that
\beq\label{72-39}
&&z_n(0)=1,\,\,\, z_n(1)=(\widetilde{u_n})^*\widetilde{u_{n+1}}\andeqn\\\label{72-40}
&&\|[\phi_2(a),\, z_n(t)]\|<1/2^{n+2}\rforal a\in {\cal F}_n\andeqn
t\in [0,1].
\eneq
Define
$$
u(t+n-1)=\widetilde{u_n}z_{n+1}(t)\,\,\,t\in (0,1].
$$
Note that $u(n)=\widetilde{u_{n+1}}$ for all integers $n$ and $\{u(t):t\in [0,
\infty)\}$ is a continuous path of unitaries in $A.$ One estimates
that, by (\ref{72-37}) and (\ref{72-40}),
\beq\label{72-41}
{\rm Ad}\, u(t+n-1)\circ \phi_1 \approx_{\dt_n'} {\rm
Ad}\,z_{n+1}(t)\circ \phi_2
 \approx_{1/2^{n+2}} \phi_2
\,\,\,\,\,\,\,\text{on}\,\,\,\,{\cal
F}_n
\eneq
 for all $t\in (0,1).$
It then follows that
\beq\label{72-42}
\lim_{t\to\infty}u^*(t)\phi_1(a) u(t)=\phi_2(a)\rforal a\in C.
\eneq

\end{proof}

\section{Rotation maps and strong asymptotic equivalence}

\begin{lem}\label{L91}
{Let {{$A$}} be a unital separable simple \CA\, of stable rank one.}
Suppose that
$u\in CU(A).$ Then, for any continuous and piecewise smooth  path
$\{u(t): t\in [0,1]\}\subset U(A)$ with $u(0)=u$ and $u(1)=1_A,$
\beq\label{L91-1}
D_A(\{u(t)\})\in \overline{\rho_A(K_0(A))}\,\,\,\,\,\,\hspace{0.3in}\textrm{(recall Definition 2.16 of \cite{GLN-I}
 for $D_A$).}
\eneq
\end{lem}

\begin{proof}
It follows from Corollary 11.11 of \cite{GLN-I} 
that the map
$j: u\mapsto {\rm diag}(u,1,..., 1)$ from $U(A)$ to $U(M_n(A))$ induces
an isomorphism from $U(A)/CU(A)$ to $U(M_n(A))/CU(M_n(A)).$ Then
the conclusion follows from 3.1 and 3.2 of \cite{Thomsen-rims}.
\end{proof}

\begin{lem}\label{L92}
{Let {{$A$}}  be a unital separable simple \CA\, of stable rank one.}
{Suppose that $B$ is a unital separable \CA\,}
and suppose that $\phi,\,\psi: B \to A$ are two unital monomorphisms such that
\beq\label{L92-1}
[\phi]=[\psi]\,\,\,{\textrm in}\,\,\,{{KK(B,A),}} \\\label{L92-2}
\phi_T=\psi_T\tand \phi^{\ddag}=\psi^{\ddag}.
\eneq
Then
\beq\label{L92-3}
R_{\phi, \psi}\in {\rm Hom}(K_1(B), \overline{\rho_A(K_0(A))}).
\eneq

\end{lem}

\begin{proof}
Let $z\in K_1(B)$ be represented by the unitary $u\in {U(M_m(B))}$ {for some integer $m$}.
Then, by (\ref{L92-2}),
$$
{(\phi\otimes {\rm id}_{M_m})(u)(\psi\otimes {\rm id}_{M_m})(u)^*\in CU(M_m(A)).}
$$
Suppose that $\{u(t): t\in [0,1]\}$ is a continuous and piecewise smooth path in ${M_m(U(A))}$ such that
$u(0)={(\phi\otimes {\rm id}_{M_m})}(u)$ and $u(1)=(\psi\otimes{\rm id}_{M_m})(u).$ Put $w(t)={(\psi\otimes {\rm id}_{M_m})(u)^*u(t)}.$  Then $w(0)={(\psi\otimes {\rm id}_{M_m})(u)^*(\phi\otimes {\rm id}_{M_{m}})(u)}\in CU(A)$ and $w(1)=1_A.$ Thus,
\beq\label{L92-4}
R_{\phi, \psi}(z)(\tau)&=&{1\over{2\pi i}}\int_0^1 \tau({du(t)\over{dt}}u^*(t))dt
= {1\over{2\pi i}}\int_0^1 \tau(\psi(u)^*{du(t)\over{dt}}u^*(t)\psi(u))dt\\
&=& {1\over{2\pi i}}\int_0^1 \tau({dw(t)\over{dt}}w^*(t))dt
\eneq
for all $\tau\in T(A).$ By \ref{L91},
\beq\label{L92-5}
R_{\phi, \psi}(z)\in \overline{\rho_A(K_0(A))}.
\eneq
It follows that
\beq\label{L92-6}
R_{\phi, \psi}\in {\rm Hom}(K_1(B), \overline{\rho_A(K_0(A))}).
\eneq

\end{proof}

\begin{thm}\label{T94}
Let $C_1, C_2\in {\cal B}_0$ be unital separable simple \CA s, and $A=C_1\otimes U_1,$  $B=C_2\otimes U_2,$ where $U_1$ and $U_2$ are  UHF-algebras of infinite type, and  $B$ satisfies the UCT. Suppose that $B$ is a  {unital \SCA\,} of $A,$ and denote by $\imath$ the embedding. For any $\lambda\in  {\rm Hom}(K_{{1}}(B), \overline{\rho_A(K_0(A))}),$ there exists $\phi\in {\overline{{\rm Inn}}}(B,A)$ (see Definition 2.8 of \cite{GLN-I})
such that there are \hm s $\theta_i: K_i(B)\to K_i(M_{\imath, \phi})$ with $(\pi_0)_{*i}\circ \theta_i={\rm id}_{K_i({{B}})},$ $i=0,1,$ and
the rotation map $R_{\imath, \phi}: K_1({{M_{\imath, \phi}}})\to {\rm Aff}(T(A))$ is given by
\beq\label{T94-1}
R_{\imath, \phi}(x)=\rho_A(x-\theta_1(\pi_0)_{*1}(x))+\lambda\circ (\pi_0)_{*1}(x))
\eneq
for all $x\in K_1(M_{\imath, \phi}).$ In other words,
\beq\label{T94-2}
[\phi]=[\imath]\,\,\,{\text{in}}\,\,\, KK(B,A)
\eneq
and the rotation map $R_{\imath, \phi}: K_1(M_{\imath, \phi})\to
{\rm Aff}(T(A))$ is given by
\beq\label{T94-3}
R_{\imath, \phi}(a,b)=\rho_A(a)+\lambda(b)
\eneq
for some identification of $K_1(M_{\imath, \phi})$ with $K_0(A)\oplus K_1(B).$
\end{thm}

\begin{proof}
The proof is exactly the same as that of Theorem 4.2 of \cite{L-N}. By Lemma \ref{Vpair2} and
Lemma \ref{Extbot1} (see also Lemma \ref{Extbot2} {{and Remark \ref{Remark202011-1}}}), we have the properties (B1) and (B2) associated with $B$ (defined in {3.6} of \cite{L-N}) as in Theorem 4.2 of \cite{L-N}.
In 4.2 of \cite{L-N}, it is also assumed that $\rho_A(K_0(A))$ is dense in
${\rm Aff}(T(A))$, which is only used to get that $\psi(K_1(B))\subset \overline{\rho_A(K_0(A))}$, which corresponds to the assumption $\lambda(K_1(B))\subset \overline{\rho_A(K_0(A))}$ here.
\end{proof}

\begin{df}\label{d95}
Let $A$ be a unital \CA\, and let $C$ be a unital separable
\CA. Denote by ${\rm{Mon}}_{asu}^e(C,A)$  the set of all asymptotic unitary equivalence classes of unital monomorphisms from $C$ into $A.$ Denote by $\small{{\boldsymbol{K}}}:
\textrm{Mon}_{asu}^e(C, A)\to {KK}_e(C,A)^{++}$  the map defined by
$$
\phi\mapsto [\phi]\tforal \phi\in \mathrm{Mon}_{asu}^e(C,A).
$$
Let $\kappa\in {KK}_e(C,A)^{++}.$ Denote by $\langle \kappa
\rangle$ the set of classes of all $\phi\in \mathrm{Mon}_{asu}^e(C,A)$ such
that $\small{{\boldsymbol{ K}}}(\phi)=\kappa.$

Denote by $KKUT_e(A,B)^{++}$ the set of triples $(\kappa,
\af,\gamma)$ for which $\kappa\in KK_e(A,B)^{++},$ $\af:
U(A)/CU(A)\to U(B)/CU(B)$ is a \hm\,, $\gamma: T(B)\to T(A)$ is a continuous
affine map, and both $\af$ and $\gamma$ are compatible
with $\kappa.$ Denote by $\boldsymbol{\mathfrak{K}}$ the map from
$\textrm{Mon}_{asu}^e(C, A)$ into ${KKUT}(C,A)^{++}$ defined by
$$
\phi\mapsto ([\phi],\phi^{\ddag}, \phi_T)\tforal \phi\in
\mathrm{Mon}_{asu}^e(C,A).
$$
Denote by $\langle \kappa, \af,\gamma \rangle $ the subset of
$\phi\in \mathrm{Mon}_{asu}^e(C,A)$ such that
$\boldsymbol{\mathfrak{K}}(\phi)=(\kappa,\,\af,\,\gamma).$

\end{df}

\begin{thm}\label{T96}
Let $C$ and $A$ be two  unital separable  amenable \CA s. Suppose that $\phi_1, \phi_2, \phi_3: {C\to A}$
are three unital monomorphisms for which
\beq\label{Mul-1}
[\phi_1]=[\phi_2]=[\phi_3]\,\,\,{\textrm in}\,\,\, KK({C,A}))\tand
(\phi_1)_T=(\phi_2)_T=(\phi_3)_T.
\eneq
Then
\beq\label{Mul-2}
\overline{R}_{\phi_1, \phi_2}+\overline{R}_{\phi_2,
\phi_3}=\overline{R}_{\phi_1,\phi_3}.
\eneq
\end{thm}

\begin{proof}
The proof is exactly the same as that of Theorem 9.6 of \cite{Lnclasn}.
\end{proof}

\begin{lem}\label{L97}
Let $A$ {{and}} $B$ be two unital separable amenable \CA s.
Suppose that $\phi_1, \phi_2: A\to B$ are  two
unital monomorphisms such that
$$
[\phi_1]=[\phi_2]\,\,\,{\textrm in}\,\,\, KK(A,B)\tand
(\phi_1)_T=(\phi_2)_T.
$$
Suppose that $(\phi_2)_T: T(B)\to T(A)$ is an affine
homeomorphism. Suppose also that there is $\af\in Aut(B)$ such that
$$[\af]=[{\rm id}_B]\,\,\,{\textrm in}\,\,\, KK(B,B)\tand \af_T={\rm id}_T.
$$
Then
\beq\label{Group1-0}
\overline{R}_{\phi_1,\af\circ \phi_2} =\overline{R}_{{\rm
id}_B,\af}\circ (\phi_2)_{*1}+\overline{R}_{\phi_1,\phi_2}
\eneq
in ${\rm Hom}(K_1(A),
\Aff(T(B)))/{\cal R}_0.$

\end{lem}

\begin{proof}
Using \ref{T96}, we compute that
\beq\nonumber
\overline{R}_{\phi_1, \af\circ \phi_2}&=&\overline{R}_{\phi_1, \phi_2}+\overline{R}_{\phi_2,\af\circ \phi_2}=
 \overline{R}_{\phi_1, \phi_2}+\overline{R}_{{\rm id}_B, \af}\circ (\phi_2)_{*1}.
\eneq
\end{proof}

\begin{thm}\label{MT2}
Let $B$
be a unital  separable simple {{amenable}} \CA\, in ${\cal B}_0$  {{satisfying the UCT,}}
let $C=B\otimes U_1,$ where $U_1$ is a UHF-algebra
of infinite type{,} { let $A_1$ be a unital separable amenable simple \CA\, in ${\cal B}_0$, and let $A=A_1\otimes U_2,$
where $U_2$ is another UHF-algebra of infinite type. }
 Then the map $\boldsymbol{\mathfrak{K}}:
\mathrm{Mon}_{asu}^e(C,A)\to {KKUT}(C,A)^{++}$ is surjective.
Moreover, for each $(\kappa, \af, \gamma)\in {KKUT}(C,A)^{++}$,
there exists a bijection
$$
\eta: \langle \kappa,\af,\gamma \rangle \to \mathrm{Hom}({K}_1(C),
{\overline{\rho_A(K_0(A))}})/{\cal R}_0.
$$
\end{thm}

\begin{proof}
It follows from Lemma \ref{L86} {{(also Remark \ref{Remark202011-1})}} that $\boldsymbol{\mathfrak{K}}$ is
surjective.

Fix a triple $(\kappa,\af, \gamma)\in {KKT}(C,A)^{++}$ and choose
a unital monomorphism $\phi: C\to A$ such that $[\phi]=\kappa$,
$\phi^{\ddag}=\af$, and  $\phi_\mathrm{T}=\gamma.$
If $\phi_1: C\to A$ is another unital monomorphism such that
$\boldsymbol{\mathfrak{K}}(\phi_1)=\boldsymbol{\mathfrak{K}}({{\phi}}),$ then by Lemma \ref{L92},
\beq\label{T98-1}
{\overline{R}}_{\phi, \phi_1}\in {\rm Hom}(K_1(C), \overline{\rho_A(K_0(A))})/{\cal R}_0.
\eneq

Let $\lambda\in \mathrm{Hom}({K}_1(C), \overline{\rho_A(K_0(A))})$
be a \hm.  It follows from Theorem \ref{T94} that there is a unital monomorphism
$\psi\in \overline{{{\rm Inn}}}(\phi(C), A)$ with $[\psi\circ \phi]=[\phi]$ in $KK(C,A)$ such that there exists a homomorphism
$\theta: {K}_1(C)\to K_1(M_{\phi, \psi\circ \phi})$ with
$(\pi_0)_{*1}\circ \theta={\rm{id}}_{{K}_1(C)}$ for which
$R_{\phi, \psi\circ \phi}\circ \theta=\lambda.$ Let
$\beta=\psi\circ \phi.$ Then $R_{\phi, \beta}\circ
\theta=\lambda.$ Note also that, since $\psi\in
{\overline{\rm{Inn}}}(\phi(C), A),$ $\beta^{\ddag}=\phi^{\ddag}$
and $\beta_\mathrm{T}=\phi_\mathrm{T}.$ In particular, $\boldsymbol{\mathfrak{K}}(\beta)={\boldsymbol{\mathfrak{K}}}(\phi).$

Thus{{, for each unital monomorphism $\phi$,}} we obtain a well-defined  and surjective map $$\eta_{{\phi}}: \langle
[\phi], \phi^{\ddag}, \phi_T\rangle \to \mathrm{Hom}({K}_1(A),
\overline{\rho_A(K_0(A))})/{\cal R}_0.$$
To see that $\eta_\phi$  is injective,  consider two monomorphisms
%
$\phi_1, \phi_2: C\to A$
in $ \langle
[\phi], \phi^{\ddag},\phi_T\rangle$  such that
$$
\overline{R}_{\phi, \phi_1}=\overline{R}_{\phi, \phi_2}.
$$
Then, by Theorem \ref{T96},
\beq
\overline{R}_{\phi_1,\phi_2}=\overline{R}_{\phi_1,\phi}+\overline{R}_{\phi,
\phi_2}
=-\overline{R}_{\phi, \phi_1}+\overline{R}_{\phi, \phi_2}=0.
\eneq
It follows from Theorem \ref{Tm72} that $\phi_1$ and $\phi_2$ are
asymptotically unitarily equivalent. The map $\eta_{{\phi}}$ is the desired bijection $\eta$ as $\langle
[\phi], \phi^{\ddag}, \phi_T\rangle=\langle \kappa,\af,\gamma \rangle$.
\end{proof}

\begin{df}
{\rm  Denote by $KKUT_e^{-1}(A,A)^{++}$ the subset of those
elements ${{(}} \kappa, \af, \gamma{{)} }\in KKUT_e(A,A)^{++}$
for which $\kappa|_{K_i(A)}$ is an isomorphism $(i=0,1$), $\af$ is an
isomorphism, and $\gamma$ is {an}  affine homeomorphism.  {{Recall from the proof of Theorem \ref{MT2} that
$\eta_{{\rm id}_A}: {\langle [{\rm id}_A], {\rm
id}_A^{\ddag}, ({\rm id}_A)_T\rangle}
\to$ \\ $\mathrm{Hom}({K}_1(A),
\overline{\rho_A(K_0(A))})/{\cal R}_0$ is a bijection.}}

Denote by $\langle {\rm id}_A\rangle $ the class of those
automorphisms $\psi$ which are asymptotically unitarily equivalent
to ${\rm id}_A${{---this subset of ${\rm Aut}(A)$ gives rise to a single element in $\mathrm{Mon}_{asu}^e(A,A)$ which should not be confused with the subset  ${\langle [{\rm id}_A], {\rm
id}_A^{\ddag}, ({\rm id}_A)_T\rangle}\subset\mathrm{Mon}_{asu}^e(A,A)$.}}
Note that, if $\psi\in \langle {\rm id}_A\rangle
,$ then $\psi$ is {\it asymptotically inner}, i.e., there exists a
continuous path of unitaries $\{u(t): t\in [0,\infty)\}\subset A$
such that
$$
\psi(a)=\lim_{t\to\infty}u(t)^*au(t)\tforal a\in A.
$$

{{ Note that $\langle {\rm id}_A\rangle $ is a normal subgroup of ${\rm Aut}(A)$.}}

}
\end{df}

\begin{cor}\label{C910}
Let $A_1\in {\cal B}_0$
be a unital simple {{amenable}}  \CA\,  {{satisfying the UCT}} and let
$A=A_1\otimes U$ for some  UHF-algebra $U$ of infinite type. Then one
has the following short exact sequence:
\beq\label{Inn1}
\hspace{-0.4in}0 \to {\rm Hom}(K_1(A), \overline{\rho_A(K_0(A))})/{\cal
R}_0\stackrel{\eta_{{\rm id}_A}^{-1}}{\to}{\rm Aut}(A)/\langle
{\rm id}_A\rangle \stackrel{{\boldsymbol{\mathfrak{K}}}}{\to}
KKUT_e^{-1}(A,A)^{++}\to 0.
\eneq

In particular, if $\phi, \psi\in \mathrm{Aut}(A)$ are such that
$$
{\boldsymbol{\mathfrak{K}}}(\phi)={\boldsymbol{\mathfrak{K}}}(\psi)={\boldsymbol{\mathfrak{K}}}({\rm
id}_A),
$$
then
$$
\eta_{{\rm id}_A}(\phi\circ \psi)=\eta_{{\rm
id}_A}(\phi)+\eta_{{\rm id}_A}(\psi).
$$

\end{cor}

\begin{proof}
It follows from Lemma \ref{L86} {{(see also Remark \ref{Remark202011-1})}} that,  for any $\langle \kappa, \af,
\gamma\rangle,$ there is a unital monomorphism $h: A\to A$ such
that ${\boldsymbol{\mathfrak{K}}}(h)=\langle \kappa, \af,
\gamma\rangle.$ The fact that $\kappa\in KK_e^{-1}(A,A)^{++}$
implies that there is $\kappa_1\in KK_e^{-1}(A,A)^{++}$ such that
$$
\kappa\times \kappa_1=\kappa_1\times \kappa=[{\rm id}_A].
$$
Using Lemma \ref{L86}, choose $h_1: A\to A$ such that
$$
{\boldsymbol{\mathfrak{K}}}(h)=\langle \kappa_1, \af^{-1},
\gamma^{-1}\rangle.
$$

It follows from Lemma \ref{MUN2} that $h_1\circ h$
and $h\circ h_1$ are approximately unitarily equivalent. Applying
the standard approximate intertwining argument of G. A. Elliott (Theorem 2.1 of \cite{Ell-RR0}), one
obtains two isomorphisms $\phi$ and $\phi^{-1}$ such that there is
a sequence of unitaries $\{u_n\}$ in $A$ such that
$$
\phi(a)=\lim_{n\to\infty}{\rm Ad}\, u_{2n+1}\circ h(a)\andeqn
\phi^{-1}(a)=\lim_{n\to\infty}{\rm Ad}\, u_{2n}\circ h_1(a)
$$
for all $a\in A.$ Thus, $[\phi]=[h]$ in $KL(A,A)$ and
$\phi^{\ddag}=h^{\ddag}$ and $\phi_T=h_T.$ Then, as in the proof
of \ref{L86}, there is $\psi_0\in {\overline{{\rm Inn}}}(A,A)$ such that
$[\psi_0\circ \phi]=[{\rm id}_A]$ in $KK(A,A)$ as well as
$(\psi_0\circ \phi)^{\ddag}=h^{\ddag}$ and $(\psi_0\circ
\phi)_T=h_T.$  So we have $\psi_0\circ \phi\in \mathrm{Aut}(A,A)$ such that
${\boldsymbol{\mathfrak{K}}}(\psi_0\circ \phi)=\langle \kappa, \af,
\gamma\rangle.$ {This implies that ${\mathfrak{K}}$ is surjective.}

Now let $\lambda\in {\rm Hom}(K_1(C), \overline{\text{Aff}(T(A))})/{\cal
R}_0.$  The proof Theorem \ref{MT2} says that there is $\psi_{00} \in
\overline{{\rm Inn}}(A,A)$ (in place of $\psi$) such that
${\boldsymbol{\mathfrak{K}}}(\psi_{00}\circ {\rm
id}_A)={\boldsymbol{\mathfrak{K}}}({\rm id}_A)$ and
$$
\overline{R}_{{\rm id}_A, \psi_{00}}=\lambda.
$$
Note that $\psi_{00}$ is again an automorphism.
The last part of the lemma then follows from Lemma \ref{L97}.
\end{proof}


\begin{df}[Definition 10.2 of \cite{Linajm} and see also \cite{LnTAMS12}]\label{Dsu}
Let $A$ be a unital \CA\, and $B$ be another \CA. Recall
(\cite{LnTAMS12}) that
$$
H_1(K_0(A), K_1(B))=\{x\in K_1(B): \phi([1_A])=x\, \, \text{for\,some }\,\phi\in
{\rm Hom}(K_0(A), K_1(B))\}.
$$
\end{df}

\begin{prop}[Proposition 12.3 of \cite{Linajm})]\label{Sup}
Let $A$  be a unital separable \CA\, and let $B$ be a unital \CA.
Suppose that $\phi: A\to B$ is a unital \hm\, and $u\in U(B)$ is a
unitary. Suppose that there is a continuous path of unitaries
$\{u(t): t\in [0,\infty)\}\subset B$ such that
\beq\label{sup-1}
u(0)=1_B\tand \lim_{t\to\infty}{\rm Ad}\, u(t)\circ \phi(a)={\rm
Ad}\, u\circ \phi(a)
\eneq
for all $a\in A.$  Then
$$
[u]\in H_1(K_0(A), K_1(B)).
$$

\end{prop}


\begin{lem}\label{fix-unitary}
Let $C=C'\otimes U$ for some $C'=\varinjlim(C_n, \psi_n)$ and a UHF algebra $U$ of infinite type, where each $C_n$ is a direct sum of C*-algebras in $\mathcal C_0$ and $\mathbf H$. Assume that $\psi_n$ is unital and injective. Let $A\in\mathcal B_1$. Let $\phi_1, \phi_2: C\to A$ be two monomorphisms such that there is an increasing sequence of finite subsets $\mathcal F_n\subset C$ with dense union, an increasing sequence of finite subsets $\mathcal P_n\subset K_1(C)$ with union equal to $K_1(C)$, a sequence of positive numbers $(\delta_n)$ with $\sum\delta_n<1$ and a sequence of unitaries $\{u_n\}\subset A$ such that
$$\mathrm{Ad}u_n\circ\phi_1\approx_{\delta_n} \phi_2\quad\textrm{on}\ \mathcal F_n\tand\,\,
\rho_A(\mathrm{bott}_1(\phi_2, u_n^*u_{n+1}))=0\tforal  x\in\mathcal P_n.$$
Suppose that $H_1(K_0(C), K_1(A))=K_1(A).$
Then there exists a sequence of unitaries $v_n\in U_0(A)$ such that
\beq\label{strong-1}
\mathrm{Ad}v_n\circ\phi_1\approx_{\delta_n} \phi_2\quad\textrm{on}\ \mathcal F_n\tand\\
\rho_A(\mathrm{bott}_1(\phi_2, v_n^*v_{n+1}))=0,\quad x\in\mathcal P_n.
\eneq
\end{lem}

\begin{proof}
Let $x_n=[u_n]\in K_1(A)$. Since $H_1(K_0(C), K_1(A))=K_1(A)$, there is a homomorphism $$\kappa_{n, 0}: K_0(C)\to K_1(A)$$
such that $\kappa_{n, 0}([1_C])=-x_n$. Since $C$ satisfies the Universal Coefficient Theorem, there is $\kappa_n\in KL(C\otimes C(\T), A)$ such that
$$
(\kappa_n)|_{{\boldsymbol{\bt}}(K_0(C))}=\kappa_{n, 0}\quad\mathrm{and}\quad (\kappa_n)|_{{\boldsymbol{\bt}}(K_1(C))}=0.
$$
\Wlog, we may assume that
$[1_{C}]\in {\cal P}_n,$  $n=1,2,.....$
For each $\delta_n$, choose a positive number $\eta_n<\delta_n,$
such that
$$\mathrm{Ad}u_{n}\circ\phi_1\approx_{\eta_n}\phi_2\quad \textrm{on}\ \mathcal F_n. $$
By Lemma \ref{Extbot1}, there is a unitary $w_n\in U(A)$ such that
\beq\nonumber
\|[\phi_2(a), w_n]\|<(\delta_n-\eta_n)/2\quad \rforal a\in\mathcal F_n\andeqn
\mathrm{Bott}(\phi_2, w_n)|_{{\cal P}_n}=\kappa_n|_{{\boldsymbol{\bt}}({\cal P}_n)}.
\eneq
Put $v_n=u_nw_n$, $n=1, 2, ...$. Then
$$\mathrm{Ad} v_n\circ\phi_1\approx_{\delta_n} \phi_2\quad\textrm{on}\ \mathcal F_n,\,\,\,
\rho_A(\mathrm{bott}_1(\phi_2, v_n^*v_{n+1}))|_{{\cal P}_n}=0$$
and, since $[1_C]\in {\cal P}_n,$
$$[v_n]=[u_n]-x_n=0,$$
as desired.
\end{proof}

\begin{thm}\label{T105}
Let $B\in {\cal B}_1$ be a unital separable simple \CA\, which satisfies the UCT,
let $A_1\in {\cal B}_1$ be a unital separable
simple \CA, and let $C=B\otimes U_1$ and $A=A_1\otimes U_2,$ where $U_1$ and $U_2$ are unital infinite dimensional UHF-algebras. Suppose that
$H_1(K_0(C),K_1(A))=K_1(A)$ and suppose that $\phi_1, \phi_2: C\to
A$ are two unital monomorphisms which are asymptotically unitarily
equivalent. Then $\phi_1$ and $\phi_2$ are strongly asymptotically unitarily
equivalent, that is,  there exists a continuous path of unitaries
$\{u(t): t\in [0, \infty)\}\subset A$ such that
$$
u(0)=1\tand \lim_{t\to\infty}{\rm Ad}u(t)\circ
\phi_1(a)=\phi_2(a)\ \textrm{for all}\ a\in C.
$$
\end{thm}
\begin{proof}
By 4.3 of \cite{Lnclasn}, one has
\begin{eqnarray*}
[\phi_1]=[\phi_2]\quad {\rm in}\,\,\, KK(C,A),\\
\phi^{\ddag}=\psi^{\ddag},\,\,\,(\phi_1)_T=(\phi_2)_T \andeqn {\overline{R_{\phi_1,\phi_2}}}=0.
\end{eqnarray*}
Then by Lemma \ref{fix-unitary} {{(see also Remark \ref{Remark202011-1})}}, one may assume that $v_n\in U_0(A)$ ($n=1, 2, ...$) in the proof of Theorem \ref{Tm72}. It follows that $\xi_n([1_C])=0$, $n=1, 2, ...$, and therefore $\kappa_n([1_C])=0$. This implies that $\gamma_n\circ\beta([1_C])=0$. Hence $w_n\in U_0(A)$, and also $u_n\in U_0(A)$. Therefore, the continuous path of unitaries $\{u(t)\}$ constructed in Theorem \ref{Tm72} is in $U_0(A)$, and then one may require that $u(0)=1_A$ by connecting $u(0)$ to $1_A$.
\end{proof}

\section{The general classification theorem}

\begin{lem}\label{L112}
Let $A_1\in {\cal B}_0$ be a unital separable simple \CA, let $A=A_1\otimes U$ for some infinite dimensional UHF-algebra, and
let $\mathfrak{p}$ be a supernatural number of infinite type. Then
the \hm\, $\imath: a\mapsto a\otimes 1$ induces an isomorphism
from $U_0(A)/CU(A)$ to $U_0(A\otimes M_{\mathfrak{p}})/CU(A\otimes
M_{\mathfrak{p}}).$
\end{lem}

\begin{proof}
There are  sequences of positive integers $ \{m(n)\}$ and
$\{k(n)\}$ such that $A\otimes
M_{\mathfrak{p}}=\lim_{n\to\infty}(A\otimes M_{m(n)}, \imath_n),$
where
$$
\imath_n: M_{m(n)}(A)\to M_{m(n+1)}(A)
$$
is defined by $\imath(a)={\rm diag}(\overbrace{a,a,...,a}^{k(n)})$
for all $a\in M_{m(n)}(A),$ $n=1,2,....$ Note, $M_{m(n)}(A)=M_{m(n)}(A_1)\otimes U$ and $M_{m(n)}(A_1)\in {\cal B}_0.$ Let $$j_n:
U(M_{m(n)}(A))/CU(M_{m(n)}(A)))\to
U(M_{m(n+1)}(A))/CU(M_{m(n+1)}(A))
$$ be defined by
$$
j_n({\bar u})=\overline{{\rm
diag}(u,\underbrace{1,1,...,1}_{k(n)-1})}\tforal u\in
U((M_{m(n)}(A)).
$$
It follows from Corollary11.11 of \cite{GLN-I} 
that $j_n$ is
an isomorphism.
 By Corollary 11.7 of \cite{GLN-I}, 
the abelian group $U_0(M_{m(n)}(A))/CU(M_{m(n)}(A))$ is divisible. For each $n$ and $i,$ there is a unitary $U_i\in
M_{m(n+1)}(A)$ such that
$$
U_i^*E_{1,1}U_i=E_{i,i},\,\,\,i=2,3,...,k(n),
$$
where $E_{i,i}=\sum_{j=(i-1)m(n)+1}^{im(n)}e_{j,j}$  and
$\{e_{i,j}\}$ is a system of matrix units for $M_{m(n+1)}.$ Then
$$
\imath_n(u)=u'(U_2^*u'U_2)(U_3^*u'U_3)\cdots (U_{k(n)}^*u'U_{k(n)}),
$$
where $u'={\rm diag}(u,\overbrace{1,1,...,1}),$ for all $u\in
M_{m(n)}(A).$ Thus,
$$
\imath_n^{\ddag}({\bar u})=k(n)j_n(\bar u).
$$
It follows that
$\imath_n^{\ddag}|_{U_0(M_{m(n)}(A))/CU(M_{m(n)}(A))}$ is
injective, since $U_0(M_{m(n+1)}(A))/CU(M_{m(n+1)}(A))$ is torsion
free (see Lemma 11.5 of \cite{GLN-I}) 
and $j_n$ is injective. For each $z\in U_0(M_{m(n+1)}(A)/CU(M_{m(n+1)}),$ there is a unitary $v\in
M_{m(n+1)}(A)$ such that
$$
j_n({\bar v})=z,
$$
since $j_n$ is an isomorphism. By the divisibility of
$U_0(M_{m(n)}(A)/CU(M_{m(n)}),$ there is $u\in M_{m(n)}(A)$ such
that
$$
\overline{u^{k(n)}}=\overline{u}^{k(n)}=\overline{v}.
$$
As above,
$$
\imath_n^{\ddag}({\bar u})=k(n)j_n(\bar v)=z.
$$
So $\imath_n^{\ddag}|_{U_0(M_{m(n)}(A))/CU(M_{m(n)}(A))}$ is
surjective.  It follows that $\imath_{n,
\infty}^{\ddag}|_{U_0(M_{m(n)}(A))/CU(M_{m(n)}(A))}$ is an
isomorphism. One then concludes that
$\imath^{\ddag}|_{U_0(A)/CU(A)}$ is an isomorphism.
\end{proof}

\begin{lem}\label{L113}
Let $A_1$ and $B_1$ be two unital separable simple \CA s  in ${\cal B}_0,$ let $A=A_1\otimes U_1$ and let $B=B_1\otimes U_2,$ where $U_1$ and $U_2$ are two  UHF-algebras of infinite type.
Let $\phi: A\to B$ be an isomorphism and let $\bt:
B\otimes M_{\mathfrak{p}}\to B\otimes M_{\mathfrak{p}}$ be an
automorphism such that $\bt_{*1}={\rm id}_{K_1(B\otimes
M_{\mathfrak{p}})}$ for some supernatural number $\mathfrak{p}$ of infinite type.
Then
$$\psi^{\ddag}(U(A)/CU(A))=
(\phi_0)^{\ddag}(U(A)/CU(A))=U(B)/CU(B),
$$
where $\phi_0=\imath\circ \phi,$ $\psi=\bt\circ \imath\circ \phi$
and where $\imath: B\to B\otimes M_{\mathfrak{p}}$ is defined by
$\imath(b)=b\otimes 1$ for all $b\in B.$ Moreover, there is an
isomorphism $\mu: U(B)/CU(B)\to U(B)/CU(B)$ with
$\mu(U_0(B)/CU(B))\subset U_0(B)/CU(B)$ such that
$$
\imath^{\ddag}\circ \mu\circ \phi^{\ddag}=\psi^{\ddag}\andeqn
q_1\circ \mu=q_1,
$$
where $q_1: U(B)/CU(B)\to K_1(B)$ is the quotient map.
\end{lem}

\begin{proof}
The proof is exactly the same as that of Lemma 11.3 of \cite{Lnclasn}.
\end{proof}

\begin{lem}\label{L114}
Let $A_1$ and $B_1$ be two unital   simple amenable \CA s  in
${\cal B}_0$ {{satisfying the UCT,}} let $A=A_1\otimes U_1$, and let
$B=B_1\otimes U_2,$ where $U_1$ and $U_2$ are
UHF-algebras {{of infinite type}}.  Suppose that
$\phi_1, \phi_2: A\to B$ are two isomorphisms such that
$[\phi_1]=[\phi_2]$ in $KK(A,B).$ Then there exists an
automorphism $\bt: B\to B$ such that $[\bt]=[{\rm id}_B]$ in
$KK(B,B)$ and $\bt\circ \phi_2$ is asymptotically unitarily
equivalent to $\phi_1.$ Moreover, if $H_1(K_0(A), K_1(B))=K_1(B),$ then $\beta$ can be chosen so that $\beta\circ\phi_1$ and $\beta\circ\phi_2$ are strongly asymptotically unitarily equivalent.
\end{lem}

\begin{proof}
It follows from Theorem \ref{MT2} that there is an automorphism $\bt_1:
B\to B$ satisfying the following condition:
\beq
[\bt_1]=[{\rm id}_B]\,\,\,{\rm in}\,\,\,KK(B,B),\\
\bt_1^{\ddag}=\phi_1^{\ddag}\circ (\phi_2^{-1})^{\ddag}\andeqn
(\bt_1)_T=(\phi_1)_T\circ (\phi_2)_T^{-1}.
\eneq
By Corollary \ref{C910}, there is automorphism $\bt_2\in \mathrm{Aut}(B)$ such that
\beq
[\bt_2]=[{\rm id}_B]\,\,\, {\rm in}\,\,\, KK(B,B),\\
\bt_2^{\ddag}={\rm id}_B^{\ddag},\,\,\,(\bt_2)_T=({\rm
id}_B)_T, \andeqn\\
 \overline{R}_{{\rm id}_B,\bt_2}=-\overline{R}_{\phi_1,\bt_1\circ \phi_2}\circ
(\phi_2)_{*1}^{-1}.
\eneq
Put $\bt=\bt_2\circ \bt_1.$
 It follows that
\beq
[\bt\circ \phi_2]=[\phi_1]\,\,\,{\rm in}\,\,\, KK(A,B),
(\bt\circ \phi_2)^{\ddag}=\phi_1^{\ddag}, \andeqn (\bt\circ
\phi_2)_T=(\phi_1)_T.
\eneq
Moreover, by \ref{L97},
\beq
\overline{R}_{\phi_1,\bt\circ \phi_2}&=&\overline{R}_{{\rm
id}_B,\bt_2}\circ
(\phi_2)_{*1}+\overline{R}_{\phi_1,\bt_1\circ\phi_2}\\
&=&(-\overline{R}_{\phi_1,\bt_1\circ \phi_2}\circ
(\phi_2)_{*1}^{-1})\circ (\phi_2)_{*1}
+\overline{R}_{\phi_1,\bt_1\circ\phi_2}=0.
\eneq
It follows from \ref{MT2} that $\bt\circ \phi_2$ and $\phi_1$ are
asymptotically unitarily equivalent.

In the case that $H_1(K_0(A), K_1(B))=K_1(B),$ it follows from
Theorem \ref{T105} that $\bt\circ \phi_2$ and $\phi_1$ are strongly
asymptotically unitarily equivalent.
\end{proof}

\begin{lem}\label{L115}
Let $A_1$ and  $B_1$ be two  unital simple amenable \CA s in
${\cal B}_0$ {{satisfying the UCT}}  and let $A=A\otimes U_1$ and
 $B=B_1\otimes U_2$ for  UHF-algebras $U_1$ and $U_2$ of infinite type.  Let $\phi: A\to B$ be an
isomorphism. Suppose that $\bt\in \mathrm{Aut}(B\otimes M_{\mathfrak{p}})$ is such that
$$[\bt]=[{\rm id}_{B\otimes
M_{\mathfrak{p}}}]\,\,\,{\textrm in}\,\,\,KK(B\otimes
M_{\mathfrak{p}},B\otimes M_{\mathfrak{p}}) \tand \bt_T=({\rm
id}_{B\otimes M_{\mathfrak{p}}})_T
$$
for some supernatural number $\mathfrak{p}$ of infinite type.

Then there exists an automorphism $\af\in Aut(B)$ with
$[\af]=[{\rm id}_{B}]$ in $KK(B,B)$ such that $ \imath\circ
\af\circ \phi $ and $\bt\circ \imath\circ \phi$ are asymptotically
unitarily equivalent, where $\imath: B\to B\otimes
M_{\mathfrak{p}}$ is defined by $\imath(b)=b\otimes 1$ for all
$b\in B.$
\end{lem}

\begin{proof}
It follows from Lemma \ref{L113} that there is an isomorphism $\mu:
U(B)/CU(B)\to U(B)/CU(B)$ such that
$$
\imath^{\ddag}\circ \mu\circ \phi^{\ddag}=(\bt\circ \imath\circ
\phi)^{\ddag}.
$$
Note that $\imath_T: T(B\otimes M_{\mathfrak{p}})\to T(B)$ is an
affine homeomorphism.

It follows from Theorem \ref{MT2} that there is an automorphism $\af: B\to
B$ such that
\beq\label{l1-1}
&&[\af]=[{\rm id}_B]\,\,\,{\rm in}\,\,\,KK(B,
B),\\
&&\af^{\ddag}=\mu,\,\,\, \af_T=(\bt\circ \imath\circ \phi)_T\circ
((\imath\circ
\phi)_T)^{-1}=({\rm id}_{B\otimes M_{\mathfrak{p}}})_T\andeqn\\
&&\overline{R}_{{\rm
id}_B,\af}(x)(\tau)=-\overline{R}_{\bt\circ\imath\circ \phi,\,
\imath\circ \phi}(\phi_{*1}^{-1}(x))(\imath_T(\tau))\tforal x\in
K_1(A)
\eneq
and for all $\tau\in T(B).$

Denote by $\psi=\imath\circ \af\circ \phi.$ Then we have, by Lemma
\ref{L97},
\beq
&&[\psi]= [\imath\circ \phi]=[\bt\circ\imath\circ \phi]\,\,\,{\rm
in}\,\,\, KK(A, B\otimes M_{\mathfrak{p}})\\
&&\psi^{\ddag}=\imath^{\ddag}\circ\mu\circ \phi^{\ddag}=(\bt\circ
\imath\circ \phi)^{\ddag}, \, \mathrm{and}\\
&&\psi_T=(\imath\circ \af\circ \phi)_T=(\imath\circ
\phi)_T=(\bt\circ \imath\circ \phi)_T.
\eneq
Moreover, for any $x\in K_1(A)$ and $\tau\in T(B\otimes
M_{\mathfrak{p}}),$
\beq
\overline{R}_{\bt\circ \imath\circ \phi, \psi}(x)(\tau)&=&
\overline{R}_{\bt\circ \imath\circ\phi,
\imath\circ\phi}(x)(\tau)+\overline{R}_{\imath, \imath\circ \af}\circ
\phi_{*1}(x)(\tau)\\
&=&\overline{R}_{\bt\circ \imath\circ\phi,
\imath\circ\phi}(x)(\tau)+\overline{R}_{{\rm id}_B, \imath\circ \af}\circ
\phi_{*1}(x)(\imath_T^{-1}(\tau))\\
&=&\overline{R}_{\bt\circ \imath\circ\phi,
\imath\circ\phi}(x)(\tau)-\overline{R}_{\bt\circ \imath\circ\phi,
\imath\circ\phi}(\phi_{*1}^{-1})(\phi_{*1}(x))(\tau)=0.
\eneq
It follows from Theorem \ref{Tm72}  that $\imath\circ \af\circ \phi$ and
$\bt\circ \imath\circ \phi$ are asymptotically unitarily
equivalent.
\end{proof}

{{Let ${\cal N}$ be the class of separable amenable \CA s which satisfy the UCT.}}


%


  \begin{thm}\label{CMT1}
Let $A$ and $B$ be two unital separable  simple
\CA s in ${\cal N}.$  Suppose that there is an isomorphism
$$
\Gamma: {\rm Ell}(A)\to {\rm Ell}(B).
$$
Suppose also that, for some pair of relatively prime supernatural
numbers $\mathfrak{p}$ and $\mathfrak{q}$ of infinite type such
that $M_{\mathfrak{p}}\otimes M_{\mathfrak{q}}\cong Q,$ we have
$A\otimes M_{\mathfrak{p}}\in {\cal B}_0$, $B\otimes
M_{\mathfrak{p}}\in {\cal B}_0$, $A\otimes M_{\mathfrak{q}}\in {\cal B}_0$, and
$B\otimes M_{\mathfrak{q}}\in {\cal B}_0.$ Then,
$$
A\otimes {\cal Z}\cong B\otimes {\cal Z}.
$$
\end{thm}

\begin{proof}
The proof is almost identical to that of 11.7 of \cite{Lnclasn}, with a few necessary modifications.
Note that $\Gamma$ induces an isomorphism
$$
\Gamma_{\mathfrak{p}}: {\rm Ell}(A\otimes M_{\mathfrak{p}})\to
{\rm Ell}(B\otimes M_{\mathfrak{p}}).
$$
 Since $A\otimes M_{\mathfrak{p}} \in {\cal B}_0$ and
$B\otimes M_{\mathfrak{p}}\in {\cal B}_0$, by Theorem
21.10 of \cite{GLN-I},
there is an isomorphism $\phi_{\mathfrak{p}}:
A\otimes M_{\mathfrak{p}}\to B\otimes M_{\mathfrak{p}}.$ Moreover
(by the proof of Theorem 21.10 of \cite{GLN-I}), 
$\phi_{\mathfrak{p}}$ carries $\Gamma_{\mathfrak{p}}.$  In the same way, $\Gamma$ induces an isomorphism
$$
\Gamma_{\mathfrak{q}}:{\rm Ell}(A\otimes M_{\mathfrak{q}})\to
{\rm Ell}(B\otimes M_{\mathfrak{q}})
$$
and there is an isomorphism $\psi_{\mathfrak{q}}: A\otimes
M_{\mathfrak{q}}\to B\otimes M_{\mathfrak{q}}$ which induces
$\Gamma_{\mathfrak{q}}.$

 Put $\phi=\phi_{\mathfrak{p}}\otimes {\rm
id}_{M_{\mathfrak{q}}}: A\otimes Q\to B\otimes Q$ and
$\psi=\psi_{\mathfrak{q}}\otimes {\rm
id}_{{M_{\mathfrak{p}}}}: A\otimes Q\to B\otimes Q.$
Note that
$$
(\phi)_{*i}=(\psi)_{*i}\,\,{\rm (} i=0,1 {\rm )} \andeqn
\phi_T=\psi_T
$$
(all four of these maps are induced by $\Gamma$). Note that $\phi_T$ and $\psi_T$
are affine homeomorphisms. Since $K_{*i}(B\otimes Q)$ is
divisible, we in fact have $[\phi]=[\psi]$ (in $KK(A\otimes Q,
B\otimes Q)$). It follows from Lemma \ref{L114} that there is an
automorphism $\bt: B\otimes Q\to B\otimes Q$ such that
$$
[\bt]=[{\rm id}_{B\otimes Q}]\,\,\, \mathrm{in}\  KK(B\otimes Q, B\otimes Q)
$$
{and} such that $\phi$ and $\bt\circ \psi$ are asymptotically unitarily
equivalent.  Since $K_1(B\otimes Q)$ is divisible,
$H_1(K_0(A\otimes Q), K_1(B\otimes Q))=K_1(B\otimes Q).$ It
follows that $\phi$ and $\bt\circ \psi$ are strongly
asymptotically unitarily equivalent. Note also in this case
$$
\bt_T=({\rm id}_{B\otimes Q})_T.
$$
Let $\imath: B\otimes M_{\mathfrak{q}}\to B\otimes Q$ be defined by
$\imath(b)=b\otimes 1$ for $b\in B.$ We consider the pair
$\bt\circ \imath\circ {\psi}_{\mathfrak{q}}$ and $\imath \circ
{\psi}_{\mathfrak{q}}.$ Applying Lemma \ref{L115}, we obtain an
automorphism $\af: B\otimes M_{\mathfrak{q}}\to B\otimes
M_{\mathfrak{q}}$ such that $\imath\circ \af\circ
\psi_{\mathfrak{q}}$ and $\bt\circ \imath\circ
\psi_{\mathfrak{q}}$ are asymptotically unitarily equivalent (in
$B\otimes Q$). So, by  Lemma \ref{L114}, they are strongly asymptotically unitarily
equivalent in $B\otimes Q.$
Moreover,
$$
[\af]=[{\rm id}_{B\otimes M_{{\mathfrak{q}}}}]\,\,\,{\rm in}\,\,\,
KK(B\otimes M_{\mathfrak{q}},B\otimes M_{\mathfrak{q}}).
$$

We will show that $\bt\circ \psi$ and $(\af\circ{\psi}_{\mathfrak{q}})\otimes {\rm id}_{M_{\mathfrak{p}}}$
 are
strongly asymptotically unitarily equivalent. Define
$\bt_1=(\bt\circ \imath\circ\psi_{\mathfrak{q}})
\otimes {\rm
id}_{M_{\mathfrak{p}}}: B\otimes Q\otimes M_{\mathfrak{p}}\to
B\otimes Q\otimes M_{\mathfrak{p}}.$ Let $j: Q\to Q\otimes
M_{\mathfrak{p}}$ be defined by $j(b)=b\otimes 1.$ There is an
isomorphism $s: M_{\mathfrak{p}}\to M_{\mathfrak{p}}\otimes
M_{\mathfrak{p}}$ such that the homomorphism ${\rm id}_{M_{\mathfrak{q}}}\otimes s: M_{\mathfrak{q}}\otimes M_{\mathfrak{p}} (=Q)\to M_{\mathfrak{q}}\otimes M_{\mathfrak{p}}\otimes M_{\mathfrak{p}}(=Q\otimes M_{\mathfrak{p}})$ induces
$({\rm id}_{M_{\mathfrak{q}}}\otimes
s)_{*0}=j_{*0}.$
In this case, $[{\rm id}_{M_{\mathfrak{q}}}\otimes s]=[j].$ Since
$K_1(M_{\mathfrak{p}})=0{,}$ {by} Theorem \ref{Tm72}, ${\rm
id}_{M_{\mathfrak{q}}}\otimes s$ is strongly asymptotically
unitarily equivalent to $j.$ It follows that $
(\af\circ\psi_{\mathfrak{q}})
\otimes {\rm id}_{M_{\mathfrak{p}}}$ and
$
(\bt\circ \imath\circ \psi_{\mathfrak{q}}) \otimes {\rm
id}_{M_{\mathfrak{p}}}$ are strongly asymptotically unitarily
equivalent (note that $\imath\circ \af\circ
\psi_{\mathfrak{q}}$ and $\bt\circ \imath\circ
\psi_{\mathfrak{q}}$  are
strongly asymptotically unitarily equivalent).
Consider the \SCA\, $C=\bt\circ \psi(1\otimes
M_{\mathfrak{p}})\otimes M_{\mathfrak{p}}\subset B\otimes Q\otimes
M_{\mathfrak{p}}.$ In $C,$  $\bt\circ \phi|_{1\otimes
M_{\mathfrak{p}}}$ and $j_0$ are strongly asymptotically unitarily
equivalent, where $j_0: M_{\mathfrak{p}}\to C$ is defined by $j_0(a)=1\otimes
a$ for all $a\in M_{\mathfrak{p}}.$ In particular, there exists a continuous path
of unitaries $\{v(t): t\in [0,\infty)\}\subset C$ such that
\beq\label{CM1-1}
\lim_{t\to\infty}{\rm Ad}\, v(t) \circ\bt\circ \phi(1\otimes
a)=1\otimes a\tforal a\in M_{\mathfrak{p}}.
\eneq
It follows that $\bt\circ \psi$ and $\bt_1$ are strongly
asymptotically unitarily equivalent. Therefore $\bt\circ \psi$ and
$
(\af\circ \psi_{\mathfrak{q}})\otimes {\rm id}_{M_{\mathfrak{p}}}$
are strongly asymptotically unitarily equivalent. Finally, we
conclude that $
(\af\circ \psi_{\mathfrak{q}})\otimes {\rm
id}_{M_{\mathfrak{p}}}$ and $\phi$ are strongly asymptotically
unitarily equivalent. Note that $\af\circ \psi_{\mathfrak{q}}$ is
an isomorphism which induces $\Gamma_{\mathfrak{q}}.$

Let $\{u(t): t\in [0,1)\}$ be a continuous path of unitaries in
$B\otimes Q$ with $u(0)=1_{B\otimes Q}$ such that
$$
\lim_{t\to\infty}{\rm Ad}\, u(t)\circ \phi(a)=\af\circ
\psi_{\mathfrak{q}}\otimes {\rm id}_{{M_{\mathfrak{p}}}}(a)\tforal
a\in A\otimes Q.
$$
One then obtains a unitary suspended isomorphism which lifts
$\Gamma$ along $Z_{p,q}$ (see \cite{Winter-Z}). It follows from Theorem
7.1 of \cite{Winter-Z} that $A\otimes {\cal Z}$ and $B\otimes {\cal Z}$
are isomorphic.
\end{proof}

\begin{df}\label{Class}
Denote by ${\cal N}_0$ the class of those unital  simple
\CA s $A$  in ${\cal N}$ for which $A\otimes
M_{\mathfrak{p}}\in {\cal N}\cap {\cal B}_0$ for any supernatural number
${\mathfrak{p}}$ of infinite type.

Of course ${\cal N}_0$  contains all unital simple amenable \CA s
in ${\cal B}_0$  which satisfy the UCT.
It contains all unital simple inductive limits of \CA s in ${\cal C}_0.$
It should be noted that, by Theorem 19.3
 of \cite{GLN-I}, ${\cal N}_0={\cal N}_1.$

\end{df}

\begin{cor}\label{CM1}
Let $A$ and $B$ be two \CA s in ${\cal N}_0.$ Then $A\otimes {\cal
Z}\cong B\otimes {\cal Z}$ if and only if ${\rm Ell}(A\otimes {\cal
Z})\cong {\rm Ell}(B\otimes {\cal Z}).$
\end{cor}

\begin{proof}
This follows from Theorem \ref{CMT1} immediately.
\end{proof}





\begin{thm}\label{MFTh}
Let $A$ and $B$ be two unital separable simple amenable ${\cal Z}$-stable \CA s which satisfy the UCT.
Suppose that $gTR(A\otimes Q)\le 1$ and $gTR(B\otimes Q)\le 1.$ Then
$A\cong B$ if and only if
\begin{equation*}
{\rm Ell}(A)\cong {\rm Ell}(B).
\end{equation*}
\end{thm}

\begin{proof}
It follows from Corollary 19.3 of \cite{GLN-I}
that $A\otimes U, B\otimes U\in {\cal B}_0$ for any UHF-algebra $U$ of infinite type.
The theorem follows immediately by Corollary \ref{CM1}.
\end{proof}

\begin{cor}\label{MFTC}
Let $A$ and $B$ be two unital separable amenable simple \CA s which satisfy the UCT.
Suppose that $gTR(A)\le 1$ and $gTR(B)\le 1.$ Then
$A\cong B$ if and only if
\begin{equation*}
{\rm Ell}(A)\cong {\rm Ell}(B).
\end{equation*}
\end{cor}

\begin{cor}\label{CM2}
Let $A$ and $B$ be two unital simple \CA s in ${\cal B}_1 \cap {\cal N}.$
Then $A\cong B$ if and only if
$$
{\rm Ell}(A)\cong {\rm Ell}(B).
$$
\end{cor}

\begin{proof}
It follows from Theorem
10.7 of \cite{GLN-I}
that $A\otimes {\cal Z}\cong A$ and
$B\otimes {\cal Z}\cong B.$ The corollary then follows from Theorem \ref{MFTh}.
\end{proof}

\begin{rem}
Soon after this work was completed in 2015, it was shown (see \cite{EGLN})
\CA s $A$ with finite decomposition rank  which satisfy the UCT have $gTR(A\otimes U)\le 1$ for all
UHF-algebras of infinite type.  Therefore, by Theorem \ref{MFTh}, they are classified by the Elliott invariant.
\end{rem}

{\small
\bibliographystyle{amsplain}

\begin{thebibliography}{10}


\bibitem{AS} C. Akemann and F.  Shultz, {\em  Perfect
$C^\ast$-algebras},  Mem. Amer. Math. Soc.  55  (1985),  no. 326,
xiii+117 pp.



 \bibitem{BDR} B. Blackadar,  M.~D{\u a}d{\u a}rlat, and  M.~R{\o}rdam,
\emph{The real rank of inductive limit $C^*$-algebras.}
Math. Scand. {\bf 69} (1991),  {211--216} (1992).


\bibitem{Blatrace}
B.~Blackadar and M.~R{\o}rdam, \emph{Extending states on preordered semigroups
  and existence of the quasitrace on $\mbox{C}$*-algebras}, J. Algebra
  \textbf{152} (1992), 240--247.








\bibitem{BO-Book} N. Brown and N.~Ozawa, {\em C*-algebras and finite-dimensional approximations},   American Mathematical Society, Inc., Providence, RI, 2008. xv+509 pp. ISBN: 0-8218-4381.










\bibitem{HS} de la Harpe, P. and Skandalis, G., D\'{e}terminant associ\'{e} \`{a} une trace sur une alg\`{e}bre de Banach, Ann. Inst. Fourier (Grenoble), 34 (1984), 169--202.



\bibitem{DL}
M.~D{\u a}d{\u a}rlat and T.~Loring, \emph{A universal multicoefficient theorem
  for the $\textrm{Kasparov}$ groups}, Duke Math. J. \textbf{84} (1996),
  355--377.

\bibitem{DE} M.~D{\u a}d{\u a}rlat and S.  Eilers, {\em On the classification of nuclear {$C^*$}-algebras},
 Proc. London Math. Soc. {\bf 85} (2002),  168--210.

\bibitem{DNNP-AH}
  M.~D{\u a}d{\u a}rlat, G.~Nagy, A.~N\'{e}methi, and C.~Pasnicu,
  \emph{Reduction of topological stable rank in inductive limits of $\textrm{C*}$-algebras},
  Pacific J. Math. \textbf{153} (1992), 267--276.




\bibitem{ELP1} S.  Eilers, T. Loring, and G. K. Pedersen, {\em Stability of anticommutation relations: an application of noncommutative CW complexes}, J. Reine Angew. Math. {\bf 499} (1998), 101--143.


\bibitem{Ell-76} G. A.  Elliott, {\em On the classification of inductive limits of sequences of semisimple finite-dimensional algebras},  J. Algebra,  {\bf 38} (1976),  29--44.

\bibitem{Ellc1} G.  A. Elliott, {\em The classification problem for amenable $C^*$-algebras},  Proceedings of the International Congress of Mathematicians, Vol. 1, 2 (Zurich, 1994), 922--932, Birkhauser, Basel, 1995.

 \bibitem{Ell-RR0} G. A.  Elliott, {\em On the classification of {C*}-algebras of real rank zero}, J. Reine Angew. Math. {\bf 443} (1993), 179--219.

\bibitem{point-line} G.~A. Elliott. \newblock An invariant for simple $\mbox{C}$*-algebras. \newblock {\em Canadian Mathematical Society. 1945--1995},  Vol. 3, {{61--90}}, Canadian Math. Soc., Ottawa, ON, 1996.

\bibitem{EG-RR0AH}
G.~A. Elliott and G.~Gong, \emph{On the classification of {C*}-algebras of real
  rank zero. {II}}, Ann. of Math. (2) \textbf{144} (1996),  497--610.

\bibitem{EGL-AH}
G.~A. Elliott, G.~Gong, and L.~Li, \emph{On the classification of simple
  inductive limit {C*}-algebras. {II}, {T}he isomorphism theorem}, Invent.
  Math. \textbf{168} (2007),  249--320.

  \bibitem{EGLN}
  G. ~A. Elliott, G.~Gong,  H. ~Lin and  Z. Niu, {\em On the classification of  simple   \CA s with finite
decomposition rank, II},
 preprint.
   arXiv:1507.03437







\bibitem{ET-PL}
G.~A. Elliott and K.~Thomsen, \emph{The state space of the {K$_0$}-group of a simple separable {C*}-algebra}, Geom. Funct. Anal.
  \textbf{4} (1994), 522--538.
\bibitem{Elltoms} G. A.  Elliott and A. S. Toms, {\em Regularity properties in the classification program for separable amenable $C^*$-algebras},  Bull. Amer. Math. Soc. (N.S.) {\bf 45} (2008),  229--245.





\bibitem{Gong-AH}
G.~Gong, \emph{On the classification of simple inductive limit
  $\mbox{C}^*$-algebras. $\mbox{I}$, $\textrm{The}$ reduction theorem}, Doc.
  Math. \textbf{7} (2002), 255--461.

 \bibitem{GJS} G. Gong,  X. Jiang, and H.  Su,  \emph{Obstructions to ${\mathcal Z}$-stability for unital simple $\mbox{C}^*$-algebras.} Canad. Math. Bull. 43 (2000), {418--426.}


  \bibitem{GLN-I}
 G.~Gong, H.~Lin, and Z.~Niu, \emph{A classification of finite simple amenable $\mathcal Z$-stable C*-algebras, I, C*-algebras with generalized tracial rank one}, C. R. Math. Acad. Sci. Soc. R. Canada (2020), to appear.



\bibitem{GLX}
G.~Gong, H.~Lin, and Y.~Xue, \emph{Determinant rank of {C*}-algebras}, {{Pacific J. Math. {\bf 274} (2015), 405-436.
 }}

 \bibitem{Goodearl-AH} K.~R. Goodearl, \emph{Note on a class of simple C*-algebras with real rank zero}, Publications Matematiques \textbf{36} (1992), 637--654.







 \bibitem{HL} J. Hua and H. Lin, {\em Rotation algebras and Exel trace formula},  {Canad. J. Math. {\bf  67} (2015), 404--423. }


\bibitem{JS}  X. Jiang and H. Su, {\em  On a simple unital projectionless $C^*$-algebra}, Amer. J. Math. {\bf 121} (1999),  359--413.



\bibitem{Kirch-Infty}
E.~Kirchberg, \emph{The classification of purely infinite $\mbox{C}$*-algebras
  using $\mbox{Kasparov's}$ theory}, Preprint (1994).

\bibitem{KP0}
E.~Kirchberg and N.~C. Phillips, \emph{Embedding of exact {C*}-algebras in the
  {C}untz algebra {$\mathcal O_2$}}, J. Reine Angew. Math. \textbf{525} (2000),
  17--53.


{\bibitem{Li-K-theory}
L.~Li, \emph{$C^*$-algebra homomorphisms and KK-theory}, K-theory {\bf 18} (1999), 167--172.}








\bibitem{LnTAF}
H.~Lin, \emph{Tracially $\mbox{AF C*}$-algebras.},
  Trans. Amer. Math. Soc. \textbf{353} (2001),  693--722.


\bibitem{LinTAF2}
H.~Lin, \emph{Classification of simple tracially $\mbox{AF C*}$-algebras.},
  Canad. J. Math. \textbf{53} (2001), 161--194.
















 \bibitem{Lnbok} H. Lin, {\em An introduction to the classification of amenable $C^*$-algebras.}, World Scientific Publishing Co., Inc., River Edge, NJ, 2001. xii+320 pp. ISBN: 981-02-4680-3.

\bibitem{Lnjotuni} H. Lin, {\em  Stable approximate unitary equivalence of homomorphisms},  J. Operator Theory  {\bf 47} (2002),  343--378.




\bibitem{Lnduke}
H.~Lin, \emph{Classification of simple {C*}-algebras of tracial topological
  rank zero}, Duke Math. J. \textbf{125} (2004), 91--119.





\bibitem{LnAUCT} H. Lin, {\em An approximate universal coefficient theorem}, Trans. Amer. Math. Soc. {\bf 357} (2005), 3375--3405.


\bibitem{LinTAI}
H.~Lin, \emph{Simple nuclear {C*}-algebras of tracial topological rank one},
  J. Funct. Anal. \textbf{251} (2007),  601--679.



 \bibitem{Lnind} H.  Lin, {\em AF-embedding of crossed products of AH-algebras by $\mathbb Z$ and asymptotic AF-embedding},  Indiana Univ. Math. J. {\bf 57} (2008), 891--944.

\bibitem{Linajm} H. Lin, {\em Asymptotically unitary equivalence and asymptotically inner automorphisms},  Amer. J. Math. {\bf 131} (2009), 1589--1677.


\bibitem{LnHomtp}
H.~Lin, \emph{Approximate homotopy of homomorphisms from {$\textrm{C}(X)$}
  into a simple {C*}-algebra}, Mem. Amer. Math. Soc. \textbf{205} (2010),
  no.~963, vi+131.


 \bibitem{Lnjfa2010} H. Lin, {\em Inductive limits of subhomogeneous $C^*$-algebras with Hausdorff spectrum},  J. Funct. Anal. {\bf 258} (2010), 1909--1932.


\bibitem{Lnclasn}
H.~Lin, \emph{Asymptotic unitary equivalence and classification of simple
  amenable {C*}-algebras}, Invent. Math. \textbf{183} (2011),  385--450.


\bibitem{LnTAMS12} H. Lin, {\em Approximate unitary equivalence in simple $C^*$-algebras of tracial rank one},  Trans. Amer. Math. Soc. {\bf 364} (2012),  2021--2086.






\bibitem{Lin-App}
H.~Lin, \emph{Localizing the $\textrm{Elliott}$ conjecture at strongly
  self-absorbing $\textrm{C*}$-algebras, $\textrm{II}$}, J. Reine Angew. Math. {\bf 692} (2014), 233--243.


\bibitem{Lin-LAH}
H.~Lin, \emph{On local {AH} algebras}, Mem. Amer. Math. Soc. {\bf 235}, no. 1107, ix+109.









\bibitem{L-N}
H.~Lin and Z.~Niu, \emph{Lifting {KK}-elements, asymptotic unitary equivalence
  and classification of simple {C*}-algebras}, Adv. Math. \textbf{219} (2008),
  1729--1769.


\bibitem{LNjfa} H. Lin and Z. Niu, {\em The range of a class of classifiable separable simple amenable $C^*$-algebras}, J. Funct. Anal. {\bf 260} (2011), 1--29.









\bibitem{Niu}
Z.~Niu, \emph{On the classification of $\mbox{TAI}$ algebras}, C. R. Math.
  Acad. Sci. Soc. R. Can. \textbf{26} (2004), 18--24.






\bibitem{Ph1}
N.~C. Phillips, \emph{A classification theorem for nuclear purely infinite
  simple {C*}-algebras}, Doc. Math. \textbf{5} (2000), 49--114 (electronic).


 \bibitem{Rf} M.~Rieffel, {\em The homotopy groups of the unitary groups of noncommutative tori},   J. Operator Theory  {\bf 17}  (1987),   237--254.






\bibitem{Ror-KL-I} M.~R{\o}rdam, {\em Classification of certain infinite simple {C*}-algebras},  J. Funct. Anal.  {\bf 131}  (1995),   415--458.

\bibitem{Ror-infproj}
M.~R{\o}rdam, \emph{A simple {C*}-algebra with a finite and an infinite projection},
  Acta Math. \textbf{191} (2003), 109--142.



\bibitem{TWW}
A~Tikuisis, S.~White, and W.~Winter,
\emph{Quasidiagonality of nuclear {C*}-algebras},
Ann. of Math. (2) {\bf 185} (2017),  229--284.


\bibitem{Thomsen-rims}
K.~Thomsen, \emph{Traces, unitary characters and crossed products by {${\mathbb
  Z}$}}, Publ. Res. Inst. Math. Sci. \textbf{31} (1995), 1011--1029.


\bibitem{Toms-Ann}
A.~S. Toms, \emph{On the classification problem for nuclear {C*}-algebras},
  Ann. of Math. (2) \textbf{167} (2008),  1029--1044.


\bibitem{Vill-perf}
J.~Villadsen, \emph{Simple {C*}-algebras with perforation}, J. Funct. Anal.
  \textbf{154} (1998), 110--116.

\bibitem{Vill-sr}
J.~Villadsen, \emph{On the stable rank of simple {C*}-algebras}, J. Amer. Math. Soc.
  \textbf{12} (1999), 1091--1102.




\bibitem{Winter-Z}
W.~Winter, \emph{Localizing the $\textrm{Elliott}$ conjecture at strongly
  self-absorbing $\textrm{C*}$-algebras}, J. Reine Angew. Math. {\bf 692} (2014), 193--231.


\end{thebibliography}

\providecommand{\bysame}{\leavevmode\hbox to3em{\hrulefill}\thinspace}
\providecommand{\MR}{\relax\ifhmode\unskip\space\fi MR }
\providecommand{\MRhref}[2]{%
  \href{https://urldefense.com/v3/__http://www.ams.org/mathscinet-getitem?mr=*1*7D*7B*2__;IyUlIw!5W9E9PnL_ac!Uc_V1SBthxM_rOFn2LVg5SG0uHr0VdNB3FZg1tT6CRB9vOCgEx3u3zrx_Z7VjNs$ }
}
\providecommand{\href}[2]{#2}

}

\noindent
G. Gong:  Department of Mathematics, Hebei Normal University, Shijiazhuang, Hebei 050016, China and
     Department of mathematics, University of Puerto Rico, Rio Piedras, PR 00936, USA.
ghgong@gmail.com,

\noindent
H. Lin:  Department of Mathematic, East China Normal University, Shanghai 200062, China\\
   and Current\\
    Department of Mathematics, University of Oregon,
    Eugene, Oregon, 97402, USA. \\
hlin@uoregon.edu,

\noindent
Z. Niu: Department of Mathematics, University of Wyoming, Laramie, WY 82071,  USA.\\
zniu@uwyo.edu.


\end{document}